\documentclass[11pt, twoside]{amsart}
\usepackage{amsmath,mathtools}
\usepackage{cancel}
\usepackage{amsmath, amsthm, amssymb, amsfonts, enumerate}
\usepackage[colorlinks=true,linkcolor=blue,citecolor=blue]{hyperref}
\usepackage{autonum}
\usepackage{dsfont}
\usepackage{color}
\usepackage{geometry}
\usepackage{todonotes}
\usepackage{epstopdf}
\usepackage{bbm}
\usepackage{geometry}
\usepackage[utf8]{inputenc}
\usepackage{bm}
\usepackage{amsfonts}
\usepackage{amsfonts}
\usepackage{textcomp}
\usepackage{amssymb}
\usepackage{float}
\usepackage{tikz}
\usepackage{epsfig}
\usepackage{amsmath}
\usepackage[english]{babel}
\usepackage{a4}
\usepackage{enumerate,enumitem}
\usepackage{tcolorbox}
\usepackage{bm} 
\usepackage{mathrsfs}
\usepackage{soul}
\newcommand{\stkout}[1]{\ifmmode\text{\sout{\ensuremath{#1}}}\else\sout{#1}\fi}
\usepackage{ulem}
\usepackage{cancel}

\makeatletter
\newsavebox{\@brx}
\newcommand{\llangle}[1][]{\savebox{\@brx}{\(\m@th{#1\langle}\)}%
  \mathopen{\copy\@brx\mkern2mu\kern-0.9\wd\@brx\usebox{\@brx}}}
\newcommand{\rrangle}[1][]{\savebox{\@brx}{\(\m@th{#1\rangle}\)}%
  \mathclose{\copy\@brx\mkern2mu\kern-0.9\wd\@brx\usebox{\@brx}}}
\makeatother

\newtheorem{theorem}{Theorem}[section]

\newtheorem{remark}[theorem]{Remark}

\newtheorem{lemma}[theorem]{Lemma}
\newtheorem{prop}[theorem]{Proposition}
\newtheorem{corollary}[theorem]{Corollary}
\newtheorem{definition}[theorem]{Definition}

\newtheorem{assumption}[theorem]{Assumption}
\newtheorem{notation}[theorem]{Notation}

\DeclareMathOperator*{\argmin}{argmin}

\definecolor{red}{rgb}{1.0,0.0,0.0}

\definecolor{blu}{rgb}{0.0,0.0,1.0}

\definecolor{gre}{rgb}{0.03,0.50,0.03}

\definecolor{darkviolet}{rgb}{0.58, 0.0, 0.83}

\def \eps{\varepsilon}

\usepackage[T1]{fontenc}

\title[Stochastic Optimal Control of Interacting Particle Systems in Hilbert Spaces]{Stochastic Optimal Control of Interacting Particle Systems in Hilbert Spaces and Applications}
\author[F. de Feo]{Filippo de Feo}
\address{Institut f\"{u}r Mathematik, Technische Universit\"{a}t Berlin, Berlin, Germany}
\email{\href{defeo@math.tu-berlin.de}{defeo@math.tu-berlin.de}}

\author[F. Gozzi]{Fausto Gozzi}
\address{Dipartimento di AI, Data and Decision Sciences, LUISS University, Roma, Italy}
\email{\href{fgozzi@luiss.it}{fgozzi@luiss.it}}

\author[A. \'{S}wi\k{e}ch]{Andrzej \'{S}wi\k{e}ch}
\address{School of Mathematics, Georgia Institute of Technology, Atlanta, GA 30332, USA}
\email{\href{swiech@math.gatech.edu}{swiech@math.gatech.edu}}

\author[L. Wessels]{Lukas Wessels}
\address{School of Mathematics, Georgia Institute of Technology, Atlanta, GA 30332, USA}
\curraddr{Laboratoire J.A. Dieudonn\'e, Universit\'e C\^ote d'Azur, 06108 Nice, France}
\email{\href{lukas.wessels@univ-cotedazur.fr}{lukas.wessels@univ-cotedazur.fr}}

\numberwithin{equation}{section}

\begin{document}

\begin{abstract}
Optimal control of interacting particles governed by stochastic evolution equations in Hilbert spaces is an open area of research. Such systems naturally arise  in formulations where each particle is modeled by  stochastic partial differential equations, path-dependent stochastic differential equations (such as stochastic delay differential equations or stochastic Volterra integral equations), or  partially observed stochastic systems. The purpose of this manuscript is to build  a limiting theory as the number of particles tends to infinity via the dynamic programming approach. We prove the convergence of the value functions $u_n$ of finite particle systems to a function $\mathcal{V}$, which is the unique  $L$-viscosity solution of the corresponding mean-field Hamilton-Jacobi-Bellman equation in the space of probability measures, and we  identify its lift with the value function $U$ of the so-called ``lifted'' limit optimal control problem. Under suitable additional assumptions, we show $C^{1,1}$-regularity of $U$, we prove that $\mathcal{V}$ projects precisely onto the value functions $u_n$, and that optimal (resp. optimal feedback) controls of the particle system correspond to optimal (resp. optimal feedback) controls of the lifted control problem started at the corresponding initial condition. To the best of our knowledge, these are the first results of this kind for stochastic optimal control problems for  interacting particle systems of stochastic evolution equations in Hilbert spaces. We apply the developed theory to problems arising in economics where the  particles  are modeled by stochastic delay differential equations and stochastic partial differential equations.
\end{abstract}
\maketitle


\vspace{-0.55cm}
\vspace{0.3cm}
\begin{small}
\tableofcontents
\end{small}
\section{Introduction}
In this paper, we begin the investigation {of the dynamic programming approach} for stochastic optimal control problems for large interacting particle systems of stochastic evolution equations in Hilbert spaces. 
{More precisely}, in a real, separable Hilbert space $H$, for $n\geq 1$ we consider
$\mathbf{x}(s) = (x_1(s),\dots,x_n(s)) \in H^n$ to be the solution of the system of $H$--valued stochastic differential equations (SDEs)
\begin{equation}\label{state_equation}
	\begin{cases}
		\mathrm{d}x_i(s) = [A x_i(s) +f(x_i(s),\mu_{{\mathbf x}(s)},a_i(s))]\mathrm{d}s + \sigma(x_i(s),\mu_{{\mathbf x}(s)})\mathrm{d}W(s), \quad s \in [t,T]\\
x_i(t) = x_i \in H,
	\end{cases}
\end{equation}
$i=1,\dots,n$, where $A:D(A)\subset H \to H$ is a linear densely defined maximal dissipative operator, $\mu_{\mathbf{x}(s)}:=\frac 1 n \sum_{i=1}^n \delta_{x_i(s)}$ is the empirical measure of $\mathbf{x}(s)$, and $W(s)$ is a cylindrical Wiener process in some Hilbert space $\Xi$ acting as a common noise. The precise conditions on the admissible controls $a_i(\cdot)$ and the coefficient functions will be given in Sections  \ref{sec:control}, \ref{sec:assumptions}. The cost functional is of the form
\begin{equation}\label{cost_functional}
	J_n(t,{\bf x};\mathbf{a}(\cdot))= \mathbb{E}\left [ \int_t^T\frac{1}{n}\sum_{i=1}^n l(x_i(s),\mu_{{\bf x}(s)},a_i(s)) \mathrm{d}s+ \frac{1}{n} \sum_{i=1}^n {\mathcal U}_T(x_i(T),\mu_{{\bf x}(T)}) \right ]
\end{equation}
for $(t,\mathbf{x}) = (t,x_1,\dots,x_n) \in [0,T] \times H^n$.


The study of these types of problems is motivated by control problems where the particles are  modeled by infinite-dimensional dynamics belonging to some of the most prominent and wide-ranging families of stochastic systems, including:\begin{itemize}  \item stochastic partial differential equations (SPDEs),  \item path-dependent stochastic  differential equations (e.g. stochastic delay differential equations (SDDEs) or stochastic Volterra integral equations (SVIEs)), \item partially observed stochastic systems, leading to control of Zakai-type SPDEs.\end{itemize}
These problems are natural generalizations of classical optimal control problems of finite-dimensional interacting particle systems and   find applications across most applied sciences, such as economics, finance, neuroscience, biology, physics, engineering, and many others.

{Following} the dynamic programming approach, we define the value functions $u_n:[0,T]\times H^n \to \mathbb{R}$ by
\begin{equation}\label{finite_dimensional_value_function}
	u_n(t,\mathbf{x}) := \inf_{\mathbf{a}(\cdot)\in \Lambda^n_t} J_n(t,\mathbf{x};\mathbf{a}(\cdot)),
\end{equation}
where $\Lambda^n_t$ is an appropriate set of optimal controls. The value function $u_n$ {is the $B$-continuous viscosity solution of} the Hamilton--Jacobi--Bellman (HJB) equation 
\begin{equation}\label{finite_dimensional_hjb}
\begin{cases}
	\partial_t u_n + \frac12 \text{Tr}(A_n(\mathbf{x},\mu_{\mathbf{x}}) D^2u_n)\\
    \qquad\qquad + \frac{1}{n} \sum_{i=1}^n \left ( \langle Ax_i,nD_{x_i} u_n \rangle + \mathcal{H}(x_i,\mu_{\mathbf{x}},nD_{x_i}u_n) \right ) = 0, \quad(t,\mathbf{x})\in (0,T)\times H^n\\
	u_n(T,\mathbf{x}) = \frac{1}{n} \sum_{i=1}^n \mathcal{U}_T(x_i,\mu_{\mathbf{x}}), \quad \mathbf{x}\in H^n,
\end{cases}
\end{equation}
where $\mu_{\mathbf{x}}:=\frac 1 n \sum_{i=1}^n \delta_{x_i}$ and $A_n(\mathbf{x},\mu_{\mathbf{x}})$ is an $n\times n$-matrix consisting of $n^2$ trace-class operators $(A_n)_{ij}(\mathbf{x},\mu_{\mathbf{x}}):H\to H$, given by
\begin{equation}
    (A_n)_{ij}(\mathbf{x},\mu_{\mathbf{x}}):=\sigma(x_i,\mu_{\mathbf{x}}) \sigma^{\top}(x_j,\mu_{\mathbf{x}}),
\end{equation}
$i,j=1,\dots,n$, and $\mathcal{H}: H \times \mathcal{P}_2(H) \times H \to \mathbb{R}$ denotes the Hamiltonian given by
\begin{equation}\label{hamiltonian}
	\mathcal{H}(x,\mu,p) :=  \inf_{q\in \tilde{\Lambda}} ( \langle f(x,\mu,q) , p\rangle + l(x,\mu,q)).
\end{equation}
Notice that \eqref{finite_dimensional_hjb} is a degenerate partial differential equation (PDE) on the Hilbert space $H^n$ {which contains an unbounded operator $A$ {in the linear first order term}.} Therefore, it must be dealt with {using appropriate viscosity solution techniques {(see \cite[Chapter 3]{fabbri_gozzi_swiech_2017}).}}

Formally, the value functions $u_n$, {when converted to functions of measures,} converge to the value function ${\mathcal{V}}$ of a control problem in the Wasserstein space $\mathcal{P}_2(H)$ which formally solves the HJB equation
\begin{equation}\label{intro:HJB_on_Wasserstein_space}
\begin{cases}
    \partial_t \mathcal{V}(t,\mu) {+} \int_H \langle Ax, \partial_{\mu} \mathcal{V}(t,\mu)(x) \rangle \mu(\mathrm{d}x) + \frac12 \int_{H} \text{Tr} \left [ D_x \partial_{\mu} \mathcal{V}(t,\mu)(x)\sigma(x,\mu)\sigma^{\top}(x,\mu) \right ] \mu(\mathrm{d}x)\\
    \quad + \frac12 \int_{H\times H} \text{Tr} \left [ \partial_{\mu}^2 \mathcal{V}(t,\mu)(x,x^{\prime}) \sigma (x,\mu) \sigma^{\top}(x^{\prime},\mu) \right ] \mu(\mathrm{d}x) \mu(\mathrm{d}x^{\prime}) \\
    \quad {+} \int_{H} \mathcal{H}(x,\mu,\partial_{\mu} \mathcal{V}(t,\mu)(x)) \mu(\mathrm{d}x) =0, \quad (t,\mu)\in (0,T) \times \mathcal{P}_2(H)\\
    \mathcal{V}(T,\mu) = \int_{H} \mathcal{U}_T(x,\mu) \mu(\mathrm{d}x),\quad \mu \in \mathcal{P}_2(H).
\end{cases}
\end{equation}
{However, we will neither study the control problem in the Wasserstein space nor the HJB equation \eqref{intro:HJB_on_Wasserstein_space}.} 
{Instead, we will study the ``lifted'' version of \eqref{intro:HJB_on_Wasserstein_space} on the Hilbert space $E=L^2(\Omega;H)$, where $\Omega=(0,1)$, which is satisfied by the lift $V$ of $\mathcal V$. It has the form}
\begin{equation}\label{intro:liftedHJB}
\begin{cases}
	\partial_t V + \frac12 \text{Tr}(\Sigma(X)(\Sigma(X))^{\ast} D^2V) {+} \llangle \mathcal{A}X,DV \rrangle  {+} \tilde{\mathcal{H}}(X, DV) =0,\quad (t,X)\in (0,T)\times E\\
	V(T,X) = U_T(X), \quad X\in E,
\end{cases}
\end{equation}
where  $\mathcal{A} : \mathcal{D}(\mathcal{A}) \subset E \to E$, $\mathcal{A}X(\omega) := A(X(\omega))$ is an unbounded operator and other terms are suitably defined, see Section \ref{sec:control}. {This lifting procedure allows us to use the well-developed theory of viscosity solutions for unbounded HJB equations in Hilbert spaces, see {\cite[Chapter 3]{fabbri_gozzi_swiech_2017}}.} Moreover, thanks to the structure of \eqref{intro:liftedHJB},  {it can be associated with a} ``lifted'' limit control problem, see \eqref{Lifted_State_Equation}-\eqref{Lifted_Value_Function}, {which enables us to apply control theoretic techniques.}  The lifting procedure was introduced for $H=\mathbb R^d$ in \cite{lions_2007-2011,cardaliaguet_2013} {and studied in \cite{ahuja_ren_yang_2019,bensoussan_huang_tang_yam_2025,carmona_delarue_2018,gangbo_tudorascu_2019}.} However, we remark that the ``lifted'' HJB \eqref{intro:liftedHJB} is new in the {framework of} a general Hilbert space $H$.

\paragraph{\textbf{Literature: case $H=\mathbb R^d$.}} The study of  {controlled} particle systems when $H={\mathbb R}^d$ and, especially, the problem of convergence of their value functions $u_n$ to a value function ${\mathcal{V}}$ of a control problem in the space of measures have attracted considerable attention in recent years. In particular, convergence results for the case of idiosyncratic noise were obtained in \cite{lacker_2017} (see also \cite{budhiraja_dupuis_fischer_2012,carmona_delarue_2015,feng_kurtz,fischer_livieri_2016,fornasier_solobmrino_2014} for earlier and related results, including large deviations), and for the case of common noise in \cite{djete_possamai_tan_2022}. {For related problems, see also the recent papers \cite{cardaliaguet_jackson_souganidis_2025,delarue_martini_sodini_2025}.} In \cite{djete_2022} so-called extended mean-field control problems, i.e., control problems that involve the law of the control, were considered. {Viscosity solution theory for the lifted HJB equation was applied to prove convergence in the common noise setting in \cite{gangbo_mayorga_swiech_2021,mayorga_swiech_2023,swiech_wessels_2024} via the lifting procedure (while a rigorous extension of viscosity theory for the lifted HJB equation in the case of both common noise and idiosyncratic noise is still open also for $H=\mathbb R^d$, see e.g. \cite[Remark 2.20]{daudin_seeger_2024})  and  in \cite{talbi_2024} for the HJB on the Wasserstein space}.
The rate of convergence of the value functions for finite particle control problems has also been studied extensively. The first result in this direction for both idiosyncratic and common noise was obtained in \cite{germain_pham_warin_2022}, where, assuming the existence of a classical solution of the infinite dimensional limit HJB equation, the rate of convergence of order $\mathcal{O}(1/n)$ was proved. A different and simpler argument to obtain this result was pointed out in \mbox{\cite[Section 1.3]{cardaliaguet_daudin_jackson_souganidis_2023}}. A similar method was applied earlier in the context of mean-field games \mbox{\cite{cardaliaguet_delarue_lasry_lions_2019,carmona_delarue_2018_2}}. If the value function ${\mathcal{V}}$ is not smooth, the situation becomes much more complicated. We refer to \cite{Bayraktar_Ekren_Zhang_2025,cardaliaguet_daudin_jackson_souganidis_2023,cardaliaguet_jackson_mimikos-stamatopoulos_souganidis_2023,cardaliaguet_souganidis_2023,cecchin_daudin_jackson_martini_2024,daudin_delarue_jackson_2024,daudin_jackson_seeger_2025} for various results on the rate of convergence of $u_n$ using the notion of viscosity solutions. Rate of convergence results in these papers depend on the regularity of the data and hence of the regularity of value functions, some also depend on the dimension $d$ of the underlying state space.

Moreover, second order HJB equations in spaces of probability measures on finite dimensional sets have been studied quite intensively in recent years. Various results about uniqueness of their solutions can be found in \cite{Bayraktar_Cheung_Ekren_Qiu_Tai_Zhang_2025,Bayraktar_Ekren_Zhang_2025a,Bayraktar_Ekren_He_Zhang_2025,bensoussan_graber_yam_2024,bensoussan_graber_yam_2025,bertucci_lions_2024,burzoni_ignazio_reppen_soner_2020,cardaliaguet_jackson_souganidis_2025,Cheung_Tai_Qiu_2025,cosso_gozzi_kharroubi_pham_rosestolato_2024,cox_kallblad_larsson_svaluto-ferro_2024,daudin_jackson_seeger_2023,daudin_seeger_2024,martini_2023,martini_2024,soner_yan_2024_2,soner_yan_2024,touzi_zhang_zhou_2024,wu_zhang_2020}. 

In particular, the  problem \eqref{state_equation}-\eqref{cost_functional} when $H={\mathbb R}^d$, $A=0$, and with a more specific structure of some coefficients was studied in \cite{mayorga_swiech_2023,swiech_wessels_2024} following the methods of \cite{gangbo_mayorga_swiech_2021}. It was proved there that $u_n(t,\mathbf{x})$ converge locally uniformly to a function $\mathcal{V}(t,\mu_{\mathbf{x}})$ whose ``lift'' is the unique  viscosity solution of {\eqref{intro:liftedHJB}}.
Moreover, it was proved in \cite{swiech_wessels_2024} that if the Fr\'echet derivative in the spatial variable of the lift of {$\mathcal{V}$} is continuous then
$u_n(t,\mathbf{x})=\mathcal{V}(t,\mu_{\mathbf{x}})$. A similar observation has also been made in \cite{liao_meszaros_mou_zhou_2024,cecchin_daudin_jackson_martini_2024} for problems considered there.
We remark that (when $H={\mathbb R}^d$ and $A=0$) it was already noticed in \cite{germain_pham_warin_2022} that in the absence of idiosyncratic noise, if $\mathcal{V}$ is a smooth solution of \eqref{intro:HJB_on_Wasserstein_space} then $u_n(t,\mathbf{x}):=\mathcal{V}(t,\mu_{\mathbf{x}})$ is a solution of \eqref{finite_dimensional_hjb}. 

Finally, we refer the reader to \cite{cardaliaguet_delarue_lasry_lions_2019,carmona_delarue_2018,carmona_delarue_2018_2} for introductions to the theory of mean field games and control.

\paragraph{\textbf{Literature: infinite dimensional case.}} Compared to the finite dimensional case, there are few results on mean field games, mean field control and optimal control of stochastic interacting particle systems  in infinite dimensional spaces. Some results have been obtained for mean field games in Hilbert spaces in \cite{federico_gozzi_ghilli_2024,federico_gozzi_swiech_2024,ghilli_ricciardi_2025,liu_firoozi_2025,liu_firoozi_2026,firoozi_kratsios_yang_2025,fouque_zhang_2018,ricciardi_rosestolato_2024,swiech_wessels_2026} 
and for mean field control problems in Hilbert spaces in \cite{cosso_gozzi_kharroubi_pham_rosestolato,criens_2026,criens_ritter_2023,djehiche_gozzi_zanco_zanella_2022,dumitrescu_oksendal_sulem_2018,gozzi_masiero_rosestolato_2024,guatteri_masiero_wessels_2025,spille_stannat_2025,tang_meng_wang_2019}.  We  refer to \cite{shi_wang_yong_2013} for mean field control of stochastic Volterra integral equations, \cite{meng_shen_2015} for control of mean-field jump-diffusion systems with delay, \cite{ma_shi_wang_2024} for partially observed stochastic control systems with delay, and to \cite{buckdahn_li_li_xing_2025,ren_tan_touzi_yang_2023} for path-dependent mean-field control and games.  {In particular, to the best of our knowledge, the only papers that study the convergence of particle systems of controlled stochastic evolution equations on an infinite-dimensional Hilbert space to the corresponding mean-field limits }{are \cite{criens_2026} and \cite{criens_ritter_2023}}. {Both consider idiosyncratic noise and use tightness/compactness together with martingale-problem methods: \cite{criens_2026} works in the variational SPDE framework, whereas \cite{criens_ritter_2023} treats path-dependent SPDEs in the mild semigroup framework}\footnote{Related finite-player convergence results are obtained in \cite{fouque_zhang_2018} for linear-quadratic mean-field games with delayed controls and in \cite{liu_firoozi_2025} for linear-quadratic mean-field games on a Hilbert space.}. The other contributions analyze the limiting mean-field problem directly.

Finally, we refer to \cite{aurell_carmona_lauriere_2022,caines_huang_2020,Neuman_Tuschmann_2024} for mean-field games and to \cite{cao_lauriere_2025,decrescenzo_fuhrman_kharroubi_pham_2024,decrescenzo_defeo_pham_2025,defeo_mekkaoui,djete_2025,kharroubi_mekkaoui_pham_2025} for mean-field control under non-exchangeability/heterogeneity (or graphon interactions). In particular,  \cite{caines_huang_2020,Neuman_Tuschmann_2024} address the convergence in the graphon mean-field game case  and the very recent papers \cite{cao_lauriere_2025,djete_2025} in the graphon mean-field control case.
These are  infinite dimensional formulations, but very  different from ours, as the goal there is to model heterogeneity among agents and {the equations} do not {contain unbounded operators, which are typical features of  abstract evolution equations {in infinite dimensional spaces}.}   

Among all these papers, the only ones  which deal with  viscosity solutions in spaces of probability measures over infinite dimensional state spaces  seem to be  \cite{cosso_gozzi_kharroubi_pham_rosestolato} and \cite{decrescenzo_fuhrman_kharroubi_pham_2024}, {but the authors there do not provide any uniqueness results. This is indeed a very difficult problem even when  $H=\mathbb R^d$.
}

\paragraph{\textbf{Our theoretical contributions.}}
It is evident from the literature review that the study of optimal control problems for interacting particle systems governed by stochastic evolution equations in infinite-dimensional Hilbert spaces, i.e. \eqref{state_equation}-\eqref{cost_functional}, and their mean-field limits (including convergence and uniqueness of solutions of the limiting PDE) is an open area of research. The purpose of this paper is  to build such a theory.

As already explained before, to study our problem we generalize the approach from \cite{gangbo_mayorga_swiech_2021,mayorga_swiech_2023,swiech_wessels_2024}. Instead of dealing with the PDE \eqref{intro:HJB_on_Wasserstein_space} we work directly with the  ``lifted'' HJB \eqref{intro:liftedHJB}. However, here the HJB equations \eqref{finite_dimensional_hjb} (on the Hilbert space $H^n$) and \eqref{intro:liftedHJB} (on $E$) are infinite dimensional. Thus we face significant mathematical  difficulties introduced by such intrinsic infinite dimensionality:
\begin{itemize}
    \item we lose local compactness of the state space $H$,
    \item  our abstract SDEs in $H$ and $E$ do not admit 
{strong} solutions (i.e. we   {neither} have $x_i(s) \in \mathcal D(A)$, {nor} $X(s)\in  \mathcal D(\mathcal A)$); hence we need to deal with weaker notions of solutions as typical in infinite dimensional settings,
    \item all HJB equations \eqref{finite_dimensional_hjb}, \eqref{intro:liftedHJB} are intrinsically infinite dimensional and contain unbounded operators.
\end{itemize}
{To circumvent the lack of local compactness and the unboundedness of the operator $A$, we work with a weaker space $H_{-1}\supset H$, and use the theory of the so-called $B$-continuous viscosity solutions\footnote{{After the original version of this manuscript was {posted on arXiv, the preprint \cite{tang_zhou_2026} appeared there, introducing a notion of viscosity solution in Hilbert spaces that does not require (B)-continuity.} One might be tempted to use in for our problem.} { 
However, in our setting the \(B\)-framework is not used only for comparison: the compactness of \(B\) provides the compact embedding into \(H_{-1}\) used in the particle-limit argument.  
Furthermore, the class considered in \cite{tang_zhou_2026} is not broad enough to cover our unbounded \(L^{2}\)-controls and quadratic control costs, nor does it provide the \(E_{-1}\)-regularity required for the exact projection result, making, instead, $B$-continuous viscosity solutions a very natural and  convenient framework.}} for both PDEs \eqref{finite_dimensional_hjb} and  \eqref{intro:liftedHJB}. We also work with mild solutions of the state SDEs.} {While these tools have been applied to address similar problems in Hilbert spaces, see e.g. \cite[Chapter 3]{fabbri_gozzi_swiech_2017}, we would like to emphasize that their application in this new context arising from control of interactive particle systems and its mean-field limit requires careful consideration:} for instance, we need to carefully study the semigroup of linear operators $e^{\mathcal A t}$ generated by $\mathcal A$, deal with the spaces of $H$- and $H_{-1}$-valued square integrable random variables $E= L^2(\Omega;H)$ and $E_{-1}:= L^2(\Omega;H_{-1})$, respectively, the spaces of probability measures on $H$ and $H_{-1}$, and introduce a suitable operator $\mathcal B : E \to E$ satisfying the so called weak $\mathcal B$-condition for $\mathcal A$. Moreover, due to a more general structure of the coefficients of the control problem, various properties of the Hamiltonian that are rather straightforward in the setting of \cite{gangbo_mayorga_swiech_2021,mayorga_swiech_2023,swiech_wessels_2024} require a more careful analysis in our case. For more details, see Section \ref{sec:assumptions}.

 Despite these difficult technical challenges, we are able to generalize successfully the corresponding results of \cite{swiech_wessels_2024} (and of the ones in \cite{gangbo_mayorga_swiech_2021,mayorga_swiech_2023} for (i), (ii)\footnote{{We also refer to \cite{defeo_swiech_wessels_2023} for $C^{1,1}$-regularity of the value function for stochastic optimal control on Hilbert spaces and to \cite{hu_liu_tang_2025} for for recursive stochastic optimal control problem on Hilbert spaces.}}) to our infinite dimensional framework. In particular:
 \begin{enumerate}
     \item[(i)] In  Section \ref{sec:convergence}, Theorem \ref{theorem:convergence}, we prove, {roughly speaking}, that $u_n(t,\mathbf{x})$ converge locally uniformly\footnote{{For the precise meaning, see \eqref{convergence_on_M^2_R}}} to $\mathcal{V}(t,\mu_{\mathbf{x}})$, where $\mathcal{V}$ is the unique $L$-viscosity solution of \eqref{intro:HJB_on_Wasserstein_space}, that is, its ``lift'' $V$ is the unique $\mathcal B$-continuous viscosity solution of \eqref{intro:liftedHJB}. We then identify $V$ with the value function $U$ of the ``lifted'' limit optimal control problem in the Hilbert space $E$, \eqref{Lifted_State_Equation}--\eqref{Lifted_Value_Function}.
     \item[(ii)] {Imposing additional assumptions,} in Section \ref{sec:regularity},  we use probabilistic methods to prove that $U(t,\cdot)\in C^{1,1}(E_{-1})$, where $E_{-1}$ is the space of square integrable $H_{-1}$-valued random variables.
     \item[(iii)] Under the assumption that the Fr\'echet derivative in $E_{-1}$ of $V(t,\cdot)$ is continuous, we show in Section \ref{sec:projection}, Theorem \ref{theorem:projection_C11}, that   the value function $V$  projects precisely onto the value functions $u_n$, i.e. $u_n(t,\mathbf{x})={\mathcal{V}}(t,\mu_{\mathbf{x}})$. Note that the differentiability assumption is in particular satisfied in the setting of (ii).
     \item[(iv)] In
      Section \ref{sec:lifting_proj}, we  obtain relationships between optimal {controls} and optimal feedback controls for problems \eqref{state_equation}-\eqref{cost_functional} and \eqref{Lifted_State_Equation}--\eqref{Lifted_Value_Function}, {i.e. we prove that under suitable assumptions optimal (resp. optimal feedback) controls of the particle system correspond to optimal (resp. optimal feedback) controls of the ``lifted''infinite dimensional control problem started at the corresponding initial condition.}
 \end{enumerate}
 \paragraph{\textbf{Comparison with the  {state of the art}.}}
The proofs of the above results are much more challenging than the proofs of  the corresponding results in \cite{swiech_wessels_2024} where $H={\mathbb R}^d$, due to the deep mathematical difficulties explained above which require the development of new ad-hoc techniques for these new types of problems.  To the best of our knowledge, these are the first results of this kind for stochastic optimal control problems for interacting particle systems driven by stochastic evolution equations in Hilbert spaces {via dynamic programming}. In particular, regarding (i), this seems to be the first uniqueness result for viscosity solutions of a mean-field PDE over an infinite dimensional state space (that we provide in terms of  uniqueness of $L$-viscosity solutions of  \eqref{intro:HJB_on_Wasserstein_space}, i.e.  in the sense of $\mathcal B$-continuous viscosity solution of the ``lifted'' HJB  \eqref{intro:liftedHJB}).
 {Regarding (iii) and (iv), we remark that such properties are not easy to obtain even in the case $H={\mathbb R}^d$ and even in that setting there are very few results of this kind.}

\paragraph{\textbf{Applications.}}
{In Section \ref{sec:applications}, we apply the developed theory} to problems arising in economics in the context of optimization for large companies: a (path-dependent) stochastic optimal control problem with delays arising in the context of optimal advertising, inspired by \cite{gozzi_masiero_rosestolato_2024}\footnote{see also \cite{ricciardi_rosestolato_2024} for a mean-field game framework with delays}, and a stochastic optimal control problem with vintage capital, inspired by \cite{cosso_gozzi_kharroubi_pham_rosestolato}, where the state equation is a stochastic partial differential equation. However, we remark  that the range of possible applications is much wider as explained at the beginning of the introduction, {spanning across} most applied sciences.


{To conclude, we believe that the manuscript contains substantial new mathematical insights and techniques that allow to tackle a new class of optimal control problems for interacting particle systems and their associated HJB equations. We hope that the paper will open a new line of research. The range of {potential} models and applications is {extensive}.}

\section{Notation}\label{sec:preliminaries}

Throughout this work, we use the following notation.

\begin{itemize}
    \item $H$ is a real, separable Hilbert space with inner product $\langle \cdot,\cdot \rangle$ and norm $|\cdot|$. For $\mathbf{x}\in H^n$, we denote $|\mathbf{x}|_{H^n} := ( \sum_{i=1}^n |x_i|^2 )^{1/2}$.
    \item Given a strictly positive $B\in S(H)$, we define the space $H_{-1}$ as the completion of $H$ with respect to the norm $|x|_{-1}^2 := \langle Bx,x\rangle$. {Here, $S(H)$ is the space of linear, bounded self-adjoint operators on $H$.}
    For more details, see \cite[Section 3.1.1]{fabbri_gozzi_swiech_2017}.  For $\mathbf{x}\in H_{-1}^n$, let $|\mathbf{x}|_{H_{-1}^n} := ( \sum_{i=1}^n |x_i|_{-1}^2 )^{1/2}$.
    \item For $r\in [1,2]$ and $\mathbf{x} = (x_1,\dots,x_n)\in H^n$, define
    \begin{equation}
        |{\bf x}|_{r}:=\frac{1}{n^{\frac{1}{r}}}\left(\sum_{i=1}^n |x_i|^r\right)^{\frac{1}{r}},\qquad\text{and}\qquad |{\bf x}|_{-1,r}:=\frac{1}{n^{\frac{1}{r}}}\left(\sum_{i=1}^n |x_i|_{-1}^r\right)^{\frac{1}{r}}.
    \end{equation}
        \item We set $\Omega = (0,1)$ and we denote $E:= L^2(\Omega;H)$ to be the space of all square-integrable $H$-valued random variables. We denote by $\llangle \cdot,\cdot \rrangle$ and $\|\cdot\|$ the canonical inner product and norm, respectively. Moreover, let $E_{-1} := L^2(\Omega;H_{-1})$. It is endowed with its natural inner product and norm denoted by $\llangle \cdot,\cdot \rrangle_{-1}$ and $\|\cdot\|_{-1}$.
            \item For $r\in [1,2]$, we denote by $\mathcal{P}_r(H)$ the set of all Borel probability measures on $H$ with finite $r$-th moment, i.e., $\mathcal{M}_r(\mu):= \int_{H} |x|^r \mu(\mathrm{d}x) < \infty$. Let $d_r:\mathcal{P}_r(H) \times \mathcal{P}_r(H) \to \mathbb{R}$ denote the Wasserstein distance, i.e.,
\begin{equation}\label{definition_wasserstein_distance}
    \begin{split}
        d_r(\mu,\beta) &= \inf_{\gamma\in \Gamma(\mu,\beta)} \left (\int_{H\times H} |x-y|^r \gamma(\mathrm{d}x,\mathrm{d}y) \right )^{\frac1r}\\
        &= \inf \left \{ \left ( \int_{\Omega} |X(\omega)-Y(\omega)|^r \mathrm{d}\omega \right )^{\frac1r}: \; X,Y\in L^r(\Omega;H), \; X_{\texttt{\#}} \mathcal{L}^1 = \mu, \; Y_{\texttt{\#}} \mathcal{L}^1 = \beta \right \}.
    \end{split}
    \end{equation}
    Here, $\Gamma(\mu,\beta)$ is the set of all probability measures on $H \times H$ with first and second marginals $\mu$ and $\beta$, respectively. The proof of the second equality in \eqref{definition_wasserstein_distance} can be found for instance in \cite[Theorem 3.9]{gangbo_2018}. Note that $d^r_r(\mu,\delta_0) = \mathcal{M}_r(\mu)$. For $\mathbf x \in H^n$, we denote $$\mu_{\mathbf{x}}:=\frac 1 n \sum_{i=1}^n \delta_{x_i}.$$ Then, we have $d_r(\mu_{\mathbf{x}},\mu_{\mathbf{y}}) = \inf_{\sigma} |\mathbf{x} - \mathbf{y}_{\sigma} |_r$, where the infimum is taken over all permutation $\sigma$ of $\{1,\dots,n\}$, and $\mathbf{y}_{\sigma} = (y_{\sigma(1)},\dots, y_{\sigma(n)})$. Let $\mathcal{P}_r(H_{-1})$, $\mathcal{M}_{-1,r}$ and $d_{-1,r}$ be defined analogously with $H$ replaced by $H_{-1}$. {Extensive accounts of the theory of mass transport and Wasserstein spaces in abstract metric spaces can be found in \cite{ambrosio_gigli_savare_2008,villani_2009}.}
    \item We denote by $\mathcal{L}^1$ the Lebesgue measure on $\mathbb{R}$. For $X\in E$, let $X_{\texttt{\#}} \mathcal{L}^1$ be the pushforward measure on $H$.
    \item For $n\in \mathbb{N}$, let $A_i^n = (\frac{i-1}{n},\frac{i}{n}) \subset (0,1) =\Omega$, $i=1,\dots,n$. Define the lift of $\mathbf{x} = (x_1,\dots,x_n)\in H^n$ by
    \begin{equation}\label{E_n}
        X^{\mathbf{x}}_n = \sum_{i=1}^n x_i \mathbf{1}_{A^n_i},
    \end{equation}
   and let $E_n$ denote the subspace of $E$ consisting of random variables of the form \eqref{E_n}. Note that 
   \begin{equation}\label{eq:Xnxpushforw=mux}
       (X^{\mathbf{x}}_n)_{\texttt{\#}} \mathcal{L}^1 = \mu_{\mathbf{x}}.
   \end{equation}
   \item For real, separable Hilbert spaces $\Xi$ and $H$, $L(\Xi;H)$ denotes the space of bounded linear operators from $\Xi$ to $H$ equipped with the norm $\|\cdot\|_{L(\Xi;H)}$. If $\Xi=H$ we will just write $L(H)$. We write $L_2(\Xi,H)$ to denote the space of Hilbert--Schmidt operators in $L(\Xi;H)$ and use $|\cdot|_{L_2(\Xi,H)}$ to denote its norm.
   \item $\Lambda$ is a real, separable Hilbert space with inner product $\langle \cdot,\cdot\rangle_{\Lambda}$ and norm $|\cdot|_{\Lambda}$, and let $\tilde{\Lambda} \subset \Lambda$ be convex. For $\mathbf{q} \in \Lambda^n$, we write $|\mathbf{q}|_{\Lambda^n} := ( \sum_{i=1}^n |q_i|_{\Lambda}^2 )^{1/2}$. We denote $\mathcal{E} := L^2(\Omega;\tilde{\Lambda})\subset L^2(\Omega;\Lambda)$ and use $\llangle \cdot,\cdot \rrangle_{\Lambda}$ and $\|\cdot\|_{\Lambda}$ to denote, respectively, the natural inner product and norm in this space.
\end{itemize}

\section{The Control Problems}\label{sec:control}

Let $\Xi$ be a real, separable Hilbert space. For every $0\leq t <T$, we consider a reference probability space $(\Omega',\mathcal{F},\mathcal{F}^t_s,\mathbb{P},W)$, where $W$ is a cylindrical Wiener process in $\Xi$ (see \cite[Definition 2.7]{fabbri_gozzi_swiech_2017} for the definition of a reference probability space.

\subsection{The Finite Particle System Control Problem}\label{subsec:finite_particle-system}
For the finite particle control problem \eqref{state_equation}-\eqref{cost_functional} we have $f: H \times \mathcal{P}_2(H) \times \Lambda \to H$, $\sigma: H \times \mathcal{P}_2(H) \to L_2(\Xi,H)$, $l: H \times \mathcal{P}_2(H) \times \Lambda \to \mathbb{R}$ and $\mathcal{U}_T:H\times \mathcal{P}_2(H) \to \mathbb{R}$. The precise assumptions on these functions will be given in Section \ref{sec:assumptions}. The control processes 
$\mathbf{a}(\cdot)=(a_1(\cdot),\dots,a_n(\cdot))\in \Lambda^n_t$, where the set of admissible controls $\Lambda^n_t$ consists of all processes $\mathbf{a}(\cdot):[t,T]\times \Omega^{\prime} \to \tilde{\Lambda}^n$ which are $\mathcal{F}^t_s$-progressively measurable, and such that 
\[
\| \mathbf{a}(\cdot) \|_{M^2(t,T;\Lambda^n)}^2 := \mathbb{E} \left[ \int_t^T \sum_{i=1}^n |a_i(s)|_{\Lambda}^{2} \mathrm{d}s \right] < \infty.
\]
We remind the reader that under reasonable assumptions (in particular if Assumptions \ref{Assumption_f_sigma_lipschitz} and \ref{Assumption_running_terminal_cost} are satisfied), the above described control problem does not depend on the choice of a reference probability space, see \cite[Section 2.3.2]{fabbri_gozzi_swiech_2017}.

\subsection{The Lifted Limit Control Problem}\label{subsection_lifted_control_problem}

Recall that $E=L^2(\Omega;H)$ and $\mathcal{E} = L^2(\Omega;\tilde{\Lambda})$, where $\Omega=(0,1)$. We now consider the infinite dimensional SDE
\begin{equation}\label{Lifted_State_Equation}
	\begin{cases}
		\mathrm{d}X(s) = [\mathcal{A}X(s) + F(X(s),a(s))] \mathrm{d}s + \Sigma(X(s)) \mathrm{d}W(s),\quad s\in [t,T]\\
		X(t) = X \in E,
	\end{cases}
\end{equation}
where $\mathcal{A} : \mathcal{D}(\mathcal{A}) \subset E \to E$ is defined as $\mathcal{A}X(\omega) := A(X(\omega))$, $F: E \times \mathcal{E}\to E$ is defined as $F(X,Q)(\omega) := f(X(\omega),X_{\texttt{\#}}\mathcal{L}^1,Q(\omega))$, and $\Sigma : E \to L_2(\Xi , E)$ is defined as $\Sigma(X)(\omega) := \sigma(X(\omega),X_{\texttt{\#}}\mathcal{L}^1)$. We consider the cost functional
\begin{equation}
	J(t,X;a(\cdot)) = \mathbb{E} \left [ \int_t^T L(X(s),a(s)) \mathrm{d}s + U_T(X(T)) \right ],
\end{equation}
where $L: E\times \mathcal{E} \to \mathbb{R}$ and $U_T:E\to\mathbb{R}$ are given by
\begin{equation}
	L(X,Q) := \int_{\Omega} l(X(\omega),X_{\texttt{\#}}\mathcal{L}^1,Q(\omega)) \mathrm{d}\omega, \quad U_T(X) := \int_{\Omega} \mathcal{U}_T(X(\omega),X_{\texttt{\#}}\mathcal{L}^1) \mathrm{d}\omega.
\end{equation}
The set of admissible controls $\Lambda_t$ consists of all processes $a(\cdot):[t,T]\times \Omega^{\prime} \to \mathcal{E}$ which are $\mathcal{F}^t_s$-progressively measurable, and
\[
\| a(\cdot) \|_{M^2(t,T;\mathcal{E})}^2 := \mathbb{E} \left[ \int_t^T \|a(s)\|_{\Lambda}^2 \mathrm{d}s \right] < \infty. 
\]
In this case, the value function $U:[0,T]\times E \to \mathbb{R}$ is defined as
\begin{equation}\label{Lifted_Value_Function}
	U(t,X) := \inf_{a(\cdot)\in \Lambda_t} J(t,X;a(\cdot)),
\end{equation}
and the corresponding ``lifted'' HJB equation is
\begin{equation}\label{lifted_HJB_equation}
\begin{cases}
	\partial_t V + \frac12 \text{Tr}(\Sigma(X)(\Sigma(X))^{\ast} D^2V) {+} \llangle \mathcal{A}X,DV \rrangle  {+} \tilde{\mathcal{H}}(X, DV) =0,\quad (t,X)\in (0,T)\times E\\
	V(T,X) = U_T(X), \quad X\in E,
\end{cases}
\end{equation}
where the lifted Hamiltonian $\tilde{\mathcal{H}}: E\times E \to \mathbb{R}$ is given by
\begin{equation}\label{lifted_hamiltonian_definition}
	\tilde{\mathcal{H}}(X,P) := \int_{\Omega} \mathcal{H}(X(\omega),X_{\texttt{\#}}\mathcal{L}^1,P(\omega)) \mathrm{d}\omega.
\end{equation}

\section{Assumptions and Preliminaries}\label{sec:assumptions}

\begin{assumption}\label{Assumption_A_maximal_dissipative}
    The operator $A:\mathcal{D}(A)\subset H \to H$ is a linear, densely defined, maximal dissipative operator.
\end{assumption}

Under this assumption, $A$ generates a $C_0$-semigroup of contractions $(e^{sA})_{s\geq 0}$ on $H$.

\begin{assumption}\label{Assumption_weak_B_condition}(Weak $B$-condition)
There exists a strictly positive {operator $B\in S(H)$} such that $A^{\ast} B\in L(H)$, and $-A^{\ast} B+c_0 B \geq 0$, for some $c_0\geq 0$.
\end{assumption}

\begin{assumption}\label{Assumption_B_compact}
    The operator {$B$} in Assumption \ref{Assumption_weak_B_condition} is compact.
\end{assumption}

\begin{assumption}\label{Assumption_f_sigma_lipschitz}
    Let $r \in [1,2)$. Let $f: H \times \mathcal{P}_2(H) \times \Lambda \to H$ be given by $f(x,\mu,q) = f_1(x,\mu) + f_2(x,\mu,q)$ for {continuous functions} $f_1: H\times \mathcal{P}_2(H) \to H$ and $f_2: H \times \mathcal{P}_2(H) \times \Lambda \to H$. Moreover, let $\sigma: H \times \mathcal{P}_2(H) \to L_2(\Xi,H)$. Let the following conditions be satisfied.
	\begin{enumerate}[label=(\roman*)]
		\item There exists a constant $C\geq 0$ such that
		\begin{align}\label{Lipschitz_f_weak_norm}
			\langle f(x,\mu,p) - f(y,\beta,q), B(x-y) \rangle & \leq C( |x-y|^2_{-1} + d^2_{-1,r}(\mu,\beta) + | p - q |_{\Lambda}^2 )\\
           | f_1(x,\mu) - f_1(y,\beta) |+| f_2(x,\mu,q) - f_2(y,\beta,q) |&\leq C( |x-y| + d_{r}(\mu,\beta))
        \end{align}
        for all $p,q\in \tilde{\Lambda}$, $x,y\in H$ and $\mu,\beta\in \mathcal{P}_2(H)$. 
        \item There exists a constant $C\geq 0$ such that
        \begin{align}
            |f(x,\mu,q)|_{-1} &\leq C\left ( 1 + |x|_{-1} + \mathcal{M}_{-1,r}^{\frac{1}{r}}(\mu) + |q|_{\Lambda} \right )\\
            |f_2(x,\mu,q)| &\leq C(1+|q|_{\Lambda})
        \end{align}
        for all $x\in H$, $\mu\in \mathcal{P}_2(H)$ and $q\in \tilde{\Lambda}$.  
        \item The function $H\times E \times \Lambda \ni (x,X,q) \mapsto \tilde{f}(x,X,q) := f(x,X_{\texttt{\#}}\mathcal{L}^1,q)$ is Fr\'{e}chet differentiable with Fr\'{e}chet derivative $D \tilde f=(D_x \tilde f,D_X \tilde f, D_q \tilde f) :  H\times E \times \Lambda \to L(H) \times L(E;H)\times L(\Lambda;H)$ and there is a constant $C\geq 0$ such that
        \begin{equation}
            |( D\tilde{f}(x,X,p) - D\tilde{f}(y,Y,q))(x-y,X-Y,p-q)|_{H_{-1}^3}\leq C ( |x-y|_{-1}^2 + \|X-Y\|_{-1}^2 + |p-q|_{\Lambda}^2 )
        \end{equation}
        for all $x,y\in H$, $X,Y \in E$, and $p,q\in \tilde{\Lambda}$.
		\item There exists a constant $C\geq 0$ such that	\begin{equation}\label{Lipschitz_sigma}
			|\sigma(x,\mu) - \sigma(y,\beta)|_{L_2(\Xi,H)} \leq C( |x-y|_{-1} + d_{-1,r}(\mu,\beta)),
		\end{equation}
		for all $x,y\in H$ and $\mu,\beta\in \mathcal{P}_2(H)$. 
        \item The function $H\times E \ni (x,X) \mapsto \tilde{\sigma}(x,X) := \sigma(x,X_{\texttt{\#}}\mathcal{L}^1)$ is Fr\'{e}chet differentiable with Fr\'{e}chet derivative $D \tilde \sigma=(D_x \tilde \sigma,D_X \tilde \sigma) :  H\times E  \to L(H;L_2(\Xi,H)) \times L(E;L_2(\Xi,H))$ and there is a constant $C\geq 0$ such that
        \begin{equation}
        \begin{split}
            &| ( D_x \tilde{\sigma}(x,X) - D_x \tilde{\sigma}(y,Y) ) (x-y) |_{L_2(\Xi,H_{-1})} + | ( D_X \tilde{\sigma}(x,X) - D_X \tilde{\sigma}(y,Y) ) ( X - Y ) |_{L_2(\Xi,H_{-1})}\\
            &\leq C ( |x-y|_{-1}^2 + \|X-Y\|_{-1}^2 )
        \end{split}
        \end{equation}
        for all $x,y\in H$, and $X,Y \in E$.
	\end{enumerate}
\end{assumption}

\begin{assumption}\label{Assumption_running_terminal_cost}
    Let $r\in [1,2)$. Let the running cost $l: H \times \mathcal{P}_2(H) \times \Lambda \to \mathbb{R}$ be given by $l(x,\mu,q) = l_1(x,\mu) + l_2(x,\mu,q)$ for some continuous functions $l_1: H \times \mathcal{P}_2(H) \to \mathbb{R}$ and $l_2: H \times \mathcal{P}_2(H) \times \Lambda \to \mathbb{R}$. Let the terminal cost function $\mathcal{U}_T:H\times \mathcal{P}_2(H) \to \mathbb{R}$. Let the following conditions be satisfied.
	\begin{enumerate}[label=(\roman*)]
		\item There exists a constant $C\geq 0$ such that
		\begin{equation}
			|l_1(x,\mu) - l_1(y,\beta)| + |l_2(x,\mu,q) - l_2(y,\beta,q)| \leq C(|x-y|_{-1}+d_{-1,r}(\mu,\beta))
		\end{equation}
		for all $q \in \tilde{\Lambda}$, $x,y\in H$ and $\mu,\beta\in \mathcal{P}_2(H)$.
        \item There exist constants $C_1\geq 0$ and $C_2,C_3 > 0$ such that
        \begin{equation}
            -C_1 + C_2 |q|_{\Lambda}^2 \leq l_2(x,\mu,q) \leq C_1 + C_3 |q|_{\Lambda}^2
        \end{equation}
        for all  $q \in \tilde{\Lambda}$, $x\in H$, $\mu\in \mathcal{P}_r(H)$.
        \item The function $H\times E \times \Lambda \ni (x,X,q) \mapsto \tilde{l}(x,X,q) := l(x,X_{\texttt{\#}}\mathcal{L}^1,q)$ is Fr\'{e}chet differentiable with Fr\'{e}chet derivative $D \tilde l=(D_x \tilde l,D_X \tilde l, D_q \tilde l) :  H\times E \times \Lambda \to L(H;\mathbb R) \times L(E;\mathbb R)\times L(\Lambda;\mathbb R)$ and there is a constant $C\geq 0$ such that
        \begin{equation}
        \begin{split}
            |(D\tilde{l}(x,X,p) - D\tilde{l}(y,Y,q))(x-y,X-Y,p-q) |\leq C ( |x-y|_{-1}^2 + \|X-Y\|_{-1}^2 + |p-q|_{\Lambda}^2 )
        \end{split}
        \end{equation}
        for all $x,y\in H$, $X,Y \in E$, and $p,q\in \tilde{\Lambda}$.
		\item There exists a constant $C\geq 0$ such that
		\begin{equation}
			|\mathcal{U}_T(x,\mu) - \mathcal{U}_T(y,\beta)| \leq C(|x-y|_{-1} +d_{-1,r}(\mu,\beta)),
		\end{equation}
        for all $x,y\in H$, $\mu,\beta\in \mathcal{P}_2(H)$.
        \item The function $H\times E \ni (x,X) \mapsto \tilde{\mathcal{U}}_T(x,X) := \mathcal{U}_T(x,X_{\texttt{\#}}\mathcal{L}^1)$ is Fr\'echet differentiable with Fr\'echet derivative $D \tilde {\mathcal{U}}_T=(D_x \tilde {\mathcal{U}}_T,D_X \tilde {\mathcal{U}}_T) :  H\times E  \to L(H;\mathbb R) \times L(E;\mathbb R)$, and there is a constant $C\geq 0$ such that
        \begin{equation}
            |(D\tilde{\mathcal{U}}_T(x,X) - D\tilde{\mathcal{U}}_T(y,Y))(x-y,X-Y) | \leq C ( |x-y|_{-1}^2 + \| X-Y\|_{-1}^2 ),
        \end{equation}
        for all $x,y\in H$, $X,Y\in E$.
	\end{enumerate}
\end{assumption}

\begin{assumption}\label{Assumption_running_cost_nu}
    There exist constants $C_1,C_2,\nu\geq 0$ such that the map
    \begin{equation}
        H \times E \times \tilde{\Lambda} \ni (x,X,q) \mapsto \tilde{l}(x,X,q) + C_1|x|_{-1}^2 + C_2\|X\|_{-1}^2 - \nu |q|_{\Lambda}^2
    \end{equation}
    is convex, where $\tilde{l}$ is defined as in Assumption \ref{Assumption_running_terminal_cost}.
\end{assumption}

\begin{assumption}\label{Assumption_Linear_State_Equation}
	\begin{enumerate}[label=(\roman*)]
        \item The function $f: H \times \mathcal{P}_2(H) \times \Lambda \to H$ is such that its lift $F: E \times \mathcal{E} \to E$ (as defined in Subsection \ref{subsection_lifted_control_problem}) is affine linear, and there is a constant $C\geq 0$ such that
        \begin{align}
            \| F(X,P) - F(Y,Q) \| &\leq C \left ( \|X-Y\| + \|P-Q\|_{\Lambda} \right )\\
            \| F(X,P) - F(Y,Q) \|_{-1} &\leq C \left ( \|X-Y\|_{-1} + \|P-Q\|_{\Lambda} \right )
        \end{align}
        for all $X,Y\in E$, $P,Q\in \mathcal{E}$.
		\item The function $\sigma: H \times \mathcal{P}_2(H) \to L_2(\Xi,H)$ is such that its lift $\Sigma: E \to L_2(\Xi,E)$ (as defined in Subsection \ref{subsection_lifted_control_problem}) is affine linear, and there is a constant $C\geq 0$ such that
        \begin{equation}
            \| \Sigma(X) - \Sigma(Y) \|_{L_2(\Xi,E)} \leq C \|X-Y\|_{-1}
        \end{equation}
        for all $X,Y\in E$.
		\item The functions $l: H \times \mathcal{P}_2(H) \times\Lambda \to \mathbb{R}$ and $\mathcal{U}_T : H \times \mathcal{P}_2(H) \to \mathbb{R}$ are such that $\tilde l :H\times E \times \mathcal{E} \to \mathbb{R}$ and $\tilde{\mathcal{U}}_T: H\times E \to \mathbb{R}$ (as defined in Assumption \ref{Assumption_running_terminal_cost}) are convex.
	\end{enumerate}
\end{assumption}
{We will show in Section \ref{sec:applications} how Assumption \ref{Assumption_Linear_State_Equation} can be verified in practice.}
\subsection{Properties of the Hamiltonian \texorpdfstring{$\mathcal{H}$}{H} }\label{subsection_hamiltonian}

Recall the definition of the Hamiltonian $\mathcal{H}$ from equation \eqref{hamiltonian}. We define $\mathcal{H}_n : H^n \times \mathcal{P}_2(H) \times H^n \to \mathbb{R}$ as
\begin{equation}
    \mathcal{H}_n(\mathbf{x},\mu,\mathbf{p}) := \frac1n \sum_{i=1}^n \mathcal{H}(x_i,\mu,n p_i) = \frac1n \inf_{\mathbf{q}\in \tilde{\Lambda}^n} \sum_{i=1}^n \left ( \langle f(x_i,\mu,q_i), np_i \rangle + l(x_i,\mu,q_i) \right ).
\end{equation}
For $m>0$, we define $\mathcal{H}^m_n : H^n\times \mathcal{P}_2(H) \times H^n \to \mathbb{R}$ as
\begin{equation}
    \mathcal{H}^m_n(\mathbf{x},\mu,\mathbf{p}) := \frac1n \inf_{\substack{\mathbf{q} \in \tilde{\Lambda}^n\\ |\mathbf{q}|_{\Lambda^n} \leq m}} \sum_{i=1}^n \left ( \langle f(x_i,\mu,q_i), np_i \rangle + l(x_i,\mu,q_i) \right ).
\end{equation}

\begin{lemma}\label{lemma_hamiltonian_n}
    Let Assumptions \ref{Assumption_f_sigma_lipschitz}(ii) and \ref{Assumption_running_terminal_cost}(ii) be satisfied. Then, for every $\tilde C>0$, there is a constant $K>0$ such that
    \begin{equation}
        \mathcal{H}_n(\mathbf{x},\mu,\mathbf{p}) = \mathcal{H}_n^{K\sqrt{n}}(\mathbf{x},\mu,\mathbf{p})
    \end{equation}
    for all $n\in \mathbb{N}$, $\mu\in\mathcal{P}_2(H)$, $\mathbf{x},\mathbf{p} \in H^n$, $|\mathbf{p}|_{H^n} \leq \tilde C/\sqrt{n}$.
\end{lemma}

\begin{proof}
    We have
    \begin{align}
        \mathcal{H}_n(\mathbf{x},\mu,\mathbf{p}) &=  \frac{1}{n} \inf_{\mathbf{q} \in \tilde{\Lambda}^n} \sum_{i=1}^n ( \langle f_1(x_i,\mu) , np_i \rangle + \langle f_2(x_i,\mu,q_i) , n p_i \rangle + l_1(x_i,\mu) + l_2(x_i,\mu,q_i) )\\
        \mathcal{H}^m_n(\mathbf{x},\mu,\mathbf{p}) &=  \frac{1}{n} \inf_{\substack{\mathbf{q} \in \tilde{\Lambda}^n\\ |\mathbf{q}|_{\Lambda^n} \leq m}} \sum_{i=1}^n ( \langle f_1(x_i,\mu) , np_i \rangle + \langle f_2(x_i,\mu,q_i) , n p_i \rangle + l_1(x_i,\mu) + l_2(x_i,\mu,q_i) ).
    \end{align}
    Note that using Assumptions \ref{Assumption_f_sigma_lipschitz}(ii) and \ref{Assumption_running_terminal_cost}(ii), we obtain
    \begin{equation}
        \sum_{i=1}^n ( \langle f_2(x_i,\mu,q_i) , n p_i \rangle + l_2(x_i,\mu,q_i))\geq \sum_{i=1}^n \left ( -C(1+|q_i|_{\Lambda} ) n |p_i| -C_1 + C_2 |q_i|_{\Lambda}^2 \right )\geq -Cn + C|\mathbf{q}|_{\Lambda^n}^2
    \end{equation}
    for all $|\mathbf{p}|_{H^n}\leq \tilde C/\sqrt{n}$. Moreover,
    \begin{equation}
        \inf_{\mathbf{q} \in \tilde{\Lambda}^n} \sum_{i=1}^n ( \langle f_2(x_i,\mu,q_i) , n p_i \rangle + l_2(x_i,\mu,q_i)) \leq C(1+n)
    \end{equation}
    for all $|\mathbf{p}|_{H^n}\leq C/\sqrt{n}$. Thus, $\mathcal{H}_n$ and $\mathcal{H}^m_n$ indeed coincide for $m=K\sqrt{n}$ if $K>0$ is chosen sufficiently large, which concludes the proof.
\end{proof}

Let us introduce $\mathcal{H}^m : H \times \mathcal{P}_2(H) \times H \to \mathbb{R}$ as
\begin{equation}
    \mathcal{H}^m(x,\mu,p) := \inf_{\substack{q \in \tilde{\Lambda}\\ |q|_{\Lambda} \leq m}} \left \{ \langle f(x,\mu,q),p \rangle + l(x,\mu,q) \right \}.
\end{equation}

For $n=1$ in the previous lemma, we obtain the following result.

\begin{corollary}\label{lemma_hamiltonian}
    Let Assumptions \ref{Assumption_f_sigma_lipschitz}(ii) and \ref{Assumption_running_terminal_cost}(ii) be satisfied. Then, for every $C>0$, there is a constant $K>0$ such that
    \begin{equation}
        \mathcal{H}(x,\mu,p) = \mathcal{H}^{K}(x,\mu,p)
    \end{equation}
    for all $\mu\in \mathcal{P}_2(H)$, $x ,p\in H$, $|p|\leq C$.
\end{corollary}

\begin{lemma}\label{lemma:continuity_H}
    Let Assumptions \ref{Assumption_f_sigma_lipschitz}(i)(ii) and \ref{Assumption_running_terminal_cost}(i)(ii) be satisfied. Then, there is a constant $C\geq 0$ such that
    \begin{align}
        | \mathcal{H}(x,\mu,p) - \mathcal{H}(x,\mu,p') | &\leq C ( 1 + |x| + \mathcal{M}_r^{\frac{1}{r}}(\mu) + |p| + |p'| ) |p-p'| \label{hamiltonian_first_inequality}\\
        | \mathcal{H}(x,\mu,p) - \mathcal{H}(y,\beta,p) | &\leq C (|x-y| + d_r(\mu,\beta) ) (1+|p|).\label{hamiltonian_second_inequality}
    \end{align}
    for all $x,y\in H$, $\mu,\beta\in \mathcal{P}_2(H)$ and $p,p'\in H$.
\end{lemma}

\begin{proof}
    Let us first observe that
    \begin{equation}\label{hamiltonian_quadratic_in_p}
        | \mathcal{H}(x,\mu,p) | \leq C(1+|x|+ \mathcal{M}_r^{\frac{1}{r}}(\mu) ) (1+|p|) + C|p|^2.
    \end{equation}
    for all $x\in H$, $\mu\in \mathcal{P}_2(H)$ and $p\in H$. Indeed, by Assumptions \ref{Assumption_f_sigma_lipschitz}(ii) and \ref{Assumption_running_terminal_cost}(ii), we have
    \begin{equation}
    \begin{split}
        \mathcal{H}(x,\mu,p) &\leq | \langle f_1(x,\mu),p\rangle | + |l_1(x,\mu) | + \inf_{q\in \tilde{\Lambda}} \left \{ | \langle f_2(x,\mu,q),p\rangle | + l_2(x,\mu,q) \right \}\\
        &\leq C(1+|x| + \mathcal{M}_r^{\frac{1}{r}}(\mu) ) (1+|p|) + \inf_{q\in \tilde{\Lambda}} \left \{ C(1 + |q|_{\Lambda} ) |p| + C_1 + C_3 |q|_{\Lambda}^2 \right \}\\
        &\leq C(1+|x|+ \mathcal{M}_r^{\frac{1}{r}}(\mu) ) (1+|p|),
    \end{split}
    \end{equation}
    \begin{equation}
    \begin{split}
        -\mathcal{H}(x,\mu,p) &\leq | \langle f_1(x,\mu),p\rangle | + |l_1(x,\mu)| + \sup_{q\in\tilde{\Lambda}} \left \{ | \langle f_2(x,\mu,q),p\rangle | - l_2(x,\mu,q) \right \}\\
        &\leq C(1+|x|+\mathcal{M}_r^{\frac{1}{r}}(\mu) ) (1+ |p|) + \sup_{q\in\tilde{\Lambda}} \left \{ C(1 + |q|_{\Lambda}) |p| + C_1 -C_2 |q|_{\Lambda}^2 \right \}\\
        &\leq C(1+|x|+ \mathcal{M}_r^{\frac{1}{r}}(\mu) ) (1+|p|) + C|p|^2,
    \end{split}
    \end{equation}
    where we used Young's inequality in the last step. This concludes the proof of \eqref{hamiltonian_quadratic_in_p}.
    
    Now, for $x\in H$, $\mu\in \mathcal{P}_2(H)$, and $p\in H$, since $\mathcal{H}$ is concave in its last variable, we have
    \begin{equation}
        - C|p'|^2 - C(1+|x|+\mathcal{M}_r^{\frac{1}{r}}(\mu))(1+ |p'|) \leq \mathcal{H}(x,\mu,p') \leq \mathcal{H}(x,\mu,p) + \langle \xi , p'-p \rangle,
    \end{equation}
    for all $p'\in H$ and any $\xi = \xi(x,\mu,p) \in D^+_p \mathcal{H}(x,\mu,p)$, where $D^+_p \mathcal{H}$ denotes the superdifferential of $\mathcal{H}$ in the last variable, see e.g. \cite[Definition E.1]{fabbri_gozzi_swiech_2017}. Let $p'=p-(1+|p|) \frac{\xi}{|\xi|}$. Note that $|p'| \leq 2 |p|+1$. Thus, using \eqref{hamiltonian_quadratic_in_p}, we obtain
    \begin{equation}
       (1+ |p|) |\xi(x,\mu,p)| \leq C(1+|x|+ \mathcal{M}_r^{\frac{1}{r}}(\mu) )(1+ |p|) + C|p|^2,
    \end{equation}
    from which we deduce that
    \begin{equation}\label{bound_superdifferential}
        | \xi(x,\mu,p) | \leq C(1+|x|+\mathcal{M}_r^{\frac{1}{r}}(\mu) + |p|),
    \end{equation}
    for all $x\in H$, $\mu\in \mathcal{P}_2(\mu)$, and $p\in H$.
    Now, we consider the function $\Psi: \mathbb{R} \to \mathbb{R}$, $\Psi(\theta) := \mathcal{H}(x,\mu,p'+\theta(p-p'))$. Since $\mathcal{H}$ is concave in its last varaible, $\Psi$ is also concave. Thus, $\Psi$ is Lipschitz continuous on $[0,1]$, hence differentiable almost everywhere. Let $\theta\in (0,1)$ be a point of differentiability. Then, we have
    \begin{equation}
        \Psi(\theta + h) - \Psi(\theta) = \mathcal{H}(x,\mu,p'+(\theta+h)(p-p')) + \mathcal{H}(x,\mu,p'+\theta(p-p')) \leq \langle \xi,h(p-p') \rangle.
    \end{equation}
    Hence, for all $\xi = \xi(x,\mu,p'+\theta(p-p')) \in D_p^+ \mathcal{H}(x,\mu,p'+\theta(p-p'))$, we have
    \begin{equation}
        \langle \xi,p-p'\rangle \leq \lim_{h\uparrow 0} \frac{\Psi(\theta_h)-\Psi(\theta)}{h} = \Psi'(\theta) = \lim_{h\downarrow 0} \frac{\Psi(\theta_h)-\Psi(\theta)}{h} \leq \langle \xi, p-p' \rangle, 
    \end{equation}
    i.e., $|\Psi'(\theta)| \leq |\xi(x,\mu,p'+\theta(p-p'))| |p-p'|$. Therefore,
    \begin{equation}
    \begin{split}
        &| \mathcal{H}(x,\mu,p) - \mathcal{H}(x,\mu,p') | = | \Psi(1) - \Psi(0)| = \left | \int_0^1 \Psi'(\theta) \mathrm{d}\theta \right |\\
        &\leq \int_0^1 |\xi(x,\mu,p'+\theta(p-p'))| |p-p'| \mathrm{d}\theta \leq C (1+|x|+\mathcal{M}_r^{\frac{1}{r}}(\mu) + |p| + |p'| ) |p-p'|,
    \end{split}
    \end{equation}
    which concludes the proof of \eqref{hamiltonian_first_inequality}.

    For the proof of inequality \eqref{hamiltonian_second_inequality}, we observe that
    \begin{equation}\label{hamiltonian_second_inequality_proof}
    \begin{split}
        | \mathcal{H}(x,\mu,p) - \mathcal{H}(y,\beta,p) |&\leq | \langle f_1(x,\mu) - f_1(y,\beta), p \rangle | + | l_1(x,\mu) - l_1(y,\beta) |\\
        &\quad + \left | \inf_{q\in \tilde{\Lambda}} \left \{ \langle f_2(x,\mu,q),p \rangle + l_2(x,\mu,q) \right \} - \inf_{q\in \tilde{\Lambda}} \left \{ \langle f_2(y,\beta,q),p \rangle + l_2(y,\beta,q) \right \} \right |\\
        &\leq C ( |x-y| + d_r(\mu,\beta) ) (1+|p|)\\
        &\quad + \sup_{q\in\tilde{\Lambda}} \left \{ | \langle f_2(x,\mu,q) - f_2(y,\beta,q),p \rangle | + |l_2(x,\mu,q) - l_2(y,\beta,q) | \right \}\\
        &\leq C ( |x-y| + d_r(\mu,\beta) ) (1+|p|).
    \end{split}
    \end{equation}
    where we used Assumptions \ref{Assumption_f_sigma_lipschitz}(i) and \ref{Assumption_running_terminal_cost}(i).
\end{proof}
We need the following representation of the lifted Hamiltonian $\tilde{\mathcal{H}}$ defined in \eqref{lifted_hamiltonian_definition}.
\begin{prop}
    Let Assumptions \ref{Assumption_f_sigma_lipschitz}(i)(ii), \ref{Assumption_running_terminal_cost}(i)(ii) be satisfied. Then
    \begin{equation}\label{eq:aaa1}
        \tilde{\mathcal{H}}(X,P) = \inf_{Q\in \mathcal{E}} \left \{ \llangle F(X,Q), P \rrangle + L(X,Q) \right \}
    \end{equation}
    for all $X,P \in E$.
\end{prop}
\begin{proof}
    The inequality ``$\leq$'' in \eqref{eq:aaa1} is obvious, so we only prove the opposite inequality. Given $R>0$, we denote $\Omega_R = \{ \omega\in \Omega : \max(|X(\omega)|,|P(\omega)|) >R \}$. We choose $a_0\in \tilde\Lambda$ and denote $\mu := X_{\texttt{\#}} \mathcal{L}^1$. For $\varepsilon>0$, let $R=R_{\varepsilon}$ be such that
    \begin{equation}
     \left|\int_{\Omega_R} \left ( \langle f(X(\omega),\mu,a_0), P(\omega) \rangle + L(X(\omega),\mu,a_0) \right ) \mathrm{d}\omega\right|+       \int_{\Omega_R} | \mathcal{H} (X(\omega),X_{\texttt{\#}} \mathcal{L}^1,P(\omega)) | \mathrm{d}\omega < \varepsilon.
    \end{equation}
   Given $\delta >0$, we cover $\bar{B}_R \times \bar{B}_R \subset H\times H$ by countably many non-empty, disjoint Borel sets $D_i \subset \bar{B}_R \times \bar{B}_R$, $i\in\mathbb{N}$, with $\text{diam}(D_i) <\delta$, and we choose $(x_i,p_i) \in D_i$ arbitrary. Due to Lemma \ref{lemma:continuity_H} we may choose $\delta$ sufficiently small such that
    \begin{equation}
      | \mathcal{H}(x,\mu,p) - \mathcal{H}(x_i,\mu,p_i) | < \varepsilon
    \end{equation}
    for all $(x,p)\in D_i$. For $i\in\mathbb{N}$, let $a_i\in \tilde{\Lambda}$ be such that
    \begin{equation}
        \mathcal{H}(x_i,\mu,p_i) > \langle f(x_i,\mu,a_i) ,p_i \rangle + l(x_i,\mu,a_i) - \varepsilon.
    \end{equation}
    Since $|x_i|,|p_i| \leq R$, we have $|a_i|_{\Lambda} \leq K$ for some $K>0$, which only depends on $R$ and $\mu$. By Assumptions \ref{Assumption_f_sigma_lipschitz}(i), \ref{Assumption_running_terminal_cost}(i), we can take $\delta>0$ sufficiently small such that
    \begin{equation}
 | \langle f(x,\mu,a_i),p \rangle + l(x,\mu,a_i) - ( \langle f(x_i,\mu,a_i), p_i \rangle + l(x_i,\mu,a_i) ) | < \varepsilon
    \end{equation}
    for all $(x,p)\in D_i$, $i\in\mathbb{N}$. Let $\Omega^i = \{ \omega \in \Omega:(X(\omega),P(\omega)) \in D_i \}$, $i\in \mathbb{N}$, 
    and set
    \begin{equation}
        Q_{\varepsilon}(\omega) = \sum_{i=1}^{\infty} a_i {\mathbf 1}_{\Omega^i}(\omega) + a_0 {\mathbf 1}_{\Omega_R}(\omega).
    \end{equation}
    Then,
    \begin{equation}
    \begin{split}
        \llangle F(X,Q_{\varepsilon}), P \rrangle + L(X,Q_{\varepsilon}) &= \sum_{i=1}^{\infty} \int_{\Omega^i} \left ( \langle f(X(\omega),\mu,a_i), P(\omega) \rangle + L(X(\omega),\mu,a_i) \right ) \mathrm{d}\omega\\
        &\quad + \int_{\Omega_R} \left ( \langle f(X(\omega),\mu,a_0), P(\omega) \rangle + L(X(\omega),\mu,a_0) \right ) \mathrm{d}\omega\\
        &\leq \sum_{i=1}^{\infty} \int_{\Omega^i} \left ( \langle f(x_i,\mu,a_i) , p_i \rangle + l(x_i,\mu,a_i) \right ) \mathrm{d}\omega + 2\varepsilon \\
        &\leq \sum_{i=1}^{\infty} \int_{\Omega^i} \mathcal{H}(x_i,\mu,p_i) \mathrm{d}\omega + 3 \varepsilon \leq \sum_{i=1}^{\infty} \int_{\Omega^i} \mathcal{H}(X(\omega),\mu,P(\omega)) \mathrm{d}\omega + 4\varepsilon\\
        &\leq \int_{\Omega} \mathcal{H}(X(\omega),\mu,P(\omega)) \mathrm{d}\omega + 5 \varepsilon.
    \end{split}
    \end{equation}
   This yields
    \begin{equation}
        \llangle F(X,Q_{\varepsilon}) , P \rrangle + L(X,Q_{\varepsilon}) \leq \tilde{\mathcal{H}}(X,P) + 5\varepsilon
    \end{equation}
    which concludes the proof.
\end{proof}

\subsection{{The Operators \texorpdfstring{$\mathcal{A}$}{\mathcal A}, \texorpdfstring{$e^{t\mathcal{A}}$}{e^{t\mathcal A}}}, \texorpdfstring{${\mathcal{B}}$}{\mathcal B}.  }\label{subsection_operator_A}

Throughout this subsection, we work under Assumptions \ref{Assumption_A_maximal_dissipative} and \ref{Assumption_weak_B_condition}. Recall that $\mathcal{A}: \mathcal{D}(\mathcal{A}) \subset E\to E$ is defined as $\mathcal{A}(X)(\omega) := A(X(\omega))$ with domain $\mathcal{D}(\mathcal{A}) = L^2(\Omega;\mathcal{D}(A))$, where $A:\mathcal{D}(A)\subset H\to H$ is a maximal dissipative operator.
\begin{lemma}
    The operator $\mathcal{A}: \mathcal{D}(\mathcal{A}) \subset E\to E$ is maximal dissipative.
\end{lemma} 
\begin{proof}
The dissipativity of $\mathcal{A}$ follows directly from its definition and the dissipativity of $A$.
Let us show that $\mathcal{A}$ is maximal dissipative. Since $A$ is maximal dissipative, $\mathcal{R}(I-A)=H$. Hence, for all $y\in H$, there is an $x\in \mathcal{D}(A)$ such that $x-Ax = y$. Now, for given $Y\in E$, let $X\in E$ be defined in the following way: Let $X(\omega)\in H$ be such that $X(\omega) - A(X(\omega)) = Y(\omega)$, i.e., $X(\omega) := (I-A)^{-1}(Y(\omega))$. Then we have
\begin{equation}
    Y(\omega) = (I-A)(X(\omega)) = (I-\mathcal{A})(X)(\omega),
\end{equation}
i.e., $Y=(I-\mathcal{A})(X)$. This shows that $\mathcal{R}(I-\mathcal{A}) = E$, hence $\mathcal{A}$ is maximal dissipative.
\end{proof}

It follows that $\mathcal{A}$ generates a $C_0$-semigroup of contractions on $E$, which we denote by $(e^{s\mathcal{A}})_{s\geq 0}$. Note that $(e^{s\mathcal{A}})_{s\geq 0}$ is a semigroup in $E$, and $(e^{sA})_{s\geq 0}$ is a semigroup in $H$. We have the following relation between the two: 
\begin{lemma}\label{lemma:semigroups_coincide}
    For $X\in E$, we have for almost every $\omega\in\Omega$,
\begin{equation}\label{semigroups_coincide}
    (e^{s\mathcal{A}} X)(\omega) = e^{sA}(X(\omega)).
\end{equation}
\end{lemma}
\begin{proof}
    Let $A_{(k)}=kA(kI-A)^{-1}$, $k\in\mathbb{N}$, be the Yosida approximation of $A$, and $\mathcal{A}_{(k)}=k\mathcal{A}(k\mathcal{I}-\mathcal{A})^{-1}$, $k\in\mathbb{N}$, be the Yosida approximation of $\mathcal{A}$.\footnote{We use the $A_{(k)}$ and $\mathcal{A}_{(k)}$ for the Yosida approximation to distinguish it from the $n\times n$ matrix $A_n$ that arises in the HJB equation \eqref{finite_dimensional_hjb}.} Here, $I:H\to H$ and $\mathcal{I}:E\to E$ denote the identity in $H$ and $E$, respectively. First, note that for $X\in E$, we have
    \begin{equation}
        (k\mathcal{I}-\mathcal{A})(X)(\omega) = (k\mathcal{I}X - \mathcal{A}X)(\omega) = kX(\omega) - (\mathcal{A}X)(\omega) = kX(\omega) - A(X(\omega)) = (kI-A)(X(\omega)).
    \end{equation}
    Thus,
    \begin{equation}
        (kI-A)(kI-A)^{-1}(X(\omega)) = X(\omega) = ((k\mathcal{I}-\mathcal{A})(k\mathcal{I}-\mathcal{A})^{-1} )(X)(\omega) = (kI-A)( (k\mathcal{I}-\mathcal{A})^{-1} (X)(\omega)).
    \end{equation}
    Applying $(kI-A)^{-1}$ to both sides yields
    \begin{equation}
        (kI-A)^{-1}(X(\omega))=(k\mathcal{I}-\mathcal{A})^{-1}(X)(\omega).
    \end{equation}
    In particular, this shows that 
    \begin{equation}
        A_{(k)}(X(\omega)) = kA (kI-A)^{-1}(X(\omega)) = kA( (k\mathcal{I}-\mathcal{A})^{-1}(X)(\omega)) = k \mathcal{A}((k\mathcal{I}-\mathcal{A})^{-1}(X))(\omega) = \mathcal{A}_{(k)}(X)(\omega).
    \end{equation}
    Since $\sum_{i=1}^N \frac{s^i \mathcal{A}_{(k)}^i(X)}{i!}$ converges to $e^{s\mathcal{A}_{(k)}}(X)$ in $E$, we also have convergence along some subsequence for almost every $\omega\in\Omega$, i.e.,
    \begin{equation}
        e^{sA_{(k)}}(X(\omega)) = \lim_{j\to\infty} \sum_{i=1}^{N_j} \frac{s^i A_{(k)}^i(X(\omega))}{i!} = \lim_{j\to\infty} \sum_{i=1}^{N_j} \frac{s^i \mathcal{A}_{(k)}^i(X)(\omega)}{i!} = e^{s\mathcal{A}_{(k)}}(X)(\omega).
    \end{equation}
    Moreover, since $e^{s\mathcal{A}_{(k)}}$ converges to $e^{s\mathcal{A}}$ in the strong operator topology as $k\to\infty$, we know that along some subsequence $e^{s\mathcal{A}_{(k_j)}}(X)$, we have convergence for almost every $\omega\in\Omega$, i.e.,
    \begin{equation}
        e^{sA}(X(\omega)) = \lim_{j\to\infty} e^{sA_{(k_j)}}(X(\omega)) = \lim_{j\to\infty} \left ( e^{s\mathcal{A}_{(k_j)}}(X)(\omega) \right ) = e^{s\mathcal{A}}(X)(\omega),
    \end{equation}
    which concludes the proof of \eqref{semigroups_coincide}.
\end{proof}
\begin{lemma}
    The operator $\mathcal{B}\in L(E)$, $\mathcal{B}(X)(\omega) := B(X(\omega))$, satisfies the weak $B$ condition for $\mathcal{A}$.
\end{lemma} 
\begin{proof}
We need to show that $\mathcal{B}$ is strictly positive, self-adjoint, $\mathcal{A}^*\mathcal{B} \in L(E)$ and $-\mathcal{A}^*\mathcal{B} + c_0 \mathcal{B} \geq 0$ for some constant $c_0 \geq 0$. All these properties follow from straightforward computations using the definition of $\mathcal{B}$ and the corresponding properties of $B$. For instance, the strict positivity follows by   $\llangle \mathcal{B}X,X \rrangle = \int_{\Omega} \langle B(X(\omega)), X(\omega) \rangle \mathrm{d}\omega >0.$
\end{proof}

\subsection{Properties of the  coefficients of the lifted problem.}\label{subsec:properties_lifted_coeff} Notice that, under Assumptions \ref{Assumption_f_sigma_lipschitz}, \ref{Assumption_running_terminal_cost}, the functions  $F, \Sigma, L, U_T, \tilde{\mathcal H}$ inherit estimates with respect to $\|\cdot\|, \|\cdot\|_{-1}$ from the corresponding estimates for $f, \sigma, l, \mathcal U_T, \mathcal H$. In particular, we have
\begin{align}
   & \llangle F(X,Q) - F(Y,Q), \mathcal B(X-Y) \rrangle \leq C\|X-Y\|^2_{-1},\\
   &\|\Sigma(X) - \Sigma(Y)\|_{L_2(\Xi,E)} +|L(X,Q) - L(Y,Q) |  \leq C\|X-Y\|_{-1},\\
    &| \tilde{\mathcal{H}}(X,P) - \mathcal{H}(X,P') | \leq C ( 1 + \|X\|  + \|P\| + \|P'\| ) \|P-P'\| ,\\
   &| \tilde{\mathcal{H}}(X,P) - \tilde{\mathcal{H}}(Y,P) | \leq C \|X-Y\|  (1+\|P\|),
\end{align}
for all $X,Y,P,P' \in E$ and $Q \in \mathcal E$.
\subsection{The HJB Equation in the Wasserstein Space} 
We introduce the definition of $L$-viscosity solution of \eqref{intro:HJB_on_Wasserstein_space}.
\begin{definition}
A function $\mathcal{U}:[0,T]\times \mathcal{P}_2(H)\to\mathbb{R}$ is an
$L$-viscosity solution of equation \eqref{intro:HJB_on_Wasserstein_space} if its lift $U:[0,T]\times E\to\mathbb{R}$ defined by $U(t,X):= \mathcal{U}(t,X_{\texttt{\#}}\mathcal{L}^1)$ is a $\mathcal B$-continuous viscosity solution of equation \eqref{lifted_HJB_equation}, see  Definition \ref{def:viscosity_solution_hilbert} (see also \cite[Section 3.3]{fabbri_gozzi_swiech_2017}).
\end{definition}
\section{Convergence of the Value Functions}\label{sec:convergence}
\subsection{Estimates for the Finite Particle System}

Let $n\geq 1$. For $\mathbf{x}=(x_1,\dots,x_n) \in H^n$, and $\mathbf{a}(\cdot) = (a_1(\cdot),\dots,a_n(\cdot))\in \Lambda^n_t$, let $\mathbf{x}(\cdot) = (x_1(\cdot),\dots, x_n(\cdot))$ denote the solution of the equation 
\begin{equation}\label{state_equation_vector}
\begin{cases}
	\mathrm{d}\mathbf{x}(s) = [\mathbf{A} \mathbf{x}(s) +\mathbf{f}(\mathbf{x}(s),\mu_{{\bf x}(s)},\mathbf{a}(s))]\mathrm{d}s + \bm{\sigma}(\mathbf{x}(s),\mu_{{\bf x}(s)})\mathrm{d}W(s), \quad s \in [t,T]\\
    \mathbf{x}(t) = \mathbf{x} \in H^n,
\end{cases}
\end{equation}
where $\mathbf{A}$ denotes the $n\times n$ diagonal matrix with the operator $A$ on its diagonal, $\mathbf{f}(\mathbf{x},\mu_{\mathbf{x}}, \mathbf{a}) = (f(x_1,\mu_{\mathbf{x}},a_1),\dots, f(x_n,\mu_{\mathbf{x}},a_n))$, and $\bm{\sigma}(\mathbf{x},\mu_{\mathbf{x}}) = (\sigma(x_1,\mu_{\mathbf{x}}),\dots, \sigma(x_n,\mu_{\mathbf{x}}))$.

\begin{prop}
    Let Assumptions \ref{Assumption_A_maximal_dissipative} and \ref{Assumption_f_sigma_lipschitz}(i)(ii)(iv) be satisfied. Then equation \eqref{state_equation_vector} has a unique mild solution $\mathbf{x}(\cdot) \in L^2([t,T]\times \Omega'; H^n)$ in the sense of \cite[Definition 1.119]{fabbri_gozzi_swiech_2017}, which is progressively measurable and has continuous trajectories. The components $(x_1(\cdot),\dots,x_n(\cdot))$ of $\mathbf{x}(\cdot)$ are mild solutions of the system of SDEs \eqref{state_equation}. Moreover, if $\mathbf{a}(\cdot)$ is bounded, then
   $ \mathbb{E} \left [ \sup_{s\in [t,T]} |\mathbf{x}(s) |_{H^n}^k\right ] <\infty$ for every $k\geq 1$.
\end{prop}

\begin{proof}
See \cite[Theorem 6.5, page 162]{PLChow}. We remark that since $A$ generates a semigroup of contractions, continuity of paths follows from the continuity of paths of the stochastic convolution, see \cite[Theorem 6.2, page 159]{PLChow} or \cite[Theorem 1.112]{fabbri_gozzi_swiech_2017}. For the moment estimate when $\mathbf{a}(\cdot)$ is bounded we refer for instance to \cite[Theorem 1.130]{fabbri_gozzi_swiech_2017}.
\end{proof}

Throughout this work, solutions of SDEs are always understood in the sense of mild solutions, see e.g. \cite[Definition 1.119]{fabbri_gozzi_swiech_2017}. 
{For more on the theory of mild solutions of stochastic differential equations in Hilbert spaces we refer the readers to \cite{daprato_zabczyk_2014,gawarecki_mandrekar_2011}.}

\begin{remark}\label{rem:Itoformulas}
In the paper we will use two versions of It\^o's formula, \cite[Proposition 1.164]{fabbri_gozzi_swiech_2017} and \cite[Proposition 1.166]{fabbri_gozzi_swiech_2017}. We remark that it is clear from the proof of \cite[Proposition 1.164]{fabbri_gozzi_swiech_2017}, together with \cite[Theorem 1.112 and Proposition 1.132]{fabbri_gozzi_swiech_2017}, that if $A$ generates a semigroup of contractions then It\^o's formula of \cite[Proposition 1.164]{fabbri_gozzi_swiech_2017} also holds for $p=2$ there and hence it can be applied in our case. Moreover, even though \cite[Proposition 1.166]{fabbri_gozzi_swiech_2017} is stated for coefficients $b,\sigma$ there which are bounded in the control variable, it is clear from its proof that it can be applied in our case for functions with appropriate growth bounds and stopping times which guarantee that all terms are well defined. We will use it for functions $F(t,x)=|x|^2$ in $H$ and to $F(t,X)=\|X\|^2$ in $E$. We leave the simple proofs of such modifications to the reader.
\end{remark}

\begin{prop}\label{Prop_Finite_A_Priori}
	Let Assumptions \ref{Assumption_A_maximal_dissipative}, \ref{Assumption_weak_B_condition}, and \ref{Assumption_f_sigma_lipschitz}(i)(ii)(iv) be satisfied. Let $\mathbf{x}(\cdot),\mathbf{x}^0(\cdot),\mathbf{x}^1(\cdot)$ be the solutions of equation \eqref{state_equation} with initial conditions $\mathbf{x},\mathbf{x}^0,\mathbf{x}^1\in H^n$, respectively, and control $\mathbf{a}(\cdot)\in \Lambda^n_t$. Then, there is a constant $C\geq 0$, independent of $n\in\mathbb{N}$, such that  
    \begin{align}
        &\mathbb{E} \left [ \sup_{s\in [t,T]} |\mathbf{x}(s) |_{r} \right ] \leq C \left ( 1 + | \mathbf{x}|_{r} + \frac{1}{\sqrt{n}} \| \mathbf{a}(\cdot) \|_{M^2(t,T;\Lambda^n)} \right ),\label{eq:est_|x|r}\\
		&\mathbb{E} \left [ \sup_{s^{\prime} \in [t,s]} |\mathbf{x}(s^{\prime}) - \mathbf{x} |_{-1,r} \right ] \leq \frac{C}{\sqrt{n}} \mathbb{E} \left [ \int_t^s \sum_{i=1}^n |a_i(s^{\prime})|_{\Lambda}^2 \mathrm{d}s^{\prime} \right ]^{\frac12} + C \left (1+ |\mathbf{x}|_{r} \right )(s-t)^{\frac12},\label{eq:est_|xs-x|r}\\
        &\mathbb{E} \left [ \sup_{s\in[t,T]} | \mathbf{x}^1(s) - \mathbf{x}^0(s)|_{-1,r} \right ]\leq C| \mathbf{x}^1 - \mathbf{x}^0|_{-1,r}\label{eq:est_|x1-x0|r},
    \end{align}
    for all $s\in [t,T]$, $\mathbf{x}, \mathbf{x}^1, \mathbf{x}^0 \in H^n$, and $\mathbf{a}(\cdot)\in \Lambda^n_t$.
\end{prop}

\begin{proof}
    Proof of \eqref{eq:est_|x|r}: First, note that
    \begin{equation}\label{estimate_r_norm}
        \mathbb{E} \left [ \sup_{s\in [t,T]} | \mathbf{x}(s)|_{r} \right ] \leq \left ( \frac1n \sum_{i=1}^n \mathbb{E} \left [ \sup_{s\in [t,T]} |x_i(s)|^r \right ] \right )^{\frac1r}.
    \end{equation}
    We are going to estimate the right-hand side. For $m\geq 1$, we denote by $\tau_m$ the minimum of $T$ and the exit time of $\mathbf{x}(s)$ from $\{|\mathbf{x}|\leq m\}$. Using \cite[Proposition 1.166]{fabbri_gozzi_swiech_2017} (see Remark \ref{rem:Itoformulas}), we have for every $s'\in [t,s]$
    \begin{equation}
    \begin{split}
        |x_i(s'\wedge\tau_m) |^2 &\leq |x_i|^2 + 2 \int_t^{s'\wedge\tau_m} | \langle f(x_i(t'),\mu_{\mathbf{x}(t')},a_i(t')), x_i(t') \rangle | \mathrm{d}t'\\
        &\quad + \int_t^{s'\wedge\tau_m} | \sigma(x_i(t'),\mu_{\mathbf{x}(t')}) |_{L_2(\Xi,H)}^2 \mathrm{d}t' + 2 \left | \int_t^{s'} \langle {\mathbf{1}}_{[t,\tau_m]}(t')x_i(t'), \sigma(x_i(t') , \mu_{\mathbf{x}(t')}) \mathrm{d}W(t') \rangle \right |.
    \end{split}
    \end{equation}
    Taking the power $r/2$ on both sides and the supremum over $s'\in [t,s]$, we obtain
    \begin{equation}\label{estimate_linear_growth}
    \begin{split}
        &\sup_{s'\in [t,s]} |x_i(s'\wedge\tau_m) |^r\\
        &\leq C |x_i|^r + C \left ( \int_t^s | \langle f(x_i(s'),\mu_{\mathbf{x}(s')},a_i(s')), x_i(s') \rangle | \mathrm{d}s' \right )^{\frac{r}{2}} \\
        & + C \left ( \int_t^s | \sigma(x_i(s'),\mu_{\mathbf{x}(s')}) |_{L_2(\Xi,H)}^2 \mathrm{d}s' \right )^{\frac{r}{2}} + C \sup_{s'\in [t,s]} \left | \int_t^{s'} \langle {\mathbf{1}}_{[t,\tau_m]}(t')x_i(t'), \sigma(x_i(t') , \mu_{\mathbf{x}(t')}) \mathrm{d}W(t') \rangle \right |^{\frac{r}{2}}.
    \end{split}
    \end{equation}
    For the second term on the right-hand side, by Assumption \ref{Assumption_f_sigma_lipschitz}(ii), we have
    \begin{equation}
    \begin{split}
        &\left ( \int_t^s | \langle f(x_i(s'),\mu_{\mathbf{x}(s')},a_i(s')), x_i(s') \rangle | \mathrm{d}s' \right )^{\frac{r}{2}}\\
        &\leq \left ( \int_t^s | f(x_i(s'),\mu_{\mathbf{x}(s')},a_i(s')) | |x_i(s') | \mathrm{d}s' \right )^{\frac{r}{2}} \\
        &\leq C \sup_{s'\in [t,s]} |x_i(s') |^{\frac{r}{2}} \left ( \int_t^s \left ( 1 + |x_i(s')| + \mathcal{M}_r^{\frac1r}( \mu_{\mathbf{x}(s')}) + | a_i(s') |_{\Lambda} \right ) \mathrm{d}s' \right )^{\frac{r}{2}}\\
        &\leq \varepsilon \sup_{s'\in [t,s]} |x_i(s') |^r + C_{\varepsilon} \int_t^s \left ( 1 + |x_i(s')|^r + \mathcal{M}_r( \mu_{\mathbf{x}(s')}) + | a_i(s') |^r_{\Lambda} \right ) \mathrm{d}s',
    \end{split}
    \end{equation}
    for all $\varepsilon>0$, where $C_{\varepsilon}\geq 0$ is a constant that depends on $\varepsilon$. For the third term on the right-hand side of inequality \eqref{estimate_linear_growth}, we have by Assumption \ref{Assumption_f_sigma_lipschitz}(iv)
    \begin{equation}
        \left ( \int_t^s | \sigma(x_i(s'),\mu_{\mathbf{x}(s')}) |_{L_2(\Xi,H)}^2 \mathrm{d}s' \right )^{\frac{r}{2}} \leq C \left ( 1 + \left ( \int_t^s |x_i(s')|^2 \mathrm{d}s' \right )^{\frac{r}{2}} + \left ( \int_t^s \mathcal{M}^{\frac{2}{r}}_r( \mu_{\mathbf{x}(s')}) \mathrm{d}s' \right )^{\frac{r}{2}} \right ).
    \end{equation}
    Note that
    \begin{equation}
        \left ( \int_t^s |x_i(s')|^2 \mathrm{d}s' \right )^{\frac{r}{2}} \leq \sup_{s'\in [t,s]} |x_i(s') |^{\frac{r}{2}} \left ( \int_t^s |x_i(s')| \mathrm{d}s' \right )^{\frac{r}{2}}\leq \varepsilon \sup_{s'\in [t,s]} |x_i(s') |^r + C_{\varepsilon} \int_t^s |x_i(s')|^r \mathrm{d}s',
    \end{equation}
    for all $\varepsilon>0$, as well as
    \begin{equation}
    \begin{split}
        \left ( \int_t^s \mathcal{M}^{\frac{2}{r}}_r( \mu_{\mathbf{x}(s')}) \mathrm{d}s' \right )^{\frac{r}{2}} &= \left ( \int_t^s \left ( \frac{1}{n} \sum_{i=1}^n |x_i(s')|^r \right )^{\frac{2}{r}} \mathrm{d}s' \right )^{\frac{r}{2}} \\
        &\leq \sup_{s'\in [t,s]} \left ( \frac{1}{n} \sum_{i=1}^n |x_i(s')|^r \right )^{\frac{1}{2}} \left ( \int_t^s \left ( \frac{1}{n} \sum_{i=1}^n |x_i(s')|^r \right )^{\frac{1}{r}} \mathrm{d}s' \right )^{\frac{r}{2}}\\
        &\leq \frac{\varepsilon}{n} \sum_{i=1}^n \sup_{s'\in [t,s]} |x_i(s')|^r + C_{\varepsilon} \int_t^s \left ( \frac{1}{n} \sum_{i=1}^n |x_i(s')|^r \right ) \mathrm{d}s',
    \end{split}
    \end{equation}
    for all $\varepsilon>0$. Moreover, regarding the fourth term on the right-hand side of inequality \eqref{estimate_linear_growth}, by Burkholder--Davis--Gundy inequality, we have
    \begin{equation}
    \begin{split}
        &\mathbb{E} \left [ \sup_{s'\in [t,s]} \left | \int_t^{s'} \langle {\mathbf{1}}_{[t,\tau_m]}(t') x_i(t'), \sigma(x_i(t') , \mu_{\mathbf{x}(t')}) \mathrm{d}W(t') \rangle \right |^{\frac{r}{2}} \right ]\\
        &\leq C \mathbb{E} \left [ \left ( \int_t^s {{\mathbf{1}}_{[t,\tau_m]}(s')} |x_i(s')|^2 |\sigma(x_i(s'),\mu_{\mathbf{x}(s')}) |_{L_2(\Xi,H)}^2 \mathrm{d}s' \right )^{\frac{r}{4}} \right ]\\
        &\leq C \mathbb{E} \left [ \sup_{s'\in [t,s]} \left \{ {\mathbf{1}}_{[t,\tau_m]}(s') |x_i(s')|^{\frac{r}{2}} \right \} \left ( \int_t^s  \left ( 1 + |x_i(s')|^2 + \mathcal{M}_r^{\frac{2}{r}} (\mu_{\mathbf{x}(s')}) \right )  \mathrm{d}s' \right )^{\frac{r}{4}} \right ]\\
        &\leq \varepsilon \mathbb{E} \left [ \sup_{s'\in [t,s]} |x_i(s')|^r \right ] + C_{\varepsilon} \mathbb{E} \left [ \left ( \int_t^s  \left ( 1 + |x_i(s')|^2 + \mathcal{M}_r^{\frac{2}{r}} (\mu_{\mathbf{x}(s')}) \right )  \mathrm{d}s' \right )^{\frac{r}{2}} \right ].
    \end{split}
    \end{equation}
    The second term on the right-hand side of this inequality can be estimated using the same arguments as before. Thus, choosing $\varepsilon>0$ sufficiently small, taking the expectation in equation \eqref{estimate_linear_growth}, letting $m\to\infty$, and then summing over $i=1,\dots,n$ and dividing by $n$, we obtain
    \begin{equation}
    \begin{split}
        &\frac1n \sum_{i=1}^n \mathbb{E} \left [ \sup_{s'\in [t,s]} |x_i(s')|^r \right ]\\
        &\leq C \left ( 1 + | \mathbf{x} |_r^r + \int_t^s \left ( \frac{1}{n} \sum_{i=1}^n \mathbb{E} \left [ \sup_{t'\in [t,s']} | x_i(t')|^r \right ] \right ) \mathrm{d}s' + \int_t^s \mathbb{E} \left [ \frac{1}{n} \sum_{i=1}^n | a_i(s') |_{\Lambda}^r \right ] \mathrm{d}s' \right ).
    \end{split}
    \end{equation}
    Therefore, by Gr\"onwall's inequality, we have
    \begin{equation}\label{intermediate_result}
        \frac1n \sum_{i=1}^n \mathbb{E} \left [ \sup_{s'\in [t,s]} |x_i(s')|^r \right ] \leq C \left ( 1 + | \mathbf{x} |_r^r + \int_t^s \mathbb{E} \left [ \frac{1}{n} \sum_{i=1}^n | a_i(s') |_{\Lambda}^r \right ] \mathrm{d}s' \right ).
    \end{equation}
    Taking the power $1/r$ on both sides and noting that
    \begin{equation}\label{estimate_for_a}
        \left ( \int_t^s \mathbb{E} \left [ \frac{1}{n} \sum_{i=1}^n | a_i(s') |_{\Lambda}^r \right ] \mathrm{d}s' \right )^{\frac{1}{r}} \leq C \left ( \int_t^s \mathbb{E} \left [ \frac{1}{n} \sum_{i=1}^n | a_i(s') |_{\Lambda}^2 \right ] \mathrm{d}s' \right )^{\frac{1}{2}}
    \end{equation}
    concludes the proof.
    
    In the remaining proofs we will omit the technicalities involving the stopping times and assume that all the terms are well defined and have sufficient integrability to apply the necessary theorems.
    
    Proof of \eqref{eq:est_|xs-x|r}: Applying It\^o's formula \cite[Proposition 1.164]{fabbri_gozzi_swiech_2017}, and taking the supremum over $s'\in [t,s]$ we obtain
    \begin{equation}\label{estimate_time_regularity}
    \begin{split}
        &\sup_{s'\in [t,s]} |x_i(s') - x_i |_{-1}^2\\
        &\leq 2 \int_t^{s} \left | \langle A^*B(x_i(s')-x_i),x_i(s') \rangle + \langle B(x_i(s')-x_i), f(x_i(s'),\mu_{\mathbf{x}(s')},a_i(s')) \rangle \right | \mathrm{d}s'\\
        &\quad + \int_t^{s} | B^{\frac12} \sigma(x_i(s'),\mu_{\mathbf{x}(s')}) |_{L_2(\Xi,H)}^2 \mathrm{d}s' + 2 \sup_{s'\in [t,s]} \left | \int_t^{s'} \langle B (x_i(t')-x_i), \sigma(x_i(t') , \mu_{\mathbf{x}(t')}) \mathrm{d}W(t') \rangle \right |.
    \end{split}
    \end{equation}
    For the first term on the right-hand side of this inequality, by Assumption \ref{Assumption_weak_B_condition}, we have
    \begin{equation}
        \int_t^{s} \left | \langle A^*B(x_i(s')-x_i),x_i(s') \rangle \right | \mathrm{d}s' \leq C \left ( \sup_{s'\in [t,s]} | x_i(s') |^2 + | x_i|^2 \right ) (s-t).
    \end{equation}
    For the second term on the right-hand side of inequality \eqref{estimate_time_regularity}, we have
    \begin{equation}
    \begin{split}
        &\int_t^{s} \left | \langle B(x_i(s')-x_i), f(x_i(s'),\mu_{\mathbf{x}(s')},a_i(s')) \rangle \right | \mathrm{d}s' \\
        &\leq C \left ( \sup_{s'\in [t,s]} | x_i(s') |^2 + |x_i|^2 \right ) (s-t) + \int_t^{s} | f(x_i(s'),\mu_{\mathbf{x}(s')},a_i(s')) |^2 \mathrm{d}s'.
    \end{split}
    \end{equation}
    Note that by Assumption \ref{Assumption_f_sigma_lipschitz}(i)(ii)
    \begin{equation}
    \begin{split}
        &\int_t^{s} | f(x_i(s'),\mu_{\mathbf{x}(s')},a_i(s')) |^2 \mathrm{d}s'\\
        &\leq C \int_t^{s} \left ( 1 + | x_i(s') |^2 + \mathcal{M}_r^{\frac{2}{r}}( \mu_{\mathbf{x}(s')} ) + |a_i(s') |_{\Lambda}^2 \right ) \mathrm{d}s'\\
        &\leq C \left ( 1 + \sup_{s'\in [t,s]} |x_i(s') |^2 + \sup_{s'\in [t,s]} \mathcal{M}_r^{\frac{2}{r}}( \mu_{\mathbf{x}(s')} ) \right ) (s-t) + C \int_t^{s} |a_i(s') |_{\Lambda}^2 \mathrm{d}s'.
    \end{split}
    \end{equation}
    For the third term on the right-hand side of inequality \eqref{estimate_time_regularity}, we have by Assumption \ref{Assumption_f_sigma_lipschitz}(iv)
    \begin{equation}
    \begin{split}
        \int_t^{s} | B^{\frac12} \sigma(x_i(s'),\mu_{\mathbf{x}(s')}) |_{L_2(\Xi,H)}^2 \mathrm{d}s' &\leq C \int_t^{s} \left ( 1 + | x_i(s') |^2 + \mathcal{M}_r^{\frac{2}{r}}( \mu_{\mathbf{x}(s')} ) \right ) \mathrm{d}s'\\
        &\leq C \left ( 1 + \sup_{s'\in [t,s]} |x_i(s') |^2 + \sup_{s'\in [t,s]} \mathcal{M}_r^{\frac{2}{r}}( \mu_{\mathbf{x}(s')} ) \right ) (s-t).
    \end{split}
    \end{equation}
    For the stochastic integral in inequality \eqref{estimate_time_regularity}, we obtain using Burkholder--Davis--Gundy inequality and Assumption \ref{Assumption_f_sigma_lipschitz}(iv)
    \begin{equation}
    \begin{split}
        &\mathbb{E} \left [ \sup_{s'\in [t,s]} \left | \int_t^{s'} \langle B (x_i(t')-x_i), \sigma(x_i(t') , \mu_{\mathbf{x}(t')}) \mathrm{d}W(t') \rangle \right |^{\frac{r}{2}} \right ]\\
        &\leq C \mathbb{E} \left [ \left ( \int_t^s | \sigma(x_i(s'),\mu_{\mathbf{x}(s')}) |_{L_2(\Xi,H)}^2 | B(x_i(s')-x_i) |^2 \mathrm{d}s' \right )^{\frac{r}{4}} \right ]\\
        &\leq C \mathbb{E} \left [ \left ( \int_t^s \left ( 1 + |x_i(s')|^2 + \mathcal{M}_{r}^{\frac{2}{r}}(\mu_{\mathbf{x}(s')}) \right ) | x_i(s')-x_i |_{-1}^2 \mathrm{d}s' \right )^{\frac{r}{4}} \right ]\\
        &\leq \frac14 \mathbb{E}\left [ \sup_{s'\in [t,s]} | x_i(s') - x_i |_{-1}^r \right ] + C \mathbb{E} \left [ \left ( \int_t^s \left ( 1 + |x_i(s')|^2 + \mathcal{M}_{r}^{\frac{2}{r}}(\mu_{\mathbf{x}(s')}) \right ) \mathrm{d}s' \right )^{\frac{r}{2}} \right ]\\
        &\leq \frac14 \mathbb{E}\left [ \sup_{s'\in [t,s]} | x_i(s') - x_i |_{-1}^r \right ] + C \left ( 1 + \sup_{s'\in [t,s]} |x_i(s')|^r + \sup_{s'\in [t,s]} \mathcal{M}_{r}(\mu_{\mathbf{x}(s')}) \right ) (s-t)^{\frac{r}{2}}.
    \end{split}
    \end{equation}
    Finally, we note that
    $  \sup_{s'\in [t,s]} \mathcal{M}_{r}(\mu_{\mathbf{x}(s')}) \leq \frac{1}{n} \sum_{i=1}^n \sup_{s'\in [t,s]} |x_i(s')|^r.$
    Therefore, taking the power $r/2$ in inequality \eqref{estimate_time_regularity}, taking the expectation, the sum over $i=1,\dots, n$, and dividing by $n$, we obtain
    \begin{equation}
    \begin{split}
        &\frac{1}{n} \sum_{i=1}^n \mathbb{E} \left [ \sup_{s'\in [t,s]} | x_i(s') - x_i |_{-1}^r \right ]\\
        &\leq C \left ( 1 + |\mathbf{x}|_r^r + \frac{1}{n} \sum_{i=1}^n \mathbb{E} \left [ \sup_{s'\in [t,s]} |x_i(s') |^r \right ] \right ) (s-t)^{\frac{r}{2}} + { C \mathbb{E} \left [ \int_t^s \frac{1}{n} \sum_{i=1}^n | a_i(s') |_{\Lambda}^{2} \mathrm{d}s' \right ]^{\frac{r}{2}} }.
    \end{split}
    \end{equation}
    Applying inequality \eqref{intermediate_result}, we obtain
    \begin{equation}
        \frac{1}{n} \sum_{i=1}^n \mathbb{E} \left [ \sup_{s'\in [t,s]} | x_i(s') - x_i |_{-1}^r \right ] \leq C \left ( 1 + |\mathbf{x}|_r^r \right ) (s-t)^{\frac{r}{2}} + { C \mathbb{E} \left [ \int_t^s \frac{1}{n} \sum_{i=1}^n | a_i(s') |_{\Lambda}^{2} \mathrm{d}s' \right ]^{\frac{r}{2}} }.
    \end{equation}
    Taking the power $1/r$ on both sides and using a similar estimate to \eqref{estimate_r_norm} concludes the proof.

    Proof of \eqref{eq:est_|x1-x0|r}: Applying It\^o's formula \cite[Proposition 1.164]{fabbri_gozzi_swiech_2017}, we obtain for every $s'\in [t,s]$
    \begin{equation}\label{ito_1}
    \begin{split}
        &|x^1_i(s')-x_i^0(s')|_{-1}^2\\
        &= |x_i^1 - x_i^0|_{-1}^2 + 2 \int_t^{s'} \langle A^*B (x^1_i(t')-x^0_i(t')), x^1_i(t')-x^0_i(t') \rangle \mathrm{d}t'\\
        &\quad + 2 \int_t^{s'} \langle f(x_i^1(t'),\mu_{\mathbf{x}^1(t')},a_i(t')) - f(x_i^0(t'),\mu_{\mathbf{x}^0(t')},a_i(t')), B( x^1_i(t')-x^0_i(t') ) \rangle \mathrm{d}t'\\
        &\quad + \int_t^{s'} | B^{\frac12} ( \sigma(x_i^1(t'),\mu_{\mathbf{x}^1(t')}) - \sigma(x_i^0(t'),\mu_{\mathbf{x}^0(t')}) ) |_{L_2(\Xi,H)}^2 \mathrm{d}t'\\
        &\quad + 2 \int_t^{s'} \langle B(x_i^1(t') - x_i^0(t')), ( \sigma(x_i^1(t'),\mu_{\mathbf{x}^1(t')}) - \sigma(x_i^0(t'),\mu_{\mathbf{x}^0(t')}) ) \mathrm{d}W(t') \rangle.
    \end{split}
    \end{equation}
    First, note that by Assumption \ref{Assumption_weak_B_condition}, we have
    \begin{equation}
        \int_t^{s'} \langle A^*B (x^1_i(t')-x^0_i(t')), x^1_i(t')-x^0_i(t') \rangle \mathrm{d}t' \leq c_0 \int_t^{s'} |x^1_i(t')-x^0_i(t')|_{-1}^2 \mathrm{d}t'.
    \end{equation}
    Moreover, by Assumption \ref{Assumption_f_sigma_lipschitz}(ii), we have
    \begin{equation}
    \begin{split}
        &\int_t^{s'} \langle f(x_i^1(t'),\mu_{\mathbf{x}^1(t')},a_i(t')) - f(x_i^0(t'),\mu_{\mathbf{x}^0(t')},a_i(t')), B( x^1_i(t')-x^0_i(t') ) \rangle \mathrm{d}t'\\
        &\leq C \int_t^{s'} \left ( |x_i^1(t')-x_i^0(t')|_{-1}^2 + d_{-1,r}^2(\mu_{\mathbf{x}^1(t')},\mu_{\mathbf{x}^0(t')}) \right ) \mathrm{d}t'.
    \end{split}
    \end{equation}
    
    Thus, taking the power $r/2$ on both sides of equation \eqref{ito_1} and taking the supremum over $s'\in [t,s]$, we obtain
    \begin{equation}\label{inequality_1}
    \begin{split}
        &\sup_{s'\in [t,s]} |x^1_i(s')-x_i^0(s')|_{-1}^r\\
        &\leq C |x_i^1 - x_i^0|_{-1}^r + C \left ( \int_t^{s} |x^1_i(s')-x^0_i(s')|_{-1}^2 \mathrm{d}s' \right )^{\frac{r}{2}} + C \left ( \int_t^{s} d_{-1,r}^2(\mu_{\mathbf{x}^1(s')},\mu_{\mathbf{x}^0(s')}) \mathrm{d}s' \right )^{\frac{r}{2}}\\
        &\quad + C \left ( \int_t^{s} | B^{\frac12} ( \sigma(x_i^1(s'),\mu_{\mathbf{x}^1(s')}) - \sigma(x_i^0(s'),\mu_{\mathbf{x}^0(s')}) ) |_{L_2(\Xi,H)}^2 \mathrm{d}s' \right )^{\frac{r}{2}}\\
        &\quad + C \sup_{s'\in [t,s]} \left | \int_t^{s'} \langle B(x_i^1(t') - x_i^0(t')), ( \sigma(x_i^1(t'),\mu_{\mathbf{x}^1(t')}) - \sigma(x_i^0(t'),\mu_{\mathbf{x}^0(t')}) ) \mathrm{d}W(t') \rangle \right |^{\frac{r}{2}}.
    \end{split}
    \end{equation}
    For the second term on the right-hand side of this inequality, we have
    \begin{equation}
    \begin{split}
        \left ( \int_t^{s} |x^1_i(s')-x^0_i(s')|_{-1}^2 \mathrm{d}s' \right )^{\frac{r}{2}} &\leq \sup_{s' \in [t,s]} |x^1_i(s')-x^0_i(s')|_{-1}^{\frac{r}{2}} \left ( \int_t^{s} |x^1_i(s')-x^0_i(s')|_{-1} \mathrm{d}s' \right )^{\frac{r}{2}}\\
        &\leq \varepsilon \sup_{s'\in [t,s]} |x^1_i(s')-x^0_i(s')|_{-1}^{r} + C_{\varepsilon} \int_t^{s} \sup_{t'\in [t,s']} |x^1_i(t')-x^0_i(t')|^r_{-1} \mathrm{d}s',
    \end{split}
    \end{equation}
    for all $\varepsilon>0$. For the third term on the right-hand side of inequality \eqref{inequality_1}, we note that
    \begin{equation}
    \begin{split}
        &\mathbb{E} \left [ \left ( \int_t^s d^2_{-1,r}(\mu_{\mathbf{x}^1(s')},\mu_{\mathbf{x}^0(s')}) \mathrm{d}s' \right )^{\frac{r}{2}} \right ]\\
        &\leq \mathbb{E} \left [ \left ( \int_t^s | \mathbf{x}^1(s') - \mathbf{x}^0(s') |_{-1,r}^2 \mathrm{d}s' \right )^{\frac{r}{2}} \right ]\\
        &\leq \mathbb{E} \left [ \sup_{s'\in [t,s]} | \mathbf{x}^1(s') - \mathbf{x}^0(s') |_{-1,r}^{\frac{r}{2}} \left ( \int_t^s | \mathbf{x}^1(s') - \mathbf{x}^0(s') |_{-1,r} \mathrm{d}s' \right )^{\frac{r}{2}} \right ]\\
        &\leq \varepsilon \mathbb{E} \left [ \sup_{s'\in [t,s]} | \mathbf{x}^1(s') - \mathbf{x}^0(s') |_{-1,r}^{r} \right ] + C_{\varepsilon} \mathbb{E} \left [ \int_t^s \sup_{t'\in [t,s']} | \mathbf{x}^1(t') - \mathbf{x}^0(t') |^r_{-1,r} \mathrm{d}s' \right ]\\
        &\leq \frac{\varepsilon}{n} \sum_{i=1}^n \mathbb{E} \left [ \sup_{s'\in [t,s]} | x_i^1(s') - x_i^0(s') |_{-1}^{r} \right ] + C_{\varepsilon} \int_t^s \frac{1}{n} \sum_{i=1}^n \mathbb{E} \left [ \sup_{t'\in [t,s']} | x_i^1(t') - x_i^0(t') |^r_{-1} \right ] \mathrm{d}s',
    \end{split}
    \end{equation}
    for all $\varepsilon>0$. By Assumption \ref{Assumption_f_sigma_lipschitz}(iv), we can estimate the fourth term on the right-hand side of inequality \eqref{inequality_1} by
    \begin{equation}
    \begin{split}
        &\left ( \int_t^{s} | B^{\frac12} ( \sigma(x_i^1(s'),\mu_{\mathbf{x}^1(s')}) - \sigma(x_i^0(s'),\mu_{\mathbf{x}^0(s')}) ) |_{L_2(\Xi,H)}^2 \mathrm{d}s' \right )^{\frac{r}{2}}\\
        &\leq C \left ( \int_t^{s} \left ( | x_i^1(s') - x_i^0(s') |_{-1}^2 + d^2_{-1,r}(\mu_{\mathbf{x}^1(s')},\mu_{\mathbf{x}^0(s')}) \right ) \mathrm{d}s' \right )^{\frac{r}{2}}.
    \end{split}
    \end{equation}
    Moreover, by Burkholder--Davis--Gundy inequality, we obtain for the stochastic integral in \eqref{inequality_1}
    \begin{equation}
    \begin{split}
        &\mathbb{E} \left [ \sup_{s'\in [t,s]} \left | \int_t^{s'} \langle B(x_i^1(t') - x_i^0(t')), ( \sigma(x_i^1(t'),\mu_{\mathbf{x}^1(t')}) - \sigma(x_i^0(t'),\mu_{\mathbf{x}^0(t')}) ) \mathrm{d}W(t') \rangle \right |^{\frac{r}{2}} \right ]\\
        &\leq C \mathbb{E} \left [ \left ( \int_t^s | x_i^1(s') - x_i^0(s') |^2_{-1} \left ( | x_i^1(s') - x_i^0(s') |^2_{-1} + d^2_{-1,r}(\mu_{\mathbf{x}^1(s')},\mu_{\mathbf{x}^0(s')}) \right ) \mathrm{d}s' \right )^{\frac{r}{4}} \right ]\\
        &\leq C \mathbb{E} \left [ \sup_{s'\in [t,s]} | x_i^1(s') - x_i^0(s') |^{\frac{r}{2}}_{-1} \left ( \int_t^s \left ( | x_i^1(s') - x_i^0(s') |^2_{-1} + d^2_{-1,r}(\mu_{\mathbf{x}^1(s')},\mu_{\mathbf{x}^0(s')}) \right ) \mathrm{d}s' \right )^{\frac{r}{4}} \right ]\\
        &\leq \varepsilon \mathbb{E} \left [ \sup_{s'\in [t,s]} | x_i^1(s') - x_i^0(s') |^r_{-1} \right ] + C_{\varepsilon} \mathbb{E} \left [ \left ( \int_t^s \left ( | x_i^1(s') - x_i^0(s') |^2_{-1} + d^2_{-1,r}(\mu_{\mathbf{x}^1(s')},\mu_{\mathbf{x}^0(s')}) \right ) \mathrm{d}s' \right )^{\frac{r}{2}} \right ],
    \end{split}
    \end{equation}
    for all $\varepsilon>0$. Now, we again use the same arguments as above to estimate the second term. Therefore, choosing $\varepsilon>0$ sufficiently small, taking the expectation in inequality \eqref{inequality_1}, summing over $i=1,\dots,n$, and dividing by $n$, we obtain
    \begin{equation}
    \begin{split}
        &\frac{1}{n} \sum_{i=1}^n \mathbb{E} \left [ \sup_{s'\in [t,s]} | x^1_i(s') - x^0_i(s') |_{-1}^r \right ]\leq \frac{C}{n} \sum_{i=1}^n | x^1_i - x^0_i|_{-1}^r + C \int_t^s \frac{1}{n} \sum_{i=1}^n \mathbb{E} \left [ \sup_{t'\in [t,s']} | x^1_i(t') - x^0_i(t') |_{-1}^r \right ] \mathrm{d}s'.
    \end{split}
    \end{equation}
    Applying Gr\"onwall's inequality, taking the power $1/r$ on both sides, and using a similar estimate to \eqref{estimate_r_norm} completes the proof.
\end{proof}

\begin{prop}\label{Lipschitz_Continuity_u_n}
	Let Assumptions \ref{Assumption_A_maximal_dissipative}, \ref{Assumption_weak_B_condition}, \ref{Assumption_f_sigma_lipschitz}(i)(ii)(iv), and \ref{Assumption_running_terminal_cost}(i)(ii)(iv) be satisfied. Then, there is a constant $C\geq 0$, independent of $n\in \mathbb{N}$, such that
    \begin{align}
        \label{u_n_linear_growth} |u_n(t,\mathbf{x})| &\leq C (1+ |\mathbf{x}|_{r} )\\
		\label{u_n_Lipschitz} |u_n(t,\mathbf{x}) - u_n(t,\mathbf{y})| &\leq C | \mathbf{x} - \mathbf{y} |_{-1,r}
	\end{align}
	for all $t\in [0,T]$ and $\mathbf{x},\mathbf{y}\in H^n$.
\end{prop}

\begin{proof}
    Let us start with inequality \eqref{u_n_linear_growth}. Using the growth assumption on $l$ and $\mathcal{U}_T$, see Assumption \ref{Assumption_running_terminal_cost}, as well as \eqref{eq:est_|x|r}, we obtain the lower bound
    \begin{equation}
    \begin{split}
        u_n(t,\mathbf{x}) &= \inf_{\mathbf{a}(\cdot) \in \Lambda^n_t} \mathbb{E} \left [ \int_t^T \frac{1}{n} \sum_{i=1}^n \left ( l_1(x_i(s) , \mu_{\mathbf{x}(s)}) + l_2(x_i(s),\mu_{\mathbf{x}(s)},a_i(s)) \right ) \mathrm{d}s + \frac{1}{n} \sum_{i=1}^n \mathcal{U}_T(x_i(T),\mu_{\mathbf{x}(T)}) \right ]\\
        &\geq \inf_{\mathbf{a}(\cdot) \in \Lambda^n_t} \Bigg \{ \mathbb{E} \left [ \int_t^T \frac{1}{n} \sum_{i=1}^n \left ( - C(1+|x_i(s)|_{-1} + \mathcal{M}_{-1,r}^{\frac{1}{r}}(\mu_{\mathbf{x}(s)}) ) - C_1 + C_2 | a_i(s) |_{\Lambda}^2 \right ) \mathrm{d}s \right ]\\
        &\qquad\qquad\qquad + \mathbb{E} \left [ \frac{1}{n} \sum_{i=1}^n \left ( - C( 1 + |x_i(T)|_{-1} + \mathcal{M}_{-1,r}^{\frac{1}{r}}(\mu_{\mathbf{x}(T)})) \right ) \right ] \Bigg \} \\
        &\geq \inf_{\mathbf{a}(\cdot)\in \Lambda^n_t} \left \{ - C - C \mathbb{E} \left [ \sup_{s\in[t,T]} |\mathbf{x}(s)|_r \right ] + C_2 \int_t^T \frac{1}{n} \sum_{i=1}^n | a_i(s) |_{\Lambda}^2 \mathrm{d}s \right \}\\
        &\geq \inf_{\mathbf{a}(\cdot)\in \Lambda^n_t} \left \{ - C ( 1+ |\mathbf{x}|_r ) - \frac{C}{\sqrt{n}} \| \mathbf{a}(\cdot) \|_{M^2(t,T;\Lambda^n)} + \frac{C_2}{n} \| \mathbf{a}(\cdot) \|_{M^2(t,T;\Lambda^n)}^2 \right \}\\
        &\geq - C(1+|\mathbf{x}|_r),
    \end{split}
    \end{equation}
    where we used Young's inequality in the last step. The upper bound for $u_n(t,\mathbf{x})$ follows from similar arguments if we use fixed controls $a_i(\cdot)=q_0, i=1,\dots,n$, for some fixed $q_0\in\tilde\Lambda$. Inequality \eqref{u_n_Lipschitz} follows from the Lipschitz assumptions on $l$ and $\mathcal{U}_T$, see again Assumption \ref{Assumption_running_terminal_cost}, as well as \eqref{eq:est_|x1-x0|r}.
\end{proof}
Define for $m>0$, $\Lambda_t^{n,m} := \{ \mathbf{a}(\cdot) \in \Lambda_t^n : \mathbf{a}(\cdot) \text{ has values in } B_{m}(0)\,\,\text{in}\,\,\Lambda^n \}$.

\begin{prop}\label{Lipschitz_Time_u_n}
	Let Assumptions \ref{Assumption_A_maximal_dissipative}, \ref{Assumption_weak_B_condition}, \ref{Assumption_f_sigma_lipschitz}(i)(ii)(iv), and \ref{Assumption_running_terminal_cost}(i)(ii)(iv) be satisfied. Then, there is a constant $C\geq 0$, independent of $n\in \mathbb{N}$, such that
	\begin{equation}\label{Lipschitz_Time_u_n_first_estimate}
		|u_n(s,\mathbf{x}) - u_n(t,\mathbf{x})| \leq C(1+|\mathbf{x}|_{r}) |s-t|^{\frac12}
	\end{equation}
	for all $s,t\in [0,T]$ and $\mathbf{x}\in H^n$. Moreover, there exists an absolute constant $K>0$ such that
	    \[
        u_n(t,\mathbf{x}) = \inf_{\mathbf{a}(\cdot) \in \Lambda_t^{n,K\sqrt{n}}} \mathbb{E}\left [ \int_t^T \frac{1}{n}\sum_{i=1}^n l(x_i(s),\mu_{{\bf x}(s)},a_i(s)) \mathrm{d}s + \frac{1}{n} \sum_{i=1}^n {\mathcal U}_T(x_i(T),\mu_{{\bf x}(T)})  \right ]
    \]
and $u_n$ is the unique $B$-continuous  solution of \eqref{finite_dimensional_hjb} in the class of functions $v$ which satisfy for some $C\geq 0$
	\[
	|v(t,{\mathbf x})-v(t,{\mathbf y})|\leq C|{\bf x}-{\bf y}|_{H^n}\quad\forall {\mathbf x},{\mathbf y}\in H^n,t\in[0,T].
	\]
	\end{prop}
\begin{proof}
    We are going to combine the  arguments from \cite[Proposition 3.1]{mayorga_swiech_2023}, \cite[Proposition 3.3]{swiech_wessels_2024} and \cite[Proposition 4.6]{defeo_swiech_wessels_2023}. Without loss of generality, let $t<s$.

    First, note that inequality \eqref{u_n_Lipschitz} implies that $u_n$ is Lipschitz continuous with respect to the $H_{-1}^n$-norm, and hence also with respect to the $H^n$-norm, with Lipschitz constant $C/\sqrt{n}$. Recall the definition of $\mathcal{H}_n$ and $\mathcal{H}^m_n$ from Section \ref{subsection_hamiltonian}. By Lemma \ref{lemma_hamiltonian_n}, there is a constant $K>0$ such that $\mathcal{H}_n(\mathbf{x},\mu_{\mathbf{x}},\mathbf{p}) = \mathcal{H}^{K\sqrt{n}}_n(\mathbf{x},\mu_{\mathbf{x}},\mathbf{p})$ for all $|\mathbf{p}|_{H^n} \leq C/\sqrt{n}$. 

    Now, for $m>0$, we introduce the auxiliary equation
    \begin{equation}\label{eq:HJBfdm}
    \begin{cases}
	   \partial_t u^m_n + \frac12 \text{Tr}(A_n(\mathbf{x},\mu_{\mathbf{x}}) D^2u^m_n) {+} \frac{1}{n} \sum_{i=1}^n \langle x_i,nA^*D_{x_i} u^m_n \rangle\\
	    \qquad\qquad+ \mathcal{H}^m_n(\mathbf{x},\mu_{\mathbf{x}} , {D} u^m_n ) = 0, \quad(t,\mathbf{x})\in (0,T)\times H^n\\
	    u^m_n(T,\mathbf{x}) = \frac{1}{n} \sum_{i=1}^n \mathcal{U}_T(x_i,\mu_{\mathbf{x}}), \quad \mathbf{x}\in H^n,
    \end{cases}
    \end{equation}
   It is well known, see \cite[Theorem 3.66]{fabbri_gozzi_swiech_2017}, that equation \eqref{eq:HJBfdm} has a unique $B$-continuous viscosity solution which is given by the value function of the associated stochastic control problem, i.e.,
    \begin{equation}
        u^m_n(t,\mathbf{x}) = \inf_{\mathbf{a}(\cdot) \in \Lambda_t^{n,m}} \mathbb{E}\left [ \int_t^T\frac{1}{n}\sum_{i=1}^n l(x_i(s),\mu_{{\bf x}(s)},a_i(s)) \mathrm{d}s+ \frac{1}{n} \sum_{i=1}^n {\mathcal U}_T(x_i(T),\mu_{{\bf x}(T)}) \right ].
    \end{equation}
The functions $u_n^m$ satisfy
    \eqref{u_n_Lipschitz} and thus are Lipschitz continuous in $H^n$ with Lipschitz constant $C/\sqrt{n}$. If $K\sqrt{n}\leq m_1<m_2$ then $u_n^{m_1}$ and $u_n^{m_2}$ are both $B$-continuous viscosity solutions of \eqref{eq:HJBfdm} with either $H_n^{m_1}$ or $H_n^{m_2}$ so by uniqueness we have $u_n^{m_1}=u_n^{m_2}$. However it is easy to see that $u_n=\lim_{m\to\infty} u_n^m$ so we obtain $u_n=u_n^{K\sqrt{n}}$.
    
    By the dynamic programming principle, we now have
    \begin{equation}
        u_n(t,\mathbf{x}) = \inf_{\mathbf{a}(\cdot) \in \Lambda_t^{n,K\sqrt{n}}} \mathbb{E}\left [ \int_t^s \frac{1}{n}\sum_{i=1}^n l(x_i(s'),\mu_{{\bf x}(s')},a_i(s')) \mathrm{d}s' + u_n(s,\mathbf{x}(s)) \right ].
    \end{equation}
    Thus,
    \begin{equation}\label{estimate_01}
    \begin{split}
        &| u_n(t,\mathbf{x}) - u_n(s,\mathbf{x}) |\\
        &\leq \sup_{\mathbf{a}(\cdot) \in \Lambda_t^{n,K\sqrt{n}}} \mathbb{E}\left [ \int_t^s \frac{1}{n}\sum_{i=1}^n | l(x_i(s'),\mu_{{\bf x}(s')},a_i(s')) | \mathrm{d}s' + | u_n(s,\mathbf{x}(s)) - u_n(s,\mathbf{x}) | \right ].
    \end{split}
    \end{equation}
    Due to Assumptions \ref{Assumption_running_terminal_cost}(i)(ii) and the fact that $\mathbf{a}(\cdot) \in \Lambda_t^{n,K\sqrt{n}}$, we have
    \begin{equation}
        \frac{1}{n}\sum_{i=1}^n | l(x_i(s'),\mu_{{\bf x}(s')},a_i(s')) | \! \leq \! \frac{C}{n} \sum_{i=1}^n \! \left ( 1 \! + |x_i(s') |_{-1} + \mathcal{M}^{\frac{1}{r}}_{-1,r}(\mu_{\mathbf{x}(s')}) + |a_i(s')|_{\Lambda}^2 \right )\leq C \! \left ( 1 + |\mathbf{x}(s')|_{-1,r} \right ).
    \end{equation}
    Moreover, by Proposition \ref{Lipschitz_Continuity_u_n}, we have
    \begin{equation}
        | u_n(s,\mathbf{x}(s)) - u_n(s,\mathbf{x}) | \leq C | \mathbf{x}(s) - \mathbf{x} |_{-1,r}.
    \end{equation}
    Altogether, we obtain from \eqref{estimate_01}
    \begin{equation}
    \begin{split}
        &| u_n(t,\mathbf{x}) - u_n(s,\mathbf{x}) |\\
        &\leq C \left ( 1 + \sup_{\mathbf{a}(\cdot) \in \Lambda_t^{n,K\sqrt{n}}} \mathbb{E} \left [ \sup_{s'\in [t,s]} | \mathbf{x}(s') |_{-1,r} \right ] \right ) |s-t| + \sup_{\mathbf{a}(\cdot) \in \Lambda_t^{n,K\sqrt{n}}} \mathbb{E} \left [ \left | \mathbf{x}(s) - \mathbf{x} \right |_{-1,r} \right ].
    \end{split}
    \end{equation}
    Applying \eqref{eq:est_|x|r}, \eqref{eq:est_|xs-x|r} concludes the proof. 
    
    Uniqueness of $B$-continuous viscosity solutions for our class of functions $v$ follows since each such function is a $B$-continuous viscosity solution of \eqref{eq:HJBfdm} for a sufficiently large $m$ and equation \eqref{eq:HJBfdm} has comparison principle, see \cite[Theorem 3.50]{fabbri_gozzi_swiech_2017} or \cite[Theorem 3.66]{fabbri_gozzi_swiech_2017}.
\end{proof}

\subsection{Convergence of \texorpdfstring{$u_n$}{un}}\label{sec:convergence_of_u_n}
In this section we assume that, in addition to the assumptions of Proposition \ref{Lipschitz_Time_u_n}, Assumption \ref{Assumption_B_compact} holds. First, we define $\mathcal{V}_n(t,\mu_{\mathbf{x}}) := u_n(t,\mathbf{x})$ on subsets $\mathcal{D}_n \subset \mathcal{P}_r(H_{-1})$ consisting of averages of $n$ Dirac point masses centered at $x_i \in H$, $i=1,\dots, n$. For $R>0$, let $\mathfrak{M}^{2}_R(H) := \{\mu \in \mathcal{P}_2(H) : \int_{H} |x|^{2} \mu(\mathrm{d}x) \leq R \}$. Note that $\mathfrak{M}^{2}_R(H)$ can be embedded into $\mathcal{P}_r(H_{-1})$ by extending a measure $\mu \in \mathfrak{M}^{2}_R(H)$ to a measure $\tilde{\mu} \in \mathcal{P}_2(H_{-1})$ defined by $\tilde{\mu}(\mathfrak{B}) = \mu(\mathfrak{B}\cap H)$, for all $\mathfrak{B}$ in the Borel $\sigma$-algebra of $H_{-1}$. Using this extension and using the same notation for the extension and the original measure, note that we have $\mu(H_{-1} \setminus H) = 0$.

For $m\in \mathbb{N}$, we extend $\mathcal{V}_n$ to functions $\mathcal{V}^m_n: [0,T]\times \mathcal{P}_r(H_{-1}) \to \mathbb{R}$ via
\begin{equation}
    \mathcal{V}_n^m(t,\mu) := \sup \left \{ \mathcal{V}_n(t,\beta) - C d_{-1,r}(\mu,\beta) : \beta = \frac{1}{n} \sum_{i=1}^n \delta_{x_i}, \; x_i\in H, \; i=1,\dots n, \; \beta \in \mathfrak{M}_{m+1}^2(H) \right \}.
\end{equation}
where the constant $C$ is the Lipschitz constant from Proposition \ref{Lipschitz_Continuity_u_n}. Note that $\mathcal{V}^m_n(t,\beta)$ coincides with $\mathcal{V}_n(t,\beta)$ for all $(t,\beta) \in [0,T]\times \mathfrak{M}_{m+1}^2(H)$ such that $\beta$ is an average of Dirac measures as in the supremum above. It is obvious from the definition of $\mathcal{V}_n^m$ and Proposition \ref{Lipschitz_Time_u_n} that
\begin{equation}\label{V_n^m_Lipschitz}
    |\mathcal{V}_n^m(t,\mu) - \mathcal{V}_n^m(t,\beta) | \leq C d_{-1,r}(\mu,\beta)
\end{equation}
for all $t\in [0,T]$ and $\mu,\beta \in \mathcal{P}_2(H_{-1})$, and that for every $m\in\mathbb{N}$ there is a constant $C_m>0$ such that
\begin{equation}\label{V_n^m_equicontinuous}
    |\mathcal{V}_n^m(t,\mu) - \mathcal{V}_n^m(s,\beta)| \leq C d_{-1,r}(\mu,\beta) + C_m |t-s|^{\frac12},
\end{equation}
for all $t,s\in [0,T]$ and $\mu,\beta \in \mathfrak{M}_{m+1}^2(H)$. In particular, for every $m\in\mathbb{N}$, the family $(\mathcal{V}^m_n)_{n\in\mathbb{N}}$ is equicontinuous and bounded on the sets $[0,T]\times \mathfrak{M}^2_{m+1}(H)$.


Next, we show that the sets $\mathfrak{M}_m^2(H)$ are relatively compact in $\mathcal{P}_r(H_{-1})$, $r\in [1,2)$. It follows from Chebyshev's inequality and the definition of $\mathfrak{M}^2_m(H)$ that $\mu( \{ |x| > R \} ) \leq \frac{1}{R^2} \int_{H} |x|^2 \mu(\mathrm{d}x) \leq \frac{m}{R^2}$. Since closed balls in $H$ are compact in $H_{-1}$ (recall that $B$ is compact), the sets $\mathfrak{M}^2_m(H)$ are tight, for every $m>0$. Since $r<2$, these sets are also $r$-uniformly integrable, see \cite[Equation (5.1.20)]{ambrosio_gigli_savare_2008}. By \cite[Proposition 7.1.5]{ambrosio_gigli_savare_2008}, this implies the relative compactness in $\mathcal{P}_r(H_{-1})$.

Therefore, we can apply the Arzel\`{a}--Ascoli theorem to extract uniformly convergent subsequences. More precisely, let $(n_k^1)_{k\in\mathbb{N}} \subset\mathbb{N}$ be a subsequence, such that $\mathcal{V}^m_{n_k^1} \to \mathcal{V}^1$, as $k \to \infty$, uniformly on $[0,T]\times \mathfrak{M}_1^2(H)$ for some function $\mathcal{V}^1 : [0,T]\times \mathfrak{M}_1^2(H) \to \mathbb{R}$. Next, let $(n_k^2)_{k\in\mathbb{N}}$ be a subsequence of $(n_k^1)_{k\in\mathbb{N}}$, such that $\mathcal{V}^2_{n_k^2} \to \mathcal{V}^2$, as $k \to \infty$, uniformly on $[0,T]\times \mathfrak{M}_2^2(H)$ for some function $\mathcal{V}^2 : [0,T]\times \mathfrak{M}_2^2(H) \to \mathbb{R}$. Continuing in the same fashion, we extract a subsequence $(n_k^m)_{k\in\mathbb{N}}$ of $(n_k^{m-1})_{k\in\mathbb{N}}$ such that $\mathcal{V}_{n_k^m}^m \to \mathcal{V}^m$, as $k\to\infty$, uniformly on $[0,T]\times \mathfrak{M}_m^2(H)$ for some function $\mathcal{V}^m : [0,T]\times \mathfrak{M}_m^2(H) \to \mathbb{R}$.

{Now, let us show that $\mathcal{V}^m$ coincides with $\mathcal{V}^{m-1}$ on the set $[0,T]\times \mathfrak{M}^2_{m-1}(H)$. Indeed, by the Glivenko--Cantelli theorem for random variables with values in a metric space, see \cite{varadarajan_1958}, for $\mu\in \mathfrak{M}^2_{m-1}(H)$,} {if $X_1,X_2,...$ is a sequence of $H$-valued i.i.d. random variables with law $\mu$, defining $\mathbf{X}_k = ( X_1,\dots,X_{n^m_k})$, we have 
that $\mu_{\mathbf{X}_k} $ converges weakly to $\mu$} {almost surely (recall that $\mu_{\mathbf{X}_k} = (1/n^m_k) \sum_{i=1}^{n^m_k} \delta_{X_i}$). Moreover, by the strong law of large numbers, the second moment of $\mu_{\mathbf{X}_k}$ converges to the second moment of $\mu$ almost surely. In particular, we can find points $\mathbf{x}_k = (x_1,\dots, x_{n^m_k})$ such that $\mu_{\mathbf{x}_k}$ {converges weakly to} $\mu$ and the moments converge, hence,
\begin{equation}
    d_2\left ( \mu, \mu_{\mathbf{x}_k} \right ) \to 0, \quad \text{as }k\to\infty.
\end{equation}
Thus, for some $k_0 \in\mathbb{N}$ sufficiently large, $\mu_{\mathbf{x}_k} \in \mathfrak{M}^2_m(H)$, for all $k\geq k_0$, and therefore $\mathcal{V}_{n^m_{k}}$, $\mathcal{V}_{n^m_{k}}^{m-1}$, and $\mathcal{V}_{n^m_{k}}^{m}$ all coincide at these measures. Hence
\begin{equation}
\begin{split}
    |\mathcal{V}^m(t,\mu) - \mathcal{V}^{m-1}(t,\mu)|
   & \leq \left | \mathcal{V}^m(t,\mu) - \mathcal{V}^m_{n_{k}^m}\left ( t,\mu_{\mathbf{x}_k} \right ) \right | + \left |  \mathcal{V}^{m-1}_{n_{k}^m}\left ( t,\mu_{\mathbf{x}_k}\right ) - \mathcal{V}^{m-1}(t,\mu) \right |\\
    &\leq \left | \mathcal{V}^m(t,\mu) - \mathcal{V}^m_{n_{k}^m}\left ( t,\mu \right ) \right | + \left | \mathcal{V}^m_{n_{k}^m}\left ( t,\mu \right ) - \mathcal{V}^m_{n_{k}^m}\left ( t,\mu_{\mathbf{x}_k} \right ) \right |\\
    &\quad + \left |  \mathcal{V}^{m-1}_{n_{k}^m}\left ( t,\mu_{\mathbf{x}_k} \right ) - \mathcal{V}^{m-1}_{n_{k}^m}(t,\mu) \right | + \left |  \mathcal{V}^{m-1}_{n_{k}^m}(t,\mu) - \mathcal{V}^{m-1}(t,\mu) \right |.
\end{split}
\end{equation}
Thus, due to \eqref{V_n^m_equicontinuous}, the right-hand side converges to zero as $k\to\infty$.}

Therefore, considering the diagonal sequence $(n_k^k)_{k\in\mathbb{N}}$ and the corresponding sequence of functions $\mathcal{V}_{n_k^k}^k$, $k\in\mathbb{N}$, we obtain a universal limit $\mathcal{V}:[0,T]\times \mathcal{P}_2(H) \to \mathbb{R}$, that is, we have $\mathcal{V}^k_{n_k^k} \to \mathcal{V}$ uniformly on $[0,T]\times \mathfrak{M}_m^2(H)$, for every $m\in\mathbb{N}$.

\begin{remark}
    Note that $\mathcal{V}$ can be extended continuously to $[0,T]\times \mathcal{P}_2(H_{-1})$. Indeed, due to \eqref{V_n^m_Lipschitz},  for each $t\in [0,T]$ we can extend $\mathcal{V}(t,\cdot)$ to $\mathcal{P}_2(H_{-1})$ and the extension is Lipschitz continuous in $d_{-1,2}$ metric with a Lipschitz constant that is independent of $t$. Now, let us show that $\mathcal{V}$ is continuous in both variables. To this end, it is enough to show that for every $(t,\mu)\in [0,T]\times \mathcal{P}_2(H_{-1})$ and every sequence $(t_n)$ in $[0,T]$ such that $t_n\to t$, we have $\mathcal{V}(t_n,\mu) \to \mathcal{V}(t,\mu)$. 
    
It is easy to see using the density of $H$ in $H_{-1}$ that for every $\varepsilon>0$ there is a measure $\mu_{\varepsilon} = (1/N) \sum_{i=1}^N \delta_{x_i}$, where $x_i\in H$, $i=1,\dots,N$, such that $d_{-1,2}(\mu,\mu_{\varepsilon}) < \varepsilon$. Note that $\mu_{\varepsilon}\in \mathfrak{M}^2_m(H)$ for some $m=m(\varepsilon)\in\mathbb{N}$. Thus, we have
    \begin{equation}
    \begin{split}
        | \mathcal{V}(t_n,\mu) - \mathcal{V}(t,\mu) | &\leq | \mathcal{V}(t_n,\mu) - \mathcal{V}(t_n,\mu_{\varepsilon}) | + | \mathcal{V}(t_n,\mu_{\varepsilon}) - \mathcal{V}(t,\mu_{\varepsilon}) | + | \mathcal{V}(t,\mu_{\varepsilon}) - \mathcal{V}(t,\mu) |\\
        &\leq C d_{-1,2}(\mu,\mu_{\varepsilon}) + C_{\varepsilon} |t_n-t|^{\frac12} + C d_{-1,2}(\mu_{\varepsilon},\mu).
    \end{split}
    \end{equation}
    Choosing $\varepsilon>0$ sufficiently small, and then taking the limit $n\to \infty$, we see that the right-hand side can be made arbitrarily small, which proves the continuity of $\mathcal{V}$.
\end{remark}

Now, we consider the lift $V:[0,T]\times E \to \mathbb{R}$ of $\mathcal{V}$ given by
\begin{equation}\label{Definition_of_V}
    V(t,X) := \mathcal{V}(t,X_{\texttt{\#}}\mathcal{L}^1).
\end{equation}
We will show that the limit $\mathcal{V}$ is an $L$-viscosity solution of equation \eqref{intro:HJB_on_Wasserstein_space} which is unique. Here, $\mathcal{V}$ is an $L$-viscosity solution of equation \eqref{intro:HJB_on_Wasserstein_space} if its lift $V$ is a $\mathcal B$-continuous viscosity solution of equation \eqref{lifted_HJB_equation}. The notion of $B$-continuous viscosity solution is recalled in Appendix \ref{app:viscosity_hilbert}.

\begin{theorem}\label{theorem:convergence}
	Let Assumptions \ref{Assumption_A_maximal_dissipative}, \ref{Assumption_weak_B_condition}, \ref{Assumption_B_compact}, \ref{Assumption_f_sigma_lipschitz}(i)(ii)(iv), and \ref{Assumption_running_terminal_cost}(i)(ii)(iv) be satisfied. Then, for every set $\mathfrak{M}^2_m(H)$, $m\in\mathbb{N}$, in $\mathcal{P}_2(H)$, we have
	\begin{equation}\label{convergence_on_M^2_R}
    \begin{split}
		&\lim_{n\to \infty} \sup \Bigg \{ \left | u_n(t,x_1,\dots,x_n) - \mathcal{V}\left ( t, \frac{1}{n} \sum_{i=1}^n \delta_{x_i} \right ) \right |\\
        &\qquad\qquad\qquad : (t,x_1,\dots,x_n)\in (0,T]\times H^n, \frac{1}{n} \sum_{i=1}^n \delta_{x_i} \in \mathfrak{M}^2_m(H) \Bigg \} =0,
    \end{split}
	\end{equation}
	where $\mathcal{V}$ is the unique $L$-viscosity solution of equation \eqref{intro:HJB_on_Wasserstein_space} in the class of functions
	$\mathcal{W}$ whose lifts are uniformly continuous on bounded subsets of $[0,T]\times E$ and satisfy 
	\begin{equation}\label{eq:abab1}
	|W(t,X)-W(t,Y)| \leq C\|X-Y\| \quad\text{for all}\,\,t\in[0,T],\,\,X,Y\in E.
	\end{equation}
	In particular, $V=U$, where $U$ and $V$ are given by \eqref{Lifted_Value_Function} and \eqref{Definition_of_V}, respectively.
\end{theorem}

\begin{proof}
    We are first going to show that $\mathcal{V}$ constructed before is an $L$-viscosity solution of equation \eqref{intro:HJB_on_Wasserstein_space}, i.e., $V$ is a $\mathcal B$-continuous viscosity solution of equation \eqref{lifted_HJB_equation}. Let us only show that $V$ is a $\mathcal B$-continuous viscosity subsolution. The proof that $V$ is a $\mathcal B$-continuous viscosity supersolution is analogous. For $\varphi\in C^{1,2}((0,T)\times E)$ and $h\in C^{1,2}((0,T)\times \mathbb{R})$, let $\psi=\varphi +h(t,\|X\|)$ be a test function as in Definition \ref{def:test_functions_hilbert}, and let $V - \psi$ have a strict, global maximum at $(t,X)\in (0,T)\times E$. We need to show that
    \begin{equation}
	   \partial_t \psi(t,X) + \frac12 \text{Tr}(\Sigma(X)(\Sigma(X))^{\ast} D^2\psi(t,X)){+}  \llangle X,\mathcal{A}^*D\varphi(t,X) \rrangle {+} \tilde{\mathcal{H}}(X, D\psi(t,X)) \geq 0.
    \end{equation}
    Recall that $A^n_i = (\frac{i-1}{n},\frac{i}{n})$, $i=1,\dots,n$, and for $\mathbf{x} \in H^n$, we have defined $X^{\mathbf{x}}_n := \sum_{i=1}^n x_i \mathbf{1}_{A_i^n}$. Now, for $t\in (0,T)$, $\mathbf{x} \in H^n$ and $c \in \mathbb{R}$, we define $\varphi_n(t,\mathbf{x}) := \varphi(t,X^{\mathbf{x}}_n)$, $h_n(t,c) := h(t,c/\sqrt{n})$ and $\psi_n(t,\mathbf{x}) := \varphi_n(t,\mathbf{x}) + h_n(t,|\mathbf{x}|_{H^n}) = \varphi(t,X^{\mathbf{x}}_n)+h(t,\| X^{\mathbf{x}}_n \|)$. Here we used that $\|X^{\mathbf{x}}_n \| = |\mathbf{x}|_{H^n}/\sqrt{n}$.

    By \eqref{convergence_on_M^2_R}, we have $\sup \{ |u_n(t,\mathbf{x}) - V(t,X^{\mathbf{x}}_n)|:t\in [0,T], \mathbf{x}\in H^n \text{ such that }X^{\mathbf{x}}_n \in B_1(X) \}\to 0$ as $n\to \infty$. Thus, since the maximum of $V-{\psi}$ at $(t,X)$ is strict, there is a sequence of points $(t_n,\mathbf{x}(n))$ such that the functions $u_n-\psi_n$ have a local maximum over $\{(s,\mathbf{x}) \in [0,T]\times H^n: X^{\mathbf{x}}_n \in B_1(X) \}$ at these points and $t_n\to t$, $X_n^{\mathbf{x}(n)} \to X$. {Indeed, notice that the local maximum is achieved since $u_n-\psi_n$ is weakly sequentially upper semicontinuous, which is a consequence of the $B$-continuity of $u_n$ and the $B$-lower semicontinuity of $\psi$,  together with the compactness of $B$.} Since $u_n$ is a $ B$-continuous viscosity solution of equation \eqref{finite_dimensional_hjb}, we have
\begin{equation}\label{viscosity_solution_finite_dimensions}
    \begin{cases}
        \partial_t \psi_n(t_n,\mathbf{x}(n)) + \frac12 \text{Tr}(A_n(\mathbf{x}(n),\mu_{\mathbf{x}(n)}) D^2\psi_n(t_n,\mathbf{x}(n)))\\
        \quad {+} \frac{1}{n} \sum_{i=1}^n \left ( \langle x_i(n),n A^*D_{x_i} \varphi_n(t_n,\mathbf{x}(n)) \rangle + \mathcal{H}(x_i(n),\mu_{\mathbf{x}(n)},n D_{x_i}\psi_n(t_n,\mathbf{x}(n))) \right ) \geq 0.
    \end{cases}
    \end{equation}

    For the derivatives of $\varphi_n$, $h_n$, and $\psi_n$, we have
    \begin{align}
        \label{derivative_varphi_n} &D_{x_i} \varphi_n(t,\mathbf{x}) = \int_{\Omega} D\varphi(t,X^{\mathbf{x}}_n) \mathbf{1}_{A^n_i} \mathrm{d}\omega\\
        \label{derivative_h_n} &D_{x_i} h_n(t,|\mathbf{x}|_{H^n}) = \int_{\Omega} \partial_r h(t,\|X^{\mathbf{x}}_n\|) \frac{X^{\mathbf{x}}_n}{\|X^{\mathbf{x}}_n\|} \mathbf{1}_{A^n_i} \mathrm{d}\omega\\
        \label{derivative_psi_n}& D_{x_i} \psi_n(t,\mathbf{x}) = D_{x_i} \varphi_n(t,\mathbf{x}) + D_{x_i} h_n(t,|\mathbf{x}|)\\
        \label{derivative_psi_n_2}
        &\langle D_{x_j x_i}^2 \psi_n(t,\mathbf{x}) e_k, e_l \rangle = \int_{\Omega} \langle D^2 \psi(t,X^{\mathbf{x}}_n)(e_k \mathbf{1}_{A^n_i}) , e_l \mathbf{1}_{A^n_j} \rangle \mathrm{d}\omega,
    \end{align}
    for $i,j=1,\dots,n$ and $k,l\in \mathbb{N}$, where $(e_k)_{k\in \mathbb{N}}$ is an orthonormal basis of $H$. Indeed, let $f_i$ be the $i$-th standard basis vector in $\mathbb{R}^n$. Then, for \eqref{derivative_varphi_n} we have
    \begin{equation}
    \begin{split}
        \langle D_{x_i} \varphi_n(t,\mathbf{x}), e_k \rangle &= \lim_{\varepsilon \to 0} \frac{1}{\varepsilon} \left ( \varphi_n(t,\mathbf{x} + \varepsilon f_i e_k \rangle - \varphi_n(t,\mathbf{x}) \right )\\
        &= \lim_{\varepsilon\to 0} \frac{1}{\varepsilon} \left ( \varphi \left ( t, \sum_{j=1}^n x_j \mathbf{1}_{A^n_j} + \varepsilon e_k \mathbf{1}_{A^n_i} \right ) - \varphi_n \left (t,\sum_{j=1}^n x_j \mathbf{1}_{A^n_j} \right ) \right )\\
        &= \int_{\Omega} \langle D\varphi(t,X^{\mathbf{x}}_n) , e_k \rangle \mathbf{1}_{A^n_i} \mathrm{d}\omega= \left \langle \int_{\Omega} D\varphi(t,X^{\mathbf{x}}_n) \mathbf{1}_{A^n_i} \mathrm{d}\omega , e_k \right \rangle.
    \end{split}
    \end{equation}
    A similar calculation shows \eqref{derivative_h_n} and \eqref{derivative_psi_n}. For \eqref{derivative_psi_n_2}, we have
    \begin{equation}
    \begin{split}
        \langle D_{x_j x_i}^2 \psi_n(t,\mathbf{x}) e_k, e_l \rangle &= \langle D_{x_j} D_{x_i} \psi_n(t,\mathbf{x}) e_k,e_l \rangle = \langle D_{x_j} D\psi(t,X^{\mathbf{x}}_n) (e_k \mathbf{1}_{A^n_i}), e_l \rangle\\
        &= \lim_{\varepsilon\to 0} \frac{1}{\varepsilon} \left ( D \psi(t,X^{\mathbf{x}}_n + \varepsilon e_l f_j)(e_k \mathbf{1}_{A^n_i}) - D\psi(t,X^{\mathbf{x}}_n) (e_k \mathbf{1}_{A^n_i}) \right )\\
        &= D^2 \psi(t,X^{\mathbf{x}}_n) (e_k \mathbf{1}_{A^n_i})(e_l \mathbf{1}_{A^n_j})= \int_{\Omega} \langle D^2 \psi(t,X^{\mathbf{x}}_n) (e_k \mathbf{1}_{A^n_i}), e_l \mathbf{1}_{A^n_j} \rangle \mathrm{d}\omega.
    \end{split}
    \end{equation}
    
    Thus, a straightforward computation shows that
    \begin{align}
        \label{correspondence_second_order_term} &\text{Tr}(A_n(\mathbf{x},\mu_{\mathbf{x}}) D^2\psi_n) = \text{Tr}(\Sigma(X^{\mathbf{x}}_n)(\Sigma(X^{\mathbf{x}}_n))^{\ast} D^2\psi(t,X^{\mathbf{x}}_n))\\
    \label{correspondence_unbounded_order_term} & \frac{1}{n} \sum_{i=1}^n \langle x_i,nA^*D_{x_i} \varphi_n \rangle = \sum_{i=1}^n \left \langle x_i , \int_{A^n_i} \mathcal{A}^* D\varphi(t,X^{\mathbf{x}}_n) \mathrm{d}\omega \right \rangle\\ \label{correspondence_hamiltonian}
        &\frac{1}{n} \sum_{i=1}^n \mathcal{H}(x_i,\mu_{\mathbf{x}},n D_{x_i} \psi_n)\\
        &= \frac{1}{n} \sum_{i=1}^n \mathcal{H}\left (x_i,\mu_{\mathbf{x}},n \int_{A^n_i} D\varphi(t,X^{\mathbf{x}}_n) \mathrm{d}\omega + n \int_{A^n_i} \partial_r h(t,\|X^{\mathbf{x}}_n\|) \frac{X^{\mathbf{x}}_n}{\|X^{\mathbf{x}}_n\|} \mathrm{d}\omega \right ).
    \end{align}
    Indeed, regarding \eqref{correspondence_second_order_term}, first recall that $A_n(\mathbf{x},\mu_{\mathbf{x}})$ is an $n\times n$-matrix consisting of $n^2$ trace-class operators $(A_n)_{ij}(\mathbf{x},\mu_{\mathbf{x}}) =\sigma(x_i,\mu_{\mathbf{x}}) \sigma^{\top}(x_j,\mu_{\mathbf{x}})$, for $i,j=1,\dots,n$. Let $(\xi_m)_{m\in \mathbb{N}}$ denote an orthonormal basis of $\Xi$. Then, we have
    \begin{equation}
    \begin{split}
        \text{Tr}(A_n(\mathbf{x},\mu_{\mathbf{x}}) D^2\psi_n) &= \sum_{i,j=1}^n \text{Tr} \left ( \sigma(x_i,\mu_{\mathbf{x}}) \sigma^{\top}(x_j,\mu_{\mathbf{x}}) D^2_{x_jx_i} \psi_n(t,\mathbf{x}) \right )\\
        &= \sum_{m=1}^{\infty} \sum_{i,j=1}^n 
        \langle D^2_{x_jx_i} \psi_n(t,\mathbf{x})\sigma(x_i,\mu_{\mathbf{x}}) \xi_m , \sigma(x_j,\mu_{\mathbf{x}}) \xi_m \rangle\\
        &= \sum_{m=1}^{\infty}\sum_{i,j=1}^n \sum_{k=1}^{\infty} \langle D^2_{x_jx_i} \psi_n(t,\mathbf{x})\sigma(x_i,\mu_{\mathbf{x}}) \xi_m, e_k \rangle \langle \sigma(x_j,\mu_{\mathbf{x}}) \xi_m, e_k \rangle\\
        &= \sum_{m=1}^{\infty}\sum_{i,j=1}^n \sum_{k,l=1}^{\infty} \langle D^2_{x_jx_i} \psi_n(t,\mathbf{x}) e_l, e_k \rangle \langle \sigma(x_i,\mu_{\mathbf{x}}) \xi_m, e_l \rangle \langle \sigma(x_j,\mu_{\mathbf{x}}) \xi_m, e_k \rangle.
    \end{split}
    \end{equation}
    Thus, using the formula for the second derivative of $\psi_n$ \eqref{derivative_psi_n_2}, we obtain
    \begin{equation}
    \begin{split}
        & \text{Tr}(A_n(\mathbf{x},\mu_{\mathbf{x}}) D^2\psi_n)\\
        &= \sum_{m=1}^{\infty}\sum_{i,j=1}^n \sum_{k,l=1}^{\infty} \left ( \int_{\Omega} \langle D^2 \psi(t,X^{\mathbf{x}}_n) (e_l \mathbf{1}_{A^n_i}), e_k \mathbf{1}_{A^n_j} \rangle \mathrm{d}\omega \right ) \langle \sigma(x_i,\mu_{\mathbf{x}}) \xi_m, e_l \rangle \langle \sigma(x_j,\mu_{\mathbf{x}}) \xi_m, e_k \rangle\\
        &= \sum_{m=1}^{\infty}\sum_{i,j=1}^n \int_{\Omega} \left \langle D^2 \psi(t,X^{\mathbf{x}}_n) \left ( \sum_{l=1}^{\infty} \langle \sigma(x_i,\mu_{\mathbf{x}}) \xi_m, e_l \rangle e_l \mathbf{1}_{A^n_i} \right ) , \sum_{k=1}^{\infty} \langle \sigma(x_j,\mu_{\mathbf{x}}) \xi_m, e_k \rangle e_k \mathbf{1}_{A^n_j} \right \rangle \mathrm{d}\omega\\
        &= \sum_{m=1}^{\infty}\sum_{i,j=1}^n \int_{\Omega} \left \langle D^2 \psi(t,X^{\mathbf{x}}_n) \left ( \sigma(x_i,\mu_{\mathbf{x}}) \xi_m \mathbf{1}_{A^n_i} \right ) , \sigma(x_j,\mu_{\mathbf{x}}) \xi_m \mathbf{1}_{A^n_j} \right \rangle \mathrm{d}\omega\\
        &= \sum_{m=1}^{\infty} \int_{\Omega} \left \langle D^2 \psi(t,X^{\mathbf{x}}_n) \left ( \sigma(X^{\mathbf{x}}_n,\mu_{\mathbf{x}}) \xi_m \right )(\omega) , \sigma(X^{\mathbf{x}}_n(\omega),\mu_{\mathbf{x}}) \xi_m \right \rangle \mathrm{d}\omega\\
        &= \sum_{m=1}^{\infty} \llangle D^2 \psi(t,X^{\mathbf{x}}_n) \left ( \Sigma(X^{\mathbf{x}}_n) \xi_m \right ) , \Sigma(X^{\mathbf{x}}_n) \xi_m \rrangle= \text{Tr}(\Sigma(X^{\mathbf{x}}_n)(\Sigma(X^{\mathbf{x}}_n))^{\ast} D^2\psi(t,X^{\mathbf{x}}_n)).
    \end{split}
    \end{equation}
    For \eqref{correspondence_unbounded_order_term}, we have
    \begin{equation}
    \begin{split}
        &\frac{1}{n} \sum_{i=1}^n \langle x_i,nA^*D_{x_i} \varphi_n(t,\mathbf{x}) \rangle = \sum_{i=1}^n \langle x_i,A^* \int_{\Omega} D\varphi(t,X^{\mathbf{x}}_n) \mathbf{1}_{A^n_i} \mathrm{d}\omega \rangle\\
        &= \sum_{i=1}^n \langle x_i,A^* \int_{A^n_i} D\varphi(t,X^{\mathbf{x}}_n) \mathrm{d}\omega \rangle = \sum_{i=1}^n \langle x_i, \int_{A^n_i} \mathcal{A}^* D\varphi(t,X^{\mathbf{x}}_n) \mathrm{d}\omega \rangle \\
        &= \sum_{i=1}^n \int_{A^n_i} \langle x_i, (\mathcal{A}^* D\varphi(t,X^{\mathbf{x}}_n))(\omega) \rangle \mathrm{d}\omega = \sum_{i=1}^n \int_{A^n_i} \langle X^{\mathbf{x}}_n(\omega) , (\mathcal{A}^* D\varphi(t,X^{\mathbf{x}}_n))(\omega) \rangle \mathrm{d}\omega\\
        &= \int_{\Omega} \langle X^{\mathbf{x}}_n(\omega) , (\mathcal{A}^* D\varphi(t,X^{\mathbf{x}}_n))(\omega) \rangle \mathrm{d}\omega = \llangle X^{\mathbf{x}}_n, \mathcal{A}^* D\varphi(t,X^{\mathbf{x}}_n) \rrangle,
    \end{split}
    \end{equation}
    where we used \cite[Proposition 1.29]{fabbri_gozzi_swiech_2017} to interchange the integral and the unbounded operator. Finally, \eqref{correspondence_hamiltonian} directly follows from the derivative formulas \eqref{derivative_varphi_n}, \eqref{derivative_h_n}, and \eqref{derivative_psi_n}.

    Thus, we derive from \eqref{viscosity_solution_finite_dimensions}
    \begin{equation}\label{inequality_n}
    \begin{cases}
        \partial_t \psi(t,X^{\mathbf{x}(n)}_n) + \frac12 \text{Tr}(\Sigma(X^{\mathbf{x}(n)}_n)(\Sigma(X^{\mathbf{x}(n)}_n))^{\ast} D^2\psi(t,X^{\mathbf{x}(n)}_n)) {+} \llangle X^{\mathbf{x}(n)}_n , \mathcal{A}^* D\varphi(t,X^{\mathbf{x}(n)}_n) \rrangle\\
        \quad {+} \frac{1}{n} \sum_{i=1}^n \mathcal{H}\left (x_i(n),\mu_{\mathbf{x}(n)}, n \int_{A^n_i} D \psi(t,X^{\mathbf{x}(n)}_n) \mathrm{d}\omega \right ) \geq 0.
    \end{cases}
    \end{equation}
    Now, let $X^{\varepsilon}, Y^{\varepsilon} \in E$, $\varepsilon>0$, be continuous on $\bar{\Omega} = [0,1]$ such that $\|X^{\varepsilon} - X \| < \varepsilon$ and $\|Y^{\varepsilon} - D\psi(t,X) \| < \varepsilon$. We note that
    \begin{equation}\label{proof_of_convergence}
    \begin{split}
        &\left | \frac{1}{n} \sum_{i=1}^n \mathcal{H}\left (x_i(n),\mu_{\mathbf{x}(n)}, n \int_{A^n_i} D \psi(t,X^{\mathbf{x}(n)}_n) \mathrm{d}\tilde{\omega} \right ) - \int_{\Omega} \mathcal{H}(X^{\varepsilon}(\omega),X^{\varepsilon}_{\texttt{\#}}\mathcal{L}^1, Y^{\varepsilon}(\omega)) \mathrm{d}\omega \right |\\
        &\leq \left | \frac{1}{n} \sum_{i=1}^n \mathcal{H}\left (x_i(n),\mu_{\mathbf{x}(n)}, n \int_{A^n_i} D \psi(t,X^{\mathbf{x}(n)}_n) \mathrm{d}\tilde{\omega} \right ) - \frac{1}{n} \sum_{i=1}^n \mathcal{H}\left (x_i(n),\mu_{\mathbf{x}(n)}, n \int_{A^n_i} Y^{\varepsilon} \mathrm{d}\tilde{\omega} \right ) \right |\\
        &\quad + \left | \frac{1}{n} \sum_{i=1}^n \mathcal{H}\left (x_i(n),\mu_{\mathbf{x}(n)}, n \int_{A^n_i} Y^{\varepsilon} \mathrm{d}\tilde{\omega} \right ) - \sum_{i=1}^n \int_{A_i^n} \mathcal{H}\left (X^{\varepsilon}(\omega),X^{\varepsilon}_{\texttt{\#}}\mathcal{L}^1, n \int_{A^n_i} Y^{\varepsilon} \mathrm{d}\tilde{\omega} \right ) \mathrm{d}\omega \right |\\
        &\quad + \left | \sum_{i=1}^n \int_{A^n_i} \mathcal{H}\left (X^{\varepsilon}(\omega),X^{\varepsilon}_{\texttt{\#}}\mathcal{L}^1, n \int_{A^n_i} Y^{\varepsilon} \mathrm{d}\tilde{\omega} \right ) \mathrm{d}\omega - \int_{\Omega} \mathcal{H}(X^{\varepsilon}(\omega),X^{\varepsilon}_{\texttt{\#}}\mathcal{L}^1, Y^{\varepsilon}(\omega)) \mathrm{d}\omega \right |.
    \end{split}
    \end{equation}
    Next, we estimate the terms on the right-hand side one by one. For the first term, we have
    \begin{equation}\label{proof_first_estimate}
    \begin{split}
        &\left | \frac{1}{n} \sum_{i=1}^n \mathcal{H}\left (x_i(n),\mu_{\mathbf{x}(n)}, n \int_{A^n_i} D \psi(t,X^{\mathbf{x}(n)}_n) \mathrm{d}\tilde{\omega} \right ) - \frac{1}{n} \sum_{i=1}^n \mathcal{H}\left (x_i(n),\mu_{\mathbf{x}(n)}, n \int_{A^n_i} Y^{\varepsilon} \mathrm{d}\tilde{\omega} \right ) \right |\\
        &\leq C \left ( \frac{1}{n} \sum_{i=1}^n \left ( 1 + |x_i(n) |^2 + \mathcal{M}_r^{\frac{2}{r}} ( \mu_{\mathbf{x}(n)} ) + \left | n \int_{A^n_i} D\psi(t,X^{\mathbf{x}(n)}_n) \mathrm{d}\omega \right |^2 + \left | n \int_{A^n_i} P_{\varepsilon}(\omega) \mathrm{d}\omega \right |^2 \right ) \right )^{\frac12}\\
        &\quad \times \left ( \frac{1}{n} \sum_{i=1}^n \left | n \int_{A^n_i} \left ( D\psi(t,X^{\mathbf{x}(n)}_n) - Y^{\varepsilon} \right ) \mathrm{d}\omega \right |^2 \right )^{\frac12}.
    \end{split}
    \end{equation}
    Note that all the terms in the first factor can easily be bounded uniformly in $n$. Indeed, for the first and second term, we note that
    \begin{equation}
        \mathcal{M}_r^{\frac{2}{r}} ( \mu_{\mathbf{x}(n)} ) \leq \frac{1}{n} \sum_{i=1}^n | x_i(n) |^2 = \| X^{\mathbf{x}(n)}_n \|^2.
    \end{equation}
    For the third term, we have
    \begin{equation}
        \frac{1}{n} \sum_{i=1}^n \left | n \int_{A^n_i} D\psi(t,X^{\mathbf{x}(n)}_n) \mathrm{d}\omega \right |^2 \leq \frac{1}{n} \sum_{i=1}^n n \int_{A^n_i} | D\psi(t,X^{\mathbf{x}(n)}_n) |^2 \mathrm{d}\omega = \| D\psi(t,X^{\mathbf{x}(n)}_n) \|^2,
    \end{equation}
    which converges to $\| D\psi(t,X) \|^2$ as $n\to\infty$. Finally, for the fourth term, we have
    \begin{equation}
        \frac{1}{n} \sum_{i=1}^n \left | n \int_{A^n_i} P_{\varepsilon}(\omega) \mathrm{d}\omega \right |^2 \leq \frac{1}{n} \sum_{i=1}^n n \int_{A^n_i} | P_{\varepsilon}(\omega) |^2 \mathrm{d}\omega = \| P_{\varepsilon} \|^2.
    \end{equation}
    For the second factor on the right-hand side of \eqref{proof_first_estimate}, we have
    \begin{equation}
    \begin{split}
        \frac{1}{n} \sum_{i=1}^n \left | n \int_{A^n_i} \left ( D\psi(t,X^{\mathbf{x}(n)}_n) - Y^{\varepsilon} \right ) \mathrm{d}\omega \right |^2 &\leq \int_{\Omega} \left | D\psi(t,X^{\mathbf{x}(n)}_n) - Y^{\varepsilon} \right |^2 \mathrm{d}\omega \leq \| D\psi(t,X^{\mathbf{x}(n)}_n) - Y^{\varepsilon} \|^2,
    \end{split}
    \end{equation}
    which converges to $\| D\psi(t,X) - Y^{\varepsilon} \|^2$ as $n\to\infty$. For the second term on the right-hand side of \eqref{proof_of_convergence}, we have
    \begin{equation}
    \begin{split}
        &\left | \frac{1}{n} \sum_{i=1}^n \mathcal{H}\left (x_i(n),\mu_{\mathbf{x}(n)}, n \int_{A^n_i} Y^{\varepsilon} \mathrm{d}\tilde{\omega} \right ) - \sum_{i=1}^n \int_{A_i^n} \mathcal{H}\left (X^{\varepsilon}(\omega),X^{\varepsilon}_{\texttt{\#}}\mathcal{L}^1, n \int_{A^n_i} Y^{\varepsilon} \mathrm{d}\tilde{\omega} \right ) \mathrm{d}\omega \right |\\
        &\leq C \sum_{i=1}^n \int_{A^n_i} \left ( 1+ \left | n \int_{A^n_i} Y^{\varepsilon} \mathrm{d}\tilde{\omega} \right | \right ) \left ( |x_i(n) - X^{\varepsilon}(\omega) | + d_r(\mu_{\mathbf{x}(n)},X^{\varepsilon}_{\texttt{\#}} \mathcal{L}^1) \right ) \mathrm{d}\omega\\
        &\leq C \left ( \sum_{i=1}^n \int_{A^n_i} \left ( 1+ \left | n \int_{A^n_i} Y^{\varepsilon} \mathrm{d}\tilde{\omega} \right |^2 \right ) \mathrm{d}\omega \right )^\frac12 \left ( \sum_{i=1}^n \int_{A^n_i} \left ( |x_i(n) - X^{\varepsilon}(\omega) |^2 + d^2_r(\mu_{\mathbf{x}(n)},X^{\varepsilon}_{\texttt{\#}} \mathcal{L}^1) \right ) \mathrm{d}\omega \right )^{\frac12}.
    \end{split}
    \end{equation}
    Again, the first factor can easily be bounded uniformly in $n$. Indeed, we have
    \begin{equation}
        \sum_{i=1}^n \int_{A^n_i} \left | n \int_{A^n_i} Y^{\varepsilon} \mathrm{d}\tilde{\omega} \right |^2 \mathrm{d}\omega \leq \frac{1}{n} \sum_{i=1}^n n \int_{A^n_i} | Y^{\varepsilon}(\tilde{\omega}) |^2 \mathrm{d}\tilde{\omega} = \| Y^{\varepsilon} \|^2.
    \end{equation}
    For the terms in the second factor, we note that
    \begin{equation}
    \begin{split}
        \sum_{i=1}^n \int_{A^n_i} | x_i(n) - X^{\varepsilon}(\omega) |^2 \mathrm{d}\omega = \sum_{i=1}^n \int_{A^n_i} | X^{\mathbf{x}(n)}_n - X^{\varepsilon} |^2 \mathrm{d}\omega = \| X^{\mathbf{x}(n)}_n - X^{\varepsilon} \|^2,
    \end{split}
    \end{equation}
    as well as $d^2_r(\mu_{\mathbf{x}(n)}, X^{\varepsilon}_{\texttt{\#}} \mathcal{L}^1) \leq \| X^{\mathbf{x}(n)}_n - X^{\varepsilon} \|^2$. Moreover, $\| X^{\mathbf{x}(n)}_n - X^{\varepsilon} \|$ converges to $\| X - X^{\varepsilon} \|$. Finally, for the third term on the right-hand side of \eqref{proof_of_convergence}, we have
    \begin{equation}
    \begin{split}
        &\left | \sum_{i=1}^n \int_{A^n_i} \mathcal{H}\left (X^{\varepsilon}(\omega),X^{\varepsilon}_{\texttt{\#}}\mathcal{L}^1, n \int_{A^n_i} Y^{\varepsilon} \mathrm{d}\tilde{\omega} \right ) \mathrm{d}\omega - \int_{\Omega} \mathcal{H}(X^{\varepsilon}(\omega),X^{\varepsilon}_{\texttt{\#}}\mathcal{L}^1, Y^{\varepsilon}(\omega)) \mathrm{d}\omega \right |\\
        &\leq C \sum_{i=1}^n \int_{A^n_i} \left ( 1 + X^{\varepsilon}(\omega) + \mathcal{M}^{\frac{1}{r}}(X^{\varepsilon}_{\texttt{\#}} \mathcal{L}^1) + \left | n \int_{A^n_i} Y^{\varepsilon} \mathrm{d}\tilde{\omega} \right | + |Y^{\varepsilon}(\omega)| \right ) \left | n \int_{A^n_i} Y^{\varepsilon} \mathrm{d}\tilde{\omega} - Y^{\varepsilon}(\omega) \right | \mathrm{d}\omega\\
        &\leq C \left ( \sum_{i=1}^n \int_{A^n_i} \left ( 1 + |X^{\varepsilon}(\omega)|^2 + \mathcal{M}_r^{\frac{2}{r}}(X^{\varepsilon}_{\texttt{\#}} \mathcal{L}^1) + \left | n \int_{A^n_i} Y^{\varepsilon} \mathrm{d}\tilde{\omega} \right |^2 +|Y^{\varepsilon}(\omega)|^2 \right ) \mathrm{d}\omega \right )^{\frac12}\\
        &\quad \times \left ( \sum_{i=1}^n \int_{A^n_i} \left | n \int_{A^n_i} Y^{\varepsilon} \mathrm{d}\tilde{\omega} - Y^{\varepsilon}(\omega) \right |^2 \mathrm{d}\omega \right )^{\frac12}.
    \end{split}
    \end{equation}
    The terms in the first factor can be bounded uniformly in $n$ using the same arguments as before. The second factor tends to zero as $n\to\infty$ due to the uniform continuity of $Y^{\varepsilon}$.
    
    Altogether, taking the limit $n\to\infty$ in \eqref{inequality_n}, we obtain
    \begin{equation}
        \partial_t \psi(t,X) + \frac12 \text{Tr}(\Sigma(X)(\Sigma(X))^{\ast} D^2\psi(t,X)) {+} \llangle X , \mathcal{A}^* D\varphi(t,X) \rrangle {+} \tilde{\mathcal{H}} \left (X^{\varepsilon}, Y^{\varepsilon} \right ) \geq \rho(\varepsilon),
    \end{equation}
    where $\rho(\varepsilon) \to 0$ as $\varepsilon \to 0$. Thus, taking $\varepsilon\to 0$ concludes the proof that $\mathcal{V}$ is an $L$-viscosity solution of equation \eqref{intro:HJB_on_Wasserstein_space}. For further details, see also the proof of \cite[Theorem 3.4]{swiech_wessels_2024}.

Regarding uniqueness, we notice that if ${\mathcal W}$ is another $L$-viscosity solutions of \eqref{intro:HJB_on_Wasserstein_space} whose lift $W$ satisfies \eqref{eq:abab1}, then $V$ and $W$ are $\mathcal B$-continuous viscosity solutions of equation 
     \[
\begin{cases}\label{eq:eqHJBK1}
	\partial_t W + \frac12 \text{Tr}(\Sigma(X)(\Sigma(X))^{\ast} D^2W) {+} \llangle \mathcal{A}X,DW \rrangle  {+} \tilde{\mathcal{H}}^K(X, DW) =0,\quad (t,X)\in (0,T)\times E\\
	W(T,X) = U_T(X), \quad X\in E,
\end{cases}
\]
for some $K>0$, where
\[
	\tilde{\mathcal{H}}^K(X,P) := \inf_{Q\in{\mathcal E},\|Q\|_{\Lambda}\leq K}\left\{ \llangle F(X,Q) , P\rrangle + L(X,Q)\right\}.
\]
Thus we can use \cite[Theorem 3.66]{fabbri_gozzi_swiech_2017} to claim that $V=W$. We also notice that by Proposition \ref{Value_Function_Lipschitz}, $U$ is a $\mathcal B$-continuous viscosity solution of \eqref{eq:eqHJBK1} and hence $U=V$.

Finally, the convergence of the whole sequence $\mathcal{V}_n$ follows from the fact that the same arguments can be used to prove that we can choose a convergent subsequence from every subsequence of
$\mathcal{V}_n$ and show that the limit is an $L$-viscosity solutions of \eqref{intro:HJB_on_Wasserstein_space}, hence it is equal to $\mathcal{V}$.
\end{proof}

\section{\texorpdfstring{$C^{1,1}$}{C{1,1}}\label{sec:regularity} Regularity of the Value Function of the Lifted Control Problem}

\subsection{Estimates for the Lifted Control Problem}

\begin{lemma}\label{Infinite_Dimensional_A_Priori}
	Let Assumptions \ref{Assumption_A_maximal_dissipative}, \ref{Assumption_weak_B_condition} and \ref{Assumption_f_sigma_lipschitz}(i)(ii)(iv) be satisfied. Then, for every $X\in E,a(\cdot)\in \Lambda_t$, \eqref{Lifted_State_Equation} has a unique mild solution $X(\cdot) \in L^2([t,T]\times \Omega'; E)$ which is progressively measurable and has continuous trajectories. Let $X(\cdot),X_0(\cdot),X_1(\cdot)$ be the solutions of equation \eqref{Lifted_State_Equation} with initial conditions $X,X_0,X_1\in E$, and controls $a(\cdot),a_0(\cdot),a_1(\cdot)\in \Lambda_t$, respectively. Then, there is a constant $C\geq 0$ such that
    \begin{align}
        &\mathbb{E} \left [ \int_{\Omega} \sup_{s\in [t,T]} | X(s,\omega) |_{-1}^2 \mathrm{d}\omega \right ] \leq C \left ( 1 + \|X\|_{-1}^2 + \| a(\cdot) \|^2_{M^2(t,T;\mathcal{E})} \right ) \label{Apriori_0_Infinite_State}\\
        &\mathbb{E} \left [ \int_{\Omega} \sup_{s\in [t,T]} | X(s,\omega) |^2 \mathrm{d}\omega \right ] \leq C \left ( 1 + \|X\|^2 + \| a(\cdot) \|^2_{M^2(t,T;\mathcal{E})} \right ) \label{Apriori_0_Infinite_State_strong}\\
        &\mathbb{E} \left [ \int_{\Omega} \sup_{s'\in [t,s]} | X(s',\omega) - X(\omega) |^2_{-1} \mathrm{d}\omega \right ] \leq C \mathbb{E} \left [ \int_t^s \| a(s') \|_{\Lambda}^2 \mathrm{d}s' \right ] + C(1+\|X\|^2) (s-t) \label{Apriori_2_Infinite_State}\\
        &\mathbb{E} \left [ \int_{\Omega} \sup_{s\in [t,T]} | X_1(s,\omega) - X_0(s,\omega) |_{-1}^2 \mathrm{d}\omega \right ] \leq C \left ( \|X_1-X_0\|_{-1}^2 + \| a_1(\cdot) - a_0(\cdot) \|^2_{M^2(t,T;\mathcal{E})} \right ) \label{Apriori_1_Infinite_State}
    \end{align}
    for all $s\in [t,T]$, $X,X_0,X_1\in E$, and $a(\cdot),a_1(\cdot),a_2(\cdot)\in \Lambda_t$. Moreover, if $a(\cdot)$ is bounded, then $\mathbb{E} [ \sup_{s\in [t,T]} \| X(s) \|^k]<\infty$ for all $k\geq 1$.
    
\end{lemma}

\begin{proof}
As in the proof of Proposition \ref{Prop_Finite_A_Priori}, the existence of a unique mild solution follows e.g. from \cite[Theorem 6.5, page 162]{PLChow}, the continuity of paths follows from the continuity of paths of the stochastic convolution, see \cite[Theorem 6.2, page 159]{PLChow} or \cite[Theorem 1.112]{fabbri_gozzi_swiech_2017}, and the moment estimates when $\mathbf{a}(\cdot)$ follow for instance from \cite[Theorem 1.130]{fabbri_gozzi_swiech_2017}.
 Similarly to the proof of Proposition \ref{Prop_Finite_A_Priori}, in the proof we will assume that all the terms are well defined and have sufficient integrability to apply the necessary theorems.

Since $X(\cdot)$ is a solution of equation \eqref{Lifted_State_Equation}, we have
    \begin{equation}
        X(s') = e^{(s'-t)\mathcal{A}} X + \int_t^{s'} e^{(s'-t')\mathcal{A}} F(X(t'),a(t')) \mathrm{d}t' + \int_t^{s'} e^{(s'-t')\mathcal{A}} \Sigma(X(t')) \mathrm{d}W(t').
    \end{equation}
   We note that, since $E$ and $L^2(\Omega;\Lambda)$ are separable, for a.e. $\omega$, the processes $X(\cdot,\omega)$ and $a(\cdot,\omega)$ are progressively measurable. We have by \eqref{semigroups_coincide}
    \begin{equation}
    \begin{split}
        X(s',\omega) &= e^{(s'-t)A} X(\omega) + \int_t^{s'} e^{(s'-t')A} f(X(t',\omega), X(t')_{\texttt{\#}} \mathcal{L}^1,a(t',\omega)) \mathrm{d}t'\\
        &\quad + \int_t^{s'} e^{(s'-t')A} \sigma(X(t',\omega),X(t')_{\texttt{\#}} \mathcal{L}^1) \mathrm{d}W(t')
    \end{split}
    \end{equation}
    for almost every $\omega \in \Omega$. Note that we can pull the $\omega$ inside the stochastic integral since the Wiener process $(W(s))_{s\in [t,T]}$ does not depend on $\omega \in \Omega$. Therefore, for almost every $\omega\in \Omega$, $X(\cdot,\omega)$ is a mild solution of the SDE
    \begin{equation}
    \begin{cases}
        \mathrm{d}X(s,\omega) = [AX(s,\omega) + f(X(s,\omega),X(s)_{\texttt{\#}} \mathcal{L}^1,a(s,\omega)) ] \mathrm{d}s + \sigma(X(s,\omega),X(s)_{\texttt{\#}} \mathcal{L}^1) \mathrm{d}W(s)\\
        X(t,\omega) = X(\omega).
    \end{cases}
    \end{equation}
    Here, the law $X(s)_{\texttt{\#}} \mathcal{L}^1$ is the law of the solution of equation \eqref{Lifted_State_Equation}. Thus, by \cite[Proposition 1.164]{fabbri_gozzi_swiech_2017} (recall Remark \ref{rem:Itoformulas}) and Assumption \ref{Assumption_weak_B_condition}, we obtain
    \begin{equation}\label{itos_formula}
    \begin{split}
        &|X(s',\omega)|_{-1}^2\\
        &\leq |X(\omega)|_{-1}^2 + 2 c_0 \int_t^{s'} | X(t',\omega) |_{-1}^2 \mathrm{d}t'  + 2 \int_t^{s'} \langle X(t',\omega), f(X(t',\omega),X(t')_{\texttt{\#}} \mathcal{L}^1,a(t',\omega)) \rangle_{-1} \mathrm{d}t'\\
        &\quad + \int_t^{s'} \left | \sigma(X(t',\omega),X(t')_{\texttt{\#}} \mathcal{L}^1) \right |_{L_2(\Xi,H_{-1})}^2 \mathrm{d}t' + 2 \int_t^{s'} \langle X(t',\omega), \sigma(X(t',\omega),X(t')_{\texttt{\#}} \mathcal{L}^1) \mathrm{d}W(t') \rangle_{-1}.
    \end{split}
    \end{equation}
    For the term involving $f$, using Assumption \ref{Assumption_f_sigma_lipschitz}(ii), we obtain
    \begin{equation}
    \begin{split}
        &\int_t^{s'} \langle X(t',\omega), f(X(t',\omega),X(t')_{\texttt{\#}} \mathcal{L}^1,a(t',\omega)) \rangle_{-1} \mathrm{d}t'\\
        &\leq C \int_t^{s'} | X(t',\omega) |_{-1} \left ( 1 + |X(t',\omega) |_{-1} + \mathcal{M}_{-1,r}^{\frac{1}{r}}(X(t')_{\texttt{\#}} \mathcal{L}^1) + |a(t',\omega)|_{\Lambda} \right ) \mathrm{d}t'\\
        &\leq C \left ( 1 + \int_t^{s'} | X(t',\omega) |_{-1}^2 \mathrm{d}t' + \int_t^{s'} \mathcal{M}_{-1,r}^{\frac{2}{r}}(X(t')_{\texttt{\#}} \mathcal{L}^1) \mathrm{d}t' + \int_t^{s'} |a(t',\omega)|^2_{\Lambda} \mathrm{d}t' \right ).
    \end{split}
    \end{equation}
    Moreover, by Assumption \ref{Assumption_f_sigma_lipschitz}(iv), we have
    \begin{equation}\label{estimate_sigma}
    \begin{split}
        &\int_t^{s'} \left | \sigma(X(t',\omega),X(t')_{\texttt{\#}} \mathcal{L}^1) \right |_{L_2(\Xi,H_{-1})}^2 \mathrm{d}t'\\
        &\leq C \left ( 1 + \int_t^{s'} |X(t',\omega)|_{-1}^2 \mathrm{d}t' + \int_t^{s'} \mathcal{M}_{-1,r}^{\frac{2}{r}}(X(t')_{\texttt{\#}} \mathcal{L}^1) \mathrm{d}t' \right ).
    \end{split}
    \end{equation}
    For the stochastic integral in \eqref{itos_formula}, we obtain using Burkholder--Davis--Gundy inequality
    \begin{equation}
    \begin{split}
        &\mathbb{E} \left [ \sup_{s'\in [t,s]} \left | \int_t^{s'} \langle X(t',\omega), \sigma(X(t',\omega),X(t')_{\texttt{\#}} \mathcal{L}^1) \mathrm{d}W(t') \rangle_{-1} \right | \right ]\\
        &\leq \mathbb{E} \left [ \left ( \int_t^s | X(t',\omega) |^2_{-1} | \sigma(X(t',\omega),X(t')_{\texttt{\#}} \mathcal{L}^1) |_{L_2(\Xi,H_{-1})}^2 \mathrm{d}t' \right )^{\frac12} \right ]\\
        &\leq \frac14 \mathbb{E} \left [ \sup_{t'\in [t,s]} | X(t',\omega) |_{-1}^2 \right ] + C \mathbb{E} \left [ \int_t^s | \sigma(X(t',\omega),X(t')_{\texttt{\#}} \mathcal{L}^1) |_{L_2(\Xi,H_{-1})}^2 \mathrm{d}t' \right ],
    \end{split}
    \end{equation}
    For the last term in the previous inequality, we use the same estimate as in \eqref{estimate_sigma}. Moreover, we note that
    \begin{equation}
        \mathcal{M}_{-1,r}^{\frac{2}{r}} (X(t')_{\texttt{\#}}\mathcal{L}^1) = \left ( \int_{H_{-1}} |x|_{-1}^r (X(t')_{\texttt{\#}} \mathcal{L}^1)(\mathrm{d}x) \right )^{\frac{2}{r}}  = \left ( \int_{\Omega} | X(t',\omega) |_{-1}^r \mathrm{d}\omega \right )^{\frac{2}{r}} \leq \int_{\Omega} | X(t',\omega)|_{-1}^2 \mathrm{d}\omega.
    \end{equation}
    Thus, taking the supremum over $s'\in [t,s]$ in equation \eqref{itos_formula}, taking the expectation, and integrating over $\omega\in\Omega$, we obtain
    \begin{equation}
    \begin{split}
        &\mathbb{E} \left [ \int_{\Omega} \sup_{s'\in [t,s]} |X(s',\omega) |_{-1}^2 \mathrm{d}\omega \right ]\\
        &\leq C \left ( 1 + \| X \|_{-1}^2 + \int_t^s \mathbb{E} \left [ \int_{\Omega} \sup_{t'\in [t,s']} |X(t',\omega)|_{-1}^2 \mathrm{d}\omega \right ] \mathrm{d}s' + \mathbb{E} \left [ \int_t^s \|a(t')\|_{\Lambda}^2 \mathrm{d}t' \right ] \right ).
    \end{split}
    \end{equation}
    Now, \eqref{Apriori_0_Infinite_State} follows from Gr\"onwall's inequality.

    The proofs of \eqref{Apriori_0_Infinite_State_strong}, \eqref{Apriori_2_Infinite_State}, \eqref{Apriori_1_Infinite_State} follow along the same lines as the proofs of \eqref{Apriori_0_Infinite_State} and Proposition \ref{Prop_Finite_A_Priori}.

\end{proof}

For the following lemma, we are going to use the same notation as in the previous lemma. Moreover, for $\lambda\in [0,1]$, $s\in [t,T]$, we set
\begin{align}
	&a_{\lambda}(s) := \lambda a_1(s) + (1-\lambda) a_0(s), && X_{\lambda} := \lambda X_1 + (1-\lambda)X_0\\
	&X^{\lambda}(s) := \lambda X_1(s) + (1-\lambda)X_0(s), && X_{\lambda}(s) := X(s;X_{\lambda},a_{\lambda}(\cdot)),
\end{align}
where $X(s;X_{\lambda},a_{\lambda}(\cdot))$ denotes the solution of equation \eqref{Lifted_State_Equation} at time $s$ with initial condition $X_{\lambda}$ and control $a_{\lambda}(\cdot)$.

\begin{lemma}\label{lem:Xlambda}
	Let Assumptions \ref{Assumption_A_maximal_dissipative} and \ref{Assumption_weak_B_condition} be satisfied. Moreover, let Assumption \ref{Assumption_f_sigma_lipschitz} be satisfied with $r=1$. Then, there is a constant $C\geq 0$, such that
	\begin{equation}
		\int_{\Omega} \mathbb{E} \left [ \sup_{s\in [t,T]} |X^{\lambda}(s,\omega) - X_{\lambda}(s,\omega) |_{-1} \right ] \mathrm{d}\omega \leq C \lambda (1-\lambda) \left ( \|X_1-X_0\|_{-1}^2 + \| a_0(\cdot) - a_1(\cdot) \|_{M^2(t,T;\mathcal{E})}^2 \right )
	\end{equation}
    for all $\lambda\in [0,1]$, $X_0,X_1\in E$, and $a_0(\cdot),a_1(\cdot)\in \Lambda_t$.
\end{lemma}

\begin{proof}
We remind that as before, in the proof we skip some technical details and assume that all the terms are well defined and have sufficient integrability to apply the necessary theorems.
    Let
	\begin{align}
		&\bar{X}_0(\theta) = X^{\lambda}(s) + \theta \lambda (X_0(s) - X_1(s)), && \bar{a}_0(\theta) = a_{\lambda}(s) + \theta \lambda (a_0(s) - a_1(s))\\
		&\bar{X}_1(\theta) = X^{\lambda}(s) + \theta (1-\lambda) (X_1(s) - X_0(s)), && \bar{a}_1(\theta) = a_{\lambda}(s) + \theta (1-\lambda) (a_1(s) - a_0(s)).
	\end{align}
    By Assumption \ref{Assumption_f_sigma_lipschitz}(iii), for every $s\in [t,T]$ and $\omega\in \Omega$, we have
    \begingroup\makeatletter\def\f@size{11}\check@mathfonts
	\begin{equation}\label{X_lambda_Estimate_B}
	\begin{split}
		&| \lambda \tilde{f}(X_1(s,\omega),X_1(s),a_1(s,\omega)) + (1-\lambda) \tilde{f}(X_0(s,\omega),X_0(s),a_0(s,\omega)) - \tilde{f}(X^{\lambda}(s,\omega),X^{\lambda}(s),a_{\lambda}(s,\omega)) |_{-1} \\
		&\leq \lambda(1-\lambda) \int_0^1 \left | ( D_x \tilde{f}(\bar{X}_1(\theta,\omega), \bar{X}_1(\theta),\bar{a}_1(\theta,\omega)) - D_x \tilde{f}(\bar{X}_0(\theta,\omega), \bar{X}_0(\theta),\bar{a}_0(\theta,\omega))) ( X_1(s,\omega) - X_0(s,\omega) ) \right |_{-1} \mathrm{d}\theta\\
		&\quad + \lambda(1-\lambda) \int_0^1 \left | ( D_X \tilde{f}(\bar{X}_1(\theta,\omega), \bar{X}_1(\theta),\bar{a}_1(\theta,\omega)) - D_X \tilde{f}(\bar{X}_0(\theta,\omega), \bar{X}_0(\theta),\bar{a}_0(\theta,\omega) ) ) (X_1(s) - X_0(s)) \right |_{-1} \mathrm{d}\theta\\
        &\quad + \lambda(1-\lambda) \int_0^1 \left | ( D_a \tilde{f}(\bar{X}_1(\theta,\omega), \bar{X}_1(\theta),\bar{a}_1(\theta,\omega)) - D_a \tilde{f}(\bar{X}_0(\theta,\omega), \bar{X}_0(\theta),\bar{a}_0(\theta,\omega) ) ) (a_1(s,\omega) - a_0(s,\omega)) \right |_{-1} \mathrm{d}\theta\\
		&\leq C \lambda (1-\lambda) \left ( |X_1(s,\omega) - X_0(s,\omega)|_{-1}^2 + \| X_1(s) - X_0(s) \|_{-1}^2 + |a_1(s,\omega)-a_0(s,\omega)|_{\Lambda}^2 \right )
	\end{split}
	\end{equation}\endgroup
	and similarly by Assumption \ref{Assumption_f_sigma_lipschitz}(v)
	\begin{equation}\label{X_lambda_Estimate_Sigma}
		\begin{split}
			&| \lambda \tilde{\sigma}(X_1(s,\omega),X_1(s)) + (1-\lambda) \tilde{\sigma}(X_0(s,\omega),X_0(s)) - \tilde{\sigma}(X^{\lambda}(s,\omega),X^{\lambda}(s)) |_{L_2(\Xi,H_{-1})} \\
			&\leq C \lambda (1-\lambda) \left ( |X_1(s,\omega) - X_0(s,\omega)|_{-1}^2 + \| X_1(s) - X_0(s) \|_{-1}^2 \right ).
		\end{split}
	\end{equation}
    As in the proof of Lemma \ref{Infinite_Dimensional_A_Priori}, we see that for almost every $\omega\in\Omega$, $(X_{\lambda}(s,\omega))_{s\in[t,T]}$ is a mild solution of
    \begin{equation}
    \begin{cases}
        \mathrm{d}X_{\lambda}(s,\omega) = [ A X_{\lambda}(s,\omega) + \tilde{f}(X_{\lambda}(s,\omega),X_{\lambda}(s),a_{\lambda}(s,\omega)) ] \mathrm{d}s + \tilde{\sigma}(X_{\lambda}(s,\omega),X_{\lambda}(s)) \mathrm{d}W(s)\\
        X_{\lambda}(t,\omega) = X_{\lambda}(\omega)
    \end{cases}
    \end{equation}
    and $(X^{\lambda}(s,\omega))_{s\in[t,T]}$ is a mild solution of
    \begin{equation}
    \begin{cases}
        \mathrm{d}X^{\lambda}(s,\omega) = [ AX^{\lambda}(s,\omega) + \lambda \tilde{f}(X_1(s,\omega),X_1(s), a_1(s,\omega) + (1-\lambda) \tilde{f}(X_0(s,\omega), X_0(s), a_0(s,\omega)) ] \mathrm{d}s\\
        \qquad\qquad\qquad + [ \lambda \tilde{\sigma}(X_1(s,\omega),X_1(s)) + (1-\lambda) \tilde{\sigma}(X_0(s,\omega), X_0(s)) ] \mathrm{d}W(s)\\
        X^{\lambda}(t,\omega) = X_{\lambda}(\omega).
    \end{cases}
    \end{equation}
    Thus, applying \cite[Proposition 1.164]{fabbri_gozzi_swiech_2017} (recall Remark \ref{rem:Itoformulas}) and Assumption \ref{Assumption_weak_B_condition} yields
    \begingroup\makeatletter\def\f@size{11}\check@mathfonts
    \begin{equation}\label{ito_X_lambda}
    \begin{split}
        &| X^{\lambda}(s',\omega) - X_{\lambda}(s',\omega) |_{-1}^2 \leq 2 c_0 \int_t^{s'} | X^{\lambda}(t',\omega) - X_{\lambda}(t',\omega) |_{-1}^2 \mathrm{d}t'\\
        &+ 2 \int_t^{s'} \big \langle \lambda \tilde{f}(X_1(t',\omega), X_1(t'),a_1(t')) + (1-\lambda) \tilde{f}(X_0(t',\omega), X_0(t'),a_0(t'))\\
        &\qquad\qquad - \tilde{f}(X^{\lambda}(t',\omega),X^{\lambda}(t'),a_{\lambda}(t')),  X^{\lambda}(t',\omega) - X_{\lambda}(t',\omega) \big \rangle_{-1} \mathrm{d}t'\\
        &+ 2 \int_t^{s'} \langle \tilde{f}(X^{\lambda}(t',\omega),X^{\lambda}(t'),a_{\lambda}(t')) - \tilde{f}(X_{\lambda}(t',\omega),X_{\lambda}(t'),a_{\lambda}(t')), X^{\lambda}(t',\omega) - X_{\lambda}(t',\omega) \rangle_{-1} \mathrm{d}t'\\
        &+ \int_t^{s'} | \lambda \tilde{\sigma}(X_1(t',\omega),X_1(t')) + (1-\lambda) \tilde{\sigma}(X_0(t',\omega),X_0(t')) - \tilde{\sigma}( X^{\lambda}(t',\omega),X^{\lambda}(t')) |_{L_2(\Xi,H_{-1})}^2 \mathrm{d}t'\\
        &+ \int_t^{s'} | \tilde{\sigma}(X^{\lambda}(t',\omega),X^{\lambda}(t')) - \tilde{\sigma}(X_{\lambda}(t',\omega),X_{\lambda}(t')) |_{L_2(\Xi,H_{-1})}^2 \mathrm{d}t'\\
        &+ 2 \int_t^{s'} \langle X^{\lambda}(t',\omega) - X_{\lambda}(t',\omega), \big ( \lambda \tilde{\sigma}(X_1(t',\omega),X_1(t'))\\
        &\qquad\qquad + (1-\lambda) \tilde{\sigma}(X_0(t',\omega),X_0(t')) - \tilde{\sigma}(X^{\lambda}(t',\omega),X^{\lambda}(t')) \big ) \mathrm{d}W(t') \rangle_{-1}\\
        &+ 2 \int_t^{s'} \langle X^{\lambda}(t',\omega) - X_{\lambda}(t',\omega), (\tilde{\sigma}(X^{\lambda}(t',\omega),X^{\lambda}(t')) - \tilde{\sigma}(X_{\lambda}(t',\omega),X_{\lambda}(t'))) \mathrm{d}W(t') \rangle_{-1}.
    \end{split}
    \end{equation}\endgroup
    For the first term on the right-hand side of this inequality, we have 
    \[
    \begin{split}
   &\sup_{s'\in [t,s]}   \int_t^{s'} | X^{\lambda}(t',\omega) - X_{\lambda}(t',\omega) |_{-1}^2 \mathrm{d}t'\\
   &\leq   \varepsilon \sup_{t'\in [t,s]}  | X^{\lambda}(t',\omega) - X_{\lambda}(t',\omega) |_{-1}^2 +C_\varepsilon \left(\int_t^{s} \sup_{t'\in [t,s]} | X^{\lambda}(t',\omega) - X_{\lambda}(t',\omega) |_{-1} \mathrm{d}s \right)^2
   \end{split}
   \]
     for all $\varepsilon>0$ and some constant $C_{\varepsilon}\geq 0$ that depends on $\varepsilon$. For the second term, we note that
    \begin{equation}
    \begin{split}
        &\sup_{s'\in [t,s]} \bigg | \int_t^{s'} \big \langle \lambda \tilde{f}(X_1(t',\omega), X_1(t'),a_1(t')) + (1-\lambda) \tilde{f}(X_0(t',\omega), X_0(t'),a_0(t'))\\
        &\qquad\qquad\qquad - \tilde{f}(X^{\lambda}(t',\omega),X^{\lambda}(t'),a_{\lambda}(t')), X^{\lambda}(t',\omega) - X_{\lambda}(t',\omega) \big \rangle_{-1} \mathrm{d}t' \bigg |\\
        &\leq \varepsilon \sup_{t'\in [t,s]} \big | X^{\lambda}(t',\omega) - X_{\lambda}(t',\omega) \big |_{-1}^2\\
        &\quad + C_{\varepsilon} \sup_{t'\in [t,s]} \big | \lambda \tilde{f}(X_1(t',\omega), X_1(t'),a_1(t')) + (1-\lambda) \tilde{f}(X_0(t',\omega), X_0(t'),a_0(t'))\\
        &\qquad\qquad\qquad - \tilde{f}(X^{\lambda}(t',\omega),X^{\lambda}(t'),a_{\lambda}(t')) \big |_{-1}^2.
    \end{split}
    \end{equation}
  Thus, using \eqref{X_lambda_Estimate_B}, we obtain
    \begin{equation}
    \begin{split}
        &\sup_{s'\in [t,s]} \bigg | \int_t^{s'} \big \langle \lambda \tilde{f}(X_1(t',\omega), X_1(t'),a_1(t')) + (1-\lambda) \tilde{f}(X_0(t',\omega), X_0(t'),a_0(t'))\\
        &\qquad\qquad\qquad - \tilde{f}(X^{\lambda}(t',\omega),X^{\lambda}(t'),a_{\lambda}(t')), X^{\lambda}(t',\omega) - X_{\lambda}(t',\omega) \big \rangle_{-1} \mathrm{d}t' \bigg |\\
        &\leq \varepsilon \sup_{t'\in [t,s]} \big | X^{\lambda}(t',\omega) - X_{\lambda}(t',\omega) \big |_{-1}^2\\
        &\quad + C_{\varepsilon}^2 \lambda^2 (1-\lambda)^2 \sup_{t'\in [t,s]} \left ( |X_1(t',\omega) - X_0(t',\omega)|_{-1}^4 + \| X_1(t') - X_0(t') \|_{-1}^4 \right )\\
        &\quad + C_{\varepsilon}^2 \lambda^2 (1-\lambda)^2 \left ( \int_t^s |a_1(t',\omega)-a_0(t',\omega)|_{\Lambda}^2 \mathrm{d}t' \right )^2.
    \end{split}
    \end{equation}
    Using Assumption \ref{Assumption_f_sigma_lipschitz}(i), and then taking the absolute value and the supremum, we obtain for the third term on the right-hand side of inequality \eqref{ito_X_lambda}
    \begin{equation}
    \begin{split}
        &{\int_t^{s'} \langle \tilde{f}(X^{\lambda}(t',\omega),X^{\lambda}(t'),a_{\lambda}(t')) - \tilde{f}(X_{\lambda}(t',\omega),X_{\lambda}(t'),a_{\lambda}(t')), X^{\lambda}(t',\omega) - X_{\lambda}(t',\omega) \rangle_{-1} \mathrm{d}t'} \\
        &\leq \varepsilon \sup_{s'\in [t,s]} | X^{\lambda}(t',\omega) - X_{\lambda}(t',\omega) |_{-1}^2\\
        &\quad + C_{\varepsilon} \left ( \int_t^s \left ( |X^{\lambda}(t',\omega) - X_{\lambda}(t',\omega) |_{-1} + d_{-1,1}(X^{\lambda}(t')_{\texttt{\#}} \mathcal{L}^1,X_{\lambda}(t')_{\texttt{\#}} \mathcal{L}^1 ) \right ) \mathrm{d}t' \right )^2,
    \end{split}
    \end{equation}
    for all $\varepsilon>0$. For the fourth term on the right-hand side of inequality \eqref{ito_X_lambda}, using \eqref{X_lambda_Estimate_Sigma}, we obtain
    \begin{equation}\label{estimate_sigma_tilde}
    \begin{split}
        &\int_t^{s'} | \lambda \tilde{\sigma}(X_1(t',\omega),X_1(t')) + (1-\lambda) \tilde{\sigma}(X_0(t',\omega),X_0(t')) - \tilde{\sigma}( X^{\lambda}(t',\omega),X^{\lambda}(t')) |_{L_2(\Xi,H_{-1})}^2 \mathrm{d}t'\\
        &\leq C^2 \lambda^2 (1-\lambda)^2 \sup_{t'\in [t,s]} \left ( |X_1(t',\omega) - X_0(t',\omega)|_{-1}^4 + \| X_1(t') - X_0(t') \|_{-1}^4 \right )
    \end{split}
    \end{equation}
    Moreover, by Assumption \ref{Assumption_f_sigma_lipschitz}(iv), we have for the fifth term on the right-hand side of inequality \eqref{ito_X_lambda}
    \begin{equation}
    \begin{split}
        &\sup_{s'\in [t,s]} \int_t^{s'} | \tilde{\sigma}(X^{\lambda}(t',\omega),X^{\lambda}(t')) - \tilde{\sigma}(X_{\lambda}(t',\omega),X_{\lambda}(t')) |_{L_2(\Xi,H_{-1})}^2 \mathrm{d}t'\\
        &\leq \int_t^s \left ( | X^{\lambda}(t',\omega) - X_{\lambda}(t',\omega) |_{-1}^2 + d^2_{-1,1}(X^{\lambda}(t')_{\texttt{\#}} \mathcal{L}^1, X_{\lambda}(t')_{\texttt{\#}} \mathcal{L}^1) \right ) \mathrm{d}t'.
    \end{split}
    \end{equation}
    Next, for the sixth term on the right-hand side of inequality \eqref{ito_X_lambda}, we obtain using Burkholder--Davis--Gundy inequality
    \begin{equation}
    \begin{split}
        &\mathbb{E} \bigg [ \sup_{s'\in [t,s]} \bigg | \int_t^{s'} \langle X^{\lambda}(t',\omega) - X_{\lambda}(t',\omega),\\
        &\quad ( \lambda \tilde{\sigma}(X_1(t',\omega),X_1(t')) + (1-\lambda) \tilde{\sigma}(X_0(t',\omega),X_0(t')) - \tilde{\sigma}(X^{\lambda}(t',\omega),X^{\lambda}(t')) ) \mathrm{d}W(t') \rangle_{-1} \bigg |^{\frac12} \bigg ]\\
        &\leq C \mathbb{E} \bigg [ \bigg ( \int_t^s | X^{\lambda}(t',\omega) - X_{\lambda}(t',\omega) |_{-1}^2\\
        &\quad | \lambda \tilde{\sigma}(X_1(t',\omega),X_1(t')) + (1-\lambda) \tilde{\sigma}(X_0(t',\omega),X_0(t')) - \tilde{\sigma}(X^{\lambda}(t',\omega),X^{\lambda}(t')) |_{L_2(\Xi,H_{-1})}^2 \mathrm{d}t' \bigg )^{\frac14} \bigg ]\\
        &\leq \varepsilon \mathbb{E} \left [ \sup_{t'\in [t,s]} | X^{\lambda}(t',\omega) - X_{\lambda}(t',\omega) |_{-1} \right ] + C_{\varepsilon} \mathbb{E} \Bigg [ \Bigg ( \int_t^s | \lambda \tilde{\sigma}(X_1(t',\omega),X_1(t'))\\
        &\qquad\qquad + (1-\lambda) \tilde{\sigma}(X_0(t',\omega),X_0(t')) - \tilde{\sigma}(X^{\lambda}(t',\omega),X^{\lambda}(t')) |_{L_2(\Xi,H_{-1})}^2 \mathrm{d}t' \Bigg )^{\frac12} \Bigg ],
    \end{split}
    \end{equation}
    for all $\varepsilon>0$. The second term on the right-hand side of this inequality can be estimated as in \eqref{estimate_sigma_tilde}. Finally, for the last term in inequality \eqref{ito_X_lambda}, we obtain again using Burkholder--Davis--Gundy inequality
    \begin{equation}
    \begin{split}
        &\mathbb{E} \left [ \sup_{s'\in [t,s]} \left | \int_t^{s'} \langle X^{\lambda}(t',\omega) - X_{\lambda}(t',\omega), (\tilde{\sigma}(X^{\lambda}(t',\omega),X^{\lambda}(t')) - \tilde{\sigma}(X_{\lambda}(t',\omega),X_{\lambda}(t'))) \mathrm{d}W(t') \rangle_{-1} \right |^{\frac12} \right ]\\
        &\leq \varepsilon \mathbb{E} \left [ \sup_{t'\in [t,s]} | X^{\lambda}(t',\omega) - X_{\lambda}(t',\omega) |_{-1} \right ]\\
        &\quad + C_{\varepsilon} \mathbb{E} \left [ \left ( \int_t^s \left ( | X^{\lambda}(t',\omega) - X_{\lambda}(t',\omega) |_{-1}^2 + d^2_{-1,1}(X^{\lambda}(t')_{\texttt{\#}} \mathcal{L}^1, X_{\lambda}(t')_{\texttt{\#}} \mathcal{L}^1) \right ) \mathrm{d}t' \right )^{\frac12} \right ],
    \end{split}
    \end{equation}
    for all $\varepsilon>0$. We note that
    \begin{equation}
    \begin{split}
        &\int_t^s d^2_{-1,1}(X^{\lambda}(t')_{\texttt{\#}} \mathcal{L}^1,X_{\lambda}(t')_{\texttt{\#}} \mathcal{L}^1 ) \mathrm{d}t' \leq { \int_t^s \left ( \int_{\Omega} | X^{\lambda}(t',\omega) - X_{\lambda}(t',\omega) |_{-1} \mathrm{d}\omega \right )^2 \mathrm{d}t' } \\
        &\leq \varepsilon \left ( \int_{\Omega} \sup_{t'\in [t,s]} | X^{\lambda}(t',\omega) - X_{\lambda}(t',\omega) |_{-1} \mathrm{d}\omega \right )^2 + C_{\varepsilon} \left ( \int_{\Omega} \int_t^s | X^{\lambda}(t',\omega) - X_{\lambda}(t',\omega) |_{-1} \mathrm{d}t' \mathrm{d}\omega \right )^2,
    \end{split}
    \end{equation}
    for all $\varepsilon>0$. Thus, combining these estimates with \eqref{ito_X_lambda}, choosing $\varepsilon>0$ sufficiently small, taking the square root, the supremum over $s'\in [t,s]$, the expectation, and integrating over $\omega\in \Omega$ we obtain
    \begin{equation}
    \begin{split}
        & \int_{\Omega} \mathbb{E} \left [ \sup_{s'\in [t,s]} | X^{\lambda}(s',\omega) - X_{\lambda}(s',\omega) |_{-1} \right ] \mathrm{d}\omega \\
        &\leq C_{\varepsilon} \int_t^s \int_{\Omega} \mathbb{E} \left [ \sup_{t'\in [t,s']} | X^{\lambda}(t',\omega) - X_{\lambda}(t',\omega) |_{-1} \mathrm{d}s' \right ] \mathrm{d}\omega \\
        &\quad + C_{\varepsilon} \lambda (1-\lambda) \int_{\Omega} \mathbb{E} \left [ \sup_{t'\in [t,s]} \left ( |X_1(t',\omega) - X_0(t',\omega)|_{-1}^2 + \| X_1(t') - X_0(t') \|_{-1}^2 \right ) \right ] \mathrm{d}\omega\\
        &\quad + C_{\varepsilon} \lambda (1-\lambda) \int_{\Omega} \mathbb{E} \left [ \int_t^s |a_1(t',\omega)-a_0(t',\omega)|_{\Lambda}^2 \mathrm{d}t' \right ] \mathrm{d}\omega.
    \end{split}
    \end{equation}
    Noting that
    \begin{equation}
        \mathbb{E} \left [ \sup_{t'\in [t,s]} \| X_1(t') - X_0(t') \|_{-1}^2 \right ] \leq \int_{\Omega} \mathbb{E} \left [ \sup_{t'\in [t,s]} |X_1(t',\omega) - X_0(t',\omega)|_{-1}^2 \right ] \mathrm{d}\omega
    \end{equation}
    and applying Gr\"onwall's inequality and Lemma \ref{Infinite_Dimensional_A_Priori} concludes the proof.
\end{proof}

\subsection{\texorpdfstring{$C^{1,1}$}{C{1,1}} Regularity of the Value Function}
\label{subsection:C11_regularity}

\begin{prop}\label{Value_Function_Lipschitz}
	Let Assumptions \ref{Assumption_A_maximal_dissipative}, \ref{Assumption_weak_B_condition}, \ref{Assumption_f_sigma_lipschitz}(i)(ii)(iv) and \ref{Assumption_running_terminal_cost}(i)(ii)(iv) be satisfied. Then, there is a constant $C\geq 0$, depending only on the Lipschitz constants of $b$, $\sigma$, $l$ and $\mathcal{U}_T$ as well as on $T$, such that
	\begin{equation}\label{Lipschitz_U}
		| U(t,X) - U(t,Y) | \leq C \|X-Y\|_{-1}
	\end{equation}
	for all $t\in [0,T]$ and $X,Y\in E$. Moreover, there is a constant $K>0$ such that for all $t\in [0,T], X\in E$,
	\[
	U(t,X) = \inf_{a(\cdot)\in \Lambda_t^K} J(t,X;a(\cdot)),
    \]
    where $\Lambda_t^K=\{a(\cdot)\in \Lambda_t: a(\cdot)\,\,\mbox{has values in}\,\, B_K(0)\,\,\mbox{in}\,\,{\mathcal E}\}$, $U$ is uniformly continuous on bounded subsets of $[0,T]\times E$, and $U$ is the unique $\mathcal B$-continuous viscosity solution of equation \eqref{lifted_HJB_equation} as well as equation \eqref{eq:eqHJBK1}. Here, uniqueness holds in the class of functions $W$ which satisfy for some $C\geq 0$,
    \begin{equation}
        |W(t,X) - W(t,Y)| \leq C \| X-Y\|,\quad \text{for all } t\in [0,T], \, X,Y\in E.
    \end{equation}
\end{prop}

\begin{proof}
The proof follows using the same arguments as the proof of Proposition \ref{Lipschitz_Time_u_n} if we use estimates of Lemma \ref{Infinite_Dimensional_A_Priori}. Uniqueness of $\mathcal B$-continuous viscosity solutions is showed in the proof of Theorem \ref{theorem:convergence}.
\end{proof}

\begin{prop}[Semiconcavity: Case 1]\label{Value_Function_Semiconcave_1}
Let Assumptions \ref{Assumption_A_maximal_dissipative} and \ref{Assumption_weak_B_condition} be satisfied. Moreover, let Assumptions \ref{Assumption_f_sigma_lipschitz} and \ref{Assumption_running_terminal_cost} be satisfied with $r=1$. Then, for every $t\in [0,T]$, $U(t,\cdot)$ is semiconcave with respect to the $E_{-1}$-norm, i.e., there is a constant $C\geq 0$ such that
	\begin{equation}
		\lambda U(t,X) + (1-\lambda) U(t,Y) - U(t,\lambda X + (1-\lambda) Y) \leq C \lambda (1-\lambda) \| X- Y \|_{-1}^2
	\end{equation}
	for all $\lambda \in [0,1]$ and $X,Y\in E$. Moreover, the semiconcavity constant $C$ is independent of $t\in [0,T]$.
\end{prop}

\begin{proof}
The proof uses the same techniques as the proof of semiconcavity in \cite[Proposition 6.1(ii)]{mayorga_swiech_2023} and \cite[Theorems 3.11, 5.10]{defeo_swiech_wessels_2023}, using Lemmas \ref{Infinite_Dimensional_A_Priori} and \ref{lem:Xlambda}.
\end{proof}

\begin{prop}[Semiconcavity: Case 2]\label{Value_Function_Semiconcave_2}
Let Assumptions \ref{Assumption_A_maximal_dissipative}, \ref{Assumption_weak_B_condition}, \ref{Assumption_running_terminal_cost} and \ref{Assumption_Linear_State_Equation} be satisfied. Then, for every $t\in [0,T]$, $U(t,\cdot)$ is semiconcave with respect to the $E_{-1}$-norm with the semiconcavity constant independent of $t\in [0,T]$.
\end{prop}

\begin{proof}
	The proof uses again the same techniques as the corresponding result in \cite[Proposition 6.2]{mayorga_swiech_2023}. We note that now we have $X_\lambda=X^\lambda$, so Lemma \ref{lem:Xlambda} is not needed and we do not need $r=1$.
\end{proof}

\begin{prop}[Semiconvexity: Case 1]\label{Value_Function_Semiconvex_1}
	Let Assumptions \ref{Assumption_A_maximal_dissipative}, \ref{Assumption_weak_B_condition} and \ref{Assumption_running_cost_nu} be satisfied. Moreover, let Assumptions \ref{Assumption_f_sigma_lipschitz} and \ref{Assumption_running_terminal_cost} be satisfied with $r=1$. There is a constant $\nu_0\geq 0$ such that if $\nu$ in Assumption \ref{Assumption_running_cost_nu} satisfies $\nu \geq \nu_0$, then, for every $t\in [0,T]$, $U(t,\cdot)$ is semiconvex with respect to the $E_{-1}$-norm, i.e., there is a constant $C\geq 0$ such that
    \begin{equation}
        \lambda U(t,X) + (1-\lambda) U(t,Y) - U(t,\lambda X + (1-\lambda) Y) \geq - C \lambda (1-\lambda) \| X- Y \|_{-1}^2
    \end{equation}
    for all $\lambda \in [0,1]$ and $X,Y\in E$. Moreover, the semiconvexity constant $C$ is independent of $t\in [0,T]$. 
\end{prop}

\begin{proof}
	The proof follows along the same lines as the proofs of the corresponding results \cite[Proposition 6.1(iii)]{mayorga_swiech_2023} and \cite[Theorems 3.13, 5.12]{defeo_swiech_wessels_2023}, now using Lemmas \ref{Infinite_Dimensional_A_Priori} and \ref{lem:Xlambda}.
\end{proof}

\begin{prop}[Semiconvexity: Case 2]\label{Value_Function_Semiconvex_2}
	Let Assumptions \ref{Assumption_A_maximal_dissipative}, \ref{Assumption_weak_B_condition}, \ref{Assumption_running_terminal_cost} and \ref{Assumption_Linear_State_Equation} be satisfied. Then, for every $t\in [0,T]$, $U(t,\cdot)$ is convex.
\end{prop}

\begin{proof}
	The proof again repeats the steps of the proof of the corresponding result in \cite[Proposition 6.2]{mayorga_swiech_2023}. We note again that we have $X_\lambda=X^\lambda$, so Lemma \ref{lem:Xlambda} is not needed and we do not need $r=1$.
\end{proof}

We remark that, since for fixed $t\in [0,T]$, the value function $U(t,\cdot)$ is uniformly Lipschitz, semiconcave and semiconvex with respect to the $E_{-1}$-norm, it extends to a function on $E_{-1}$ which is in $C^{1,1}(E_{-1})$, see \cite{lasry_lions_1986}. Thus we have the following theorem.
\begin{theorem}\label{th:C11}
    If either the assumptions of Propositions \ref{Value_Function_Semiconcave_1} and \ref{Value_Function_Semiconvex_1} are satisfied, or the assumptions of Propositions \ref{Value_Function_Semiconcave_2} and \ref{Value_Function_Semiconvex_2} are satisfied, then $U(t,\cdot)\in C^{1,1}(E_{-1})$, for all $t\in [0,T]$ and the Lipschitz constant of ${D_{-1}U}(t,\cdot)$ is independent of $t$. 
\end{theorem} 

\section{Projection of \texorpdfstring{$V$}{V} onto \texorpdfstring{$u_n$}{$u_n$}}\label{sec:projection} In this section, generalizing the finite-dimensional result \cite[Theorem 5.4]{swiech_wessels_2024} (i.e. $H=\mathbb R^d$) to the Hilbert space case, we  prove in Theorem \ref{theorem:projection_C11} that, under suitable regularity of the value function $U\equiv V$ of the lifted limit control problem, then it projects onto the value function $u_n$ of the particle system control problem.  The proof of the result will require a careful adaptation of the arguments in \cite{swiech_wessels_2024} to the infinite-dimensional case on $H$.

For $n\in\mathbb{N}$, $i \in \{1,\dots,n\}$, recall that $A_i^n = \left (\frac{i-1}{n},\frac{i}{n} \right ) \subset \Omega$,  $X^{\mathbf{x}}_n = \sum_{i=1}^n x_i \mathbf{1}_{A^n_i}$, and define 
\begin{equation}\label{Projection_V_n}
	V_n : [0,T]\times H^n \to \mathbb{R},\quad V_n(t,\mathbf{x}) := V\left ( t, X^{\mathbf{x}}_n \right ) {= \mathcal V(t,\mu_{\mathbf x})}.
\end{equation}
\begin{theorem}\label{theorem:projection_C11}
 Let Assumptions \ref{Assumption_A_maximal_dissipative},  \ref{Assumption_weak_B_condition}, \ref{Assumption_B_compact}, \ref{Assumption_f_sigma_lipschitz}(i)(ii)(iv), and \ref{Assumption_running_terminal_cost}(i)(ii)(iv) be satisfied. Moreover, assume that  $D_{-1} V(t,\cdot):  E_{-1}\to E_{-1}$ is uniformly continuous, uniformly in $t \in [0,T]$.  Then, for all $(t,\mathbf{x}) \in [0,T] \times H^n$, it holds
$ u_n(t,\mathbf{x})=V_n(t,\mathbf{x}) {=} {\mathcal V(t,\mu_{\mathbf x})}.$
\end{theorem}
Notice that if the assumptions of Theorem \ref{theorem:convergence} (which are included in Theorem \ref{theorem:projection_C11}) and the assumptions of Theorem \ref{th:C11} are satisfied, $D_{-1}V(=D_{-1}U)$ satisfies the above continuity requirement.
\subsection{$V_n\leq u_n$}
In this subsection, we are going to show the first inequality.

In order to prove it, we first show  the following result.
\begin{lemma}\label{lemma_subsolution}
	Let Assumptions \ref{Assumption_A_maximal_dissipative}, \ref{Assumption_f_sigma_lipschitz}(i)(ii)(iv), and \ref{Assumption_running_terminal_cost}(i)(iv), be satisfied. Then:
    \begin{enumerate}
        \item  $X^{n}(s,\omega) = X^{\mathbf{x}(s)}_n(\omega),$  for a.e. $\omega\in \Omega$, where $X^{n}(\cdot)$ is the solution of the lifted state equation \eqref{Lifted_State_Equation} with initial condition $X^{\mathbf{x}}_n$ and control $a^n(\cdot) := \sum_{i=1}^n a_i(\cdot) \mathbf{1}_{A^n_i}$ and $X^{\mathbf{x}(s)}_n:= \sum_{i=1}^n x_i(s) \mathbf{1}_{A^n_i}$ is the lift of $\mathbf{x}(s)$. 
        \item  $J_n(t,\mathbf{x};\mathbf{a}(\cdot)) = J(t,X^{\mathbf{x}}_n;a^n(\cdot))$, for every control $\mathbf{a}(\cdot)\in \mathcal{A}^n_t$.
    \end{enumerate}
\end{lemma}
 \begin{proof}
(1) Note that $\mathbb{P}$-a.s., for every $s\in [t,T]$,
\begin{equation}
	X^{n}(s) =e^{\mathcal A (s-t)} \sum_{i=1}^n x_i \mathbf{1}_{A^n_i}  +\int_t^s e^{\mathcal A (s-r)} F\left (X^{n}(r),\sum_{i=1}^n a_i(r) \mathbf{1}_{A^n_i} \right) \mathrm{d}r + \int_t^s e^{\mathcal A (s-r)} \Sigma(X^{n}(r)) \mathrm{d}W(r);
\end{equation}
therefore, for almost every $\omega\in A^n_i = (\frac{i-1}{n},\frac{i}{n})$, using \eqref{semigroups_coincide}, we have
\begin{small}
\begin{equation}
	X^{n}(s,\omega) = e^{A (s-t)} x_i + \int_t^s e^{A (s-r)} f(X^{n}(r,\omega),X^{n}(r)_{\texttt{\#}}\mathcal{L}^1,a_i(r)) \mathrm{d}r + \int_t^s e^{A (s-r)}  \sigma(X^{n}(r,\omega),X^{n}(r)_{\texttt{\#}}\mathcal{L}^1) \mathrm{d}W(r).
\end{equation}
\end{small}
On the other hand, as in \eqref{eq:Xnxpushforw=mux}, we have $\mu_{\mathbf{x}(s)} = (X_n^{\mathbf{x}(s)})_{\texttt{\#}}\mathcal{L}^1$, so that $X^{\mathbf{x}(s)}_n(\omega)$, $\omega \in A_i^n$, also satisfies the previous equation. Therefore, by uniqueness of mild solutions, we have the claim.

(2) Using  (1),  we have
\begin{equation}
\begin{split}
	J(t,X^{\mathbf{x}}_n;a^n(\cdot)) &= \mathbb{E} \left [ \int_t^T  L(X^{n}(s), a^n(s)) \mathrm{d}s + U_T(X^{n}(T)) \right ]\\
	&= \mathbb{E} \left [ \int_t^T  L(X^{\mathbf{x}(s)}_n, a^n(s)) \mathrm{d}s + U_T(X^{\mathbf{x}(T)}_n) \right ]\\
	&= \mathbb{E} \left [ \int_t^T \int_{\Omega} l(X_n^{\mathbf{x}(s)}(\omega),(X_n^{\mathbf{x}(s)})_{\texttt{\#}}\mathcal{L}^1,a^n(s,\omega)) \mathrm{d}\omega  \mathrm{d}s + U_T(X^{\mathbf{x}(T)}_n) \right ]\\
	&= \mathbb{E} \left [ \int_t^T  \frac{1}{n} \sum_{i=1}^n l(x_i(s),\mu_{\mathbf{x}(s)},a_i(s))   \mathrm{d}s + \frac{1}{n} \sum_{i=1}^n {\mathcal U}_T(x_i(T),\mu_{{\bf x}(T)}) \right ]= J_n(t,\mathbf{x};\mathbf{a}(\cdot)),
\end{split}
\end{equation}
which concludes the proof.
\end{proof}

 \begin{lemma}\label{lemma:V_n_leq_u_n}
	Let Assumptions \ref{Assumption_A_maximal_dissipative}, \ref{Assumption_weak_B_condition}, \ref{Assumption_B_compact}, \ref{Assumption_f_sigma_lipschitz}(i)(ii)(iv), and \ref{Assumption_running_terminal_cost}(i)(ii)(iv), be satisfied. Then, for all $(t,\mathbf{x}) \in [0,T] \times H^n$, it holds
$V_n(t,\mathbf{x}) \leq u_n(t,\mathbf{x})$.
\end{lemma}
\begin{proof}
We take $\inf_{\mathbf{a}(\cdot)\in \Lambda^n_t}$ in Lemma \ref{lemma_subsolution}, to obtain
\[
u_n(t,\mathbf{x})=\inf_{\mathbf{a}(\cdot)\in \Lambda^n_t} J_n(t,\mathbf{x};\mathbf{a}(\cdot))=\inf_{\mathbf{a}(\cdot)\in \Lambda^n_t} J(t,X^{\mathbf{x}}_n;a^n(\cdot)) \geq \inf_{a(\cdot)\in \Lambda_t} J(t,X^{\mathbf{x}}_n,a(\cdot))= V(t,X^{\mathbf{x}}_n)=V_n(t,\mathbf{x}).
\]
\end{proof}
We remark that the proof of the inequality $V_n\leq u_n$ does not require any differentiability of $V$.

\subsection{$V_n \geq u_n$}

Next, let us show that $U$ is a $\mathcal B$-continuous viscosity solution of equation \eqref{lifted_HJB_equation} even when using a larger class of test functions than what is usually used in  standard theory \cite[Chapter 3]{fabbri_gozzi_swiech_2017} (see also Appendix \ref{app:viscosity_hilbert}). This larger class of test functions will be needed in the proof of Lemma \ref{proposition:V_n_geq_u_n}. In order to define this class of test functions, let us introduce some notation. 
\begin{notation}\label{not:notations_Pn}
Let $f_i:= \sqrt{n} \mathbf{1}_{A^n_i}$, $i=1,\dots,n$, and extend this orthonormal set in $L^2(\Omega)$ to an orthonormal basis denoted by $(f_i)_{i\in \mathbb{N}}$. Let $F_n \subset L^2(\Omega)$ be the subspace spanned by $f_i$, $i=1,\dots n$, and let $E_n = F_n \otimes H \subset E$. Let $E_n^{\perp}$ denote its orthogonal complement, i.e., $E=E_n\oplus E_n^{\perp}$. Moreover, let $P_n$ denote the orthogonal projection onto $E_n$ and let $P_n^{\perp} := I - P_n$. 

Notice that, each element $X\in E$ can be written as
    $X = P_n X + P_n^{\perp} X = \sum_{i=1}^n y_i f_i + \sum_{i=n+1}^{\infty} y_i f_i = \sum_{i=1}^n \sum_{k=1}^\infty y_i^k f_i e_k + \sum_{i=n+1}^{\infty} \sum_{k=1}^\infty y_i^k f_i e_k,$
where $y_i = (y_i^1, y_i^2, \dots) \in H$, $i\geq 1$ and   $(e_k)_{k\in \mathbb N}$ is an orthonormal basis of $H$. Let us denote $\mathbf{y}=(y_1,\dots,y_n) \in H^n$ and $\bar{\mathbf{y}} = (y_i)_{i\geq n+1}$.  In particular, for $\mathbf{x} \in H^n$, we have $E_n \ni X^{\mathbf{x}}_n = \sum_{i=1}^n x_i \mathbf{1}_{A^n_i}=\sum_{i=1}^n \frac 1 {\sqrt{n}}  x_i f_i$, so the coefficients $\mathbf{y}$ in the basis representation of $X^{\mathbf{x}}_n$  are given by $\mathbf{y}= \mathbf{x}/\sqrt{n}$ and $\bar{\mathbf{y}}=\mathbf{0}$. 
\end{notation}
\begin{remark}
   We note that if $L\in L(H)$, and we define $\mathcal L\in L(E)$, $(\mathcal LX)(\omega):=LX(\omega)$, we have
    \begin{align}
        &P_n \mathcal LX=\mathcal LP_nX=P_n \mathcal LP_nX\label{eq:PnL}\\
        &\mathcal LP_n^\perp X=\mathcal LP_n^\perp X=P_n^\perp \mathcal LP_n^\perp X.\label{eq:PnperpL}
    \end{align}
We only show \eqref{eq:PnL} as the proof of \eqref{eq:PnperpL} is similar.
    \begin{equation}
    \begin{split}
        &(P_n \mathcal LX)(\omega) = \left (P_n \left ( \sum_{i=1}^n  L y_i f_i + \sum_{i=n+1}^{\infty}L y_i f_i \right ) \right )(\omega)= \left ( \sum_{i=1}^n \mathcal L y_i f_i \right ) (\omega) = \mathcal L ((P_n X)(\omega)) = (\mathcal L P_n X)(\omega).
    \end{split}
    \end{equation}
  where in the first equality we used the definition of $\mathcal L$ to have
    \begin{equation}
    \begin{split}
        (\mathcal LX)(\omega) &= \left ( \mathcal L \left ( \sum_{i=1}^n y_i f_i + \sum_{i=n+1}^{\infty} y_i f_i \right ) \right )(\omega) = L \left ( \sum_{i=1}^n y_i f_i(\omega) + \sum_{i=n+1}^{\infty} y_i f_i(\omega) \right )\\
        &= \sum_{i=1}^n L y_i f_i(\omega) + \sum_{i=n+1}^{\infty} L y_i f_i(\omega) = \left ( \sum_{i=1}^n L y_i f_i + \sum_{i=n+1}^{\infty} L y_i f_i \right )(\omega).
    \end{split}
    \end{equation}
Now obviously $\mathcal LP_nX=\mathcal LP_n P_n X=P_n\mathcal LP_nX$.
\end{remark}


We will work with test functions of the form $g(t,\|P_n X\|)$; however, these are not radial functions of $\|X\|$ and, therefore, they  do not fit into the usual class of test functions used in  \cite[Chapter 3]{fabbri_gozzi_swiech_2017} (recalled in Definition \ref{def:test_functions_hilbert}). Hence, we need to enlarge the set of the test functions considered there in the following way.
\begin{definition}\label{generalized_test_functions}
    A function $\psi$ is a generalized test function if $\psi = \varphi + h(t,\|\cdot \|) + g(t,\|P_n \cdot\|)$, where $\varphi \in C^{1,2}((0,T)\times E)$ is locally bounded and is such that $\varphi$ is $B$-lower semicontinuous,  $\partial_t \varphi$, $A^* D\varphi$, $D\varphi$, $D^2\varphi$ are uniformly continuous on $(0,T)\times E$; $h,g\in C^{1,2}((0,T)\times \mathbb{R})$ are such that for every {$t \in(0, T), g(t, \cdot),h(t, \cdot)$ are} even and non-decreasing on $[0,+\infty)$.
\end{definition}
Next, we prove a version of It\^o's inequality for these types of functions, which will allow us to drop the term $\llangle \mathcal{A} X, Dg(s,\|P_n X\|) \rrangle$ in the definition of $\mathcal B$-continuous viscosity solution, in a similar way to what is done for radial functions $h$ in \cite[Chapter 3]{fabbri_gozzi_swiech_2017}.
\begin{lemma}\label{Ito_test_functions}
    Let Assumptions \ref{Assumption_A_maximal_dissipative} and \ref{Assumption_f_sigma_lipschitz}(i)(ii)(iv) be satisfied. Let $G \in C^{1,2}([t,T] \times E)$ be of the form $G(s,X) = g(s,\| P_n X \|)$, for some $g(s,c)\in C^{1,2}([t,T]\times \mathbb{R})$, where $g(s,\cdot)$ is even and non-decreasing on $[0,\infty)$. Moreover, suppose that there are constants $C\geq 0$ and $N\geq 0$ such that
    \begin{equation}
        |G(s,X)| + \| DG(s,X) \| + |G_t(s,X)| + \| D^2 G(s,X) \|_{L(E)} \leq C( 1+\|X\|)^N
    \end{equation}
    for all $X\in E$ and $s\in [t,T]$. Let $X(\cdot)$ be the solution of equation \eqref{Lifted_State_Equation} with initial condition $(t,X)\in [0,T)\times E$ and bounded $a(\cdot)$, and let $\tau$ be an $\mathcal{F}^t_s$-stopping time. Then, $\mathbb{P}$-almost surely,
    \begin{equation}\label{generalized_Ito}
    \begin{split}
        G(s \wedge \tau, X(s \wedge \tau)) &\leq G(t,X) + \int_t^{s\wedge \tau} \bigg [ G_t(s',X(s')) + \llangle F(X(s'),a(s')), DG(s',X(s')) \rrangle\\
        &\quad +\frac12 \text{Tr} \left [ \left ( \Sigma(X(s'),a(s')) \right ) \left ( \Sigma(X(s'),a(s')) \right )^* D^2 G(s',X(s')) \right ] \bigg ] \mathrm{d}s'\\
        &\quad + \int_t^{s\wedge \tau} \llangle DG(s',X(s')),\Sigma(X(s'),a(s')) \mathrm{d}W(s') \rrangle.
    \end{split}
    \end{equation}
\end{lemma}

\begin{proof}
    Inequality \eqref{generalized_Ito} can formally be derived by applying It\^o's formula and dropping the term involving the unbounded operator. Thus, we have to justify why we can drop that term. To this end, first we observe that $DG(s,X) = \frac{\partial g}{\partial c}(s,\|P_n X\|) \frac{P_n X}{\|P_n X\|}$ and, since $g(s,\cdot)$ is non-decreasing on $[0,+\infty)$, $\frac{\partial g}{\partial c}(s,c) \geq 0$. Let $\mathcal{A}_{(k)}$ be the Yosida approximations of $\mathcal{A}$, as introduced in the proof of Lemma \ref{lemma:semigroups_coincide}. Note that for every $s\in [t,T]$ and $X\in E$, we have
    \begin{align}
        \llangle \mathcal{A}_{(k)} X, DG(s,X) \rrangle &= \frac{\partial g}{\partial c}(s,\|P_n X\|) \frac{1}{\|P_n X\|} \llangle \mathcal{A}_{(k)} X, P_n X \rrangle =\frac{\partial g}{\partial c}(s,\|P_n X\|) \frac{1}{\|P_n X\|} \llangle P_n \mathcal{A}_{(k)} X, P_n X \rrangle\\
        &=\frac{\partial g}{\partial c}(s,\|P_n X\|) \frac{1}{\|P_n X\|} \llangle  \mathcal{A}_{(k)} P_nX, P_n X \rrangle \leq 0,
    \end{align}
   where we used \eqref{eq:PnL} (with $\mathcal L=\mathcal{A}_{(k)}$) and the dissipativity of $\mathcal{A}_{(k)}$ (which follows since $\mathcal{A}$ is dissipative). 
    
    Hence, we can proceed as in the proof of \cite[Proposition 1.166]{fabbri_gozzi_swiech_2017}, by dropping this term when applying It\^o's formula to the approximated equation, in which the unbounded operator $\mathcal{A}$ is replaced by its Yosida approximation $\mathcal{A}_{(k)}$, and then passing to the limit for $k\to\infty$.
\end{proof}
Thanks to the previous lemma, we can consider the following stronger definition of viscosity solution in place of the standard one on Hilbert spaces \cite[Definition 3.35]{fabbri_gozzi_swiech_2017} (recalled in Definition \ref{def:viscosity_solution_hilbert}), where we use the enlarged  class of test functions in Definition \ref{generalized_test_functions} (in place of the standard ones in \cite[Definition 3.32]{fabbri_gozzi_swiech_2017}, recalled in Definition \ref{def:test_functions_hilbert}).
\begin{definition}\label{def:modified_viscosity}
\begin{enumerate}[label=(\roman*)]
\item A locally bounded $\mathcal B$-upper semicontinuous function $v:(0,T]\times E\to\mathbb{R}$ is a $\mathcal B$-continuous viscosity subsolution of \eqref{lifted_HJB_equation} if  $v(T,X)\leq U_T(X)$ for $X\in E$ and, whenever $v-\psi$ has a local maximum at $(t,X) \in (0,T)\times E$ for a test function $\psi$ as in Definition \ref{generalized_test_functions}, then 
    \begin{align}
	   \partial_t \psi(t,X)& + \frac12 \text{Tr}(\Sigma(X)(\Sigma(X))^{\ast} D^2\psi(t,X)) +  \llangle X,\mathcal{A}^*D\varphi(t,X) \rrangle + \tilde{\mathcal{H}}(X, D\psi(t,X)) \geq 0.
    \end{align}
\item A locally bounded $\mathcal B$-lower semicontinuous function $v:(0,T]\times E\to\mathbb{R}$ is a $\mathcal B$-continuous viscosity supersolution of \eqref{lifted_HJB_equation} if $v(T,X)\geq U_T(X)$ for $X\in E$ and, whenever $v+\psi$ has a local minimum at $(t,X) \in (0,T)\times E$ for a test function $\psi$ as in Definition \ref{generalized_test_functions}, then 
  \begin{align}
	   -\partial_t \psi(t,X)& - \frac12 \text{Tr}(\Sigma(X)(\Sigma(X))^{\ast} D^2\psi(t,X)) - \llangle X,\mathcal{A}^*D\varphi(t,X) \rrangle + \tilde{\mathcal{H}}(X, -D\psi(t,X)) \leq 0.
    \end{align}
\item 
A function $v:(0,T]\times E\to\mathbb{R}$ is a $\mathcal B$-continuous viscosity solution of \eqref{lifted_HJB_equation} if it is both a $\mathcal B$-continuous viscosity subsolution of \eqref{lifted_HJB_equation}  and a $\mathcal B$-continuous viscosity supersolution of \eqref{lifted_HJB_equation}.
\end{enumerate}
\end{definition}
\begin{prop}\label{prop:U_is_viscosity_with_enlarged_test_functions}
    Let Assumptions \ref{Assumption_A_maximal_dissipative}, \ref{Assumption_weak_B_condition}, \ref{Assumption_B_compact}, \ref{Assumption_f_sigma_lipschitz}(i)(ii)(iv), and \ref{Assumption_running_terminal_cost}(i)(ii)(iv) be satisfied. Then, the value function $U(=V)$ defined in \eqref{Lifted_Value_Function} is a $\mathcal B$-continuous viscosity solution of  \eqref{lifted_HJB_equation} in the sense of Definition \ref{def:modified_viscosity}.
\end{prop}

\begin{proof}
    The proof is a straightforward modification of the one of \cite[Theorem 3.66]{fabbri_gozzi_swiech_2017} using Lemma \ref{Ito_test_functions} in place  of \cite[Proposition 1.166]{fabbri_gozzi_swiech_2017}. We remark that in the proof we can assume that the controls $a(\cdot)$ are bounded {(recall Proposition \ref{Value_Function_Lipschitz})}
    and hence the solutions of the state equation have moments of any order.
\end{proof}

\begin{lemma}\label{proposition:V_n_geq_u_n}
	 Let Assumptions \ref{Assumption_A_maximal_dissipative}, \ref{Assumption_weak_B_condition}, \ref{Assumption_B_compact}, \ref{Assumption_f_sigma_lipschitz}(i)(ii)(iv), and \ref{Assumption_running_terminal_cost}(i)(ii)(iv) be satisfied. Moreover, assume that  $D_{-1} V(t,\cdot):  E_{-1}\to E_{-1}$ is uniformly continuous, uniformly in $t \in [0,T]$.
     Then, for all $(t,\mathbf{x}) \in [0,T] \times H^n$, it holds
$V_n(t,\mathbf{x}) \geq u_n(t,\mathbf{x}).$
\end{lemma}
\begin{proof}We claim that $V_n$ is a $B$-continuous viscosity supersolution of equation \eqref{finite_dimensional_hjb}. Once the claim is proved, taking into account that $u_n$ is a $B$-continuous viscosity (sub-)solution of \eqref{finite_dimensional_hjb}, the result  follows from the comparison theorem \cite[Theorem 3.50]{fabbri_gozzi_swiech_2017} {(recall Proposition \ref{Lipschitz_Time_u_n}
and the fact that both $u_n, V_N$ are Lipschitz continuous in $\mathbf{x}$).}

\textbf{Preliminaries.}  For $\varphi_n\in C^{1,2}((0,T)\times H^n)$, $h_n\in C^{1,2}((0,T)\times \mathbb{R})$, let $\psi_n(t,\mathbf x):=\varphi_n(t,\mathbf x) +h_n(t,|\mathbf{x}|)$ be a test function in the sense of \cite[Definition 3.32]{fabbri_gozzi_swiech_2017} and  and let $V_n+\psi_n$ have a global minimum at $(t_0,\mathbf{x}_0) \in (0,T) \times H^n$. Assume without loss of generality that  $(V_n+\psi_n)(t_0,\mathbf{x}_0)=0$ and that $\psi_n$ is bounded from below. We will denote $X_0:=X_n^{\mathbf{x}_0} \in E_n$ (so $P_n^\perp X_0=0$).

Using Notation \ref{not:notations_Pn}, we have
$V(t,P_nX) = V\left ( t, \sum_{i=1}^n y_i f_i \right ) = V \left ( t, \sum_{i=1}^n \sqrt{n} y_i \mathbf{1}_{A^n_i} \right ) = V_n(t, \sqrt{n} \mathbf{y})$, for $(t,X) \in [0,T]\times E$.
Note also that, by \cite{stannat_vogler}, we have  
\begin{equation}\label{eq:DV=DmuV}
    DV(t,X)(\omega)=\partial_\mu \mathcal{V}(t,X_{\texttt{\#}} \mathcal{L}^1)(X(\omega)), \quad \textit{a.e.};
\end{equation}
thus if  $X\in E_n$, we have  $DV(t,X)\in E_n$, for all $t \in [0,T]$.

Still using Notation \ref{not:notations_Pn}, we define the test functions in the sense of Definition \ref{generalized_test_functions} $\varphi:(0,T)\times E \to \mathbb{R}$,
$\varphi(t,X) := \varphi_n(t,\sqrt{n}\mathbf{y}),$ $g:(0,T)\times \mathbb{R} \to \mathbb{R}$, $g(t,c) := h_n(t,\sqrt{n}c),$ $\psi:(0,T)\times E \to \mathbb{R}$, $\psi(t,X) :=
\varphi(t,X)+g(t,\| P_n X \| )\equiv\varphi_n(t,\sqrt{n}\mathbf{y})+h_n(t,\sqrt{n}|  \mathbf{y} | )\equiv\psi_n(t,\sqrt{n}\mathbf{y})$.
In this way, for $\mathbf x \in H^n$, we have $\psi(t,X^{\mathbf{x}}_n)=\psi_n(t,\mathbf x)$.

\textbf{Step 1.} Fix $\varepsilon>0$. Let $0<\delta<\min(t_0, T-t_0,1)$ 
and consider the function, for  $(t,X) \in (0,T)\times E$,
\begin{equation}
\begin{aligned}
    	&\Psi^{\delta}(t,X):= V(t,X) + \psi(t,X) + \varepsilon((t-t_0)^2+\|P_n (X-X_0)\|_{-1}^2) + \frac{\varepsilon}{\delta^2} \|P_n^{\perp}(X-X_0)\|_{-1}^2+\tilde \chi(X)\\
   & \equiv V_n(t,\sqrt{n}\mathbf{y}) +\psi_n(t,\sqrt{n}\mathbf{y}) + \varepsilon((t-t_0)^2+\|P_n (X-X_0)\|_{-1}^2) + \frac{\varepsilon}{\delta^2} \|P^{\perp}_n X\|_{-1}^2 + V(t,X) - V(t,P_n X)+\tilde \chi(X),
\end{aligned}
\end{equation}
where $\tilde \chi(X):=\chi(\|X\|)$ for a {symmetric function $\chi \in C^2( \mathbb R)$ such that $\chi'(r)\geq 0$ for $r>0$,} $\chi(r)=0$ for $|r| \leq \| X_0\|$; $\chi(r)>0$ for $|r| > \| X_0\|$ and $\lim_{|r| \to \infty} \chi(r)/r^2=1$. Nnotice that $\chi'=\chi''=0$ for $|r| \leq \| X_0\|,$ so that $\tilde \chi(X)=0, D\tilde \chi(X)=0, D^2\tilde \chi(X)=0$, for all $\|X\| \leq \|X_0\|$. {If $X_0=0$ we assume additionally that $\chi''(0)=0$.}

{It can be easily seen using the corollary of the Ekeland--Lebourg Theorem, \cite[Theorem 3.26]{fabbri_gozzi_swiech_2017}, that for any $\eta>0$ there are $a\in \mathbb{R}$, $Z\in E$ with $\|(a,Z)\|_{\mathbb{R}\times E}<\eta$ such that the function}
$$(t,X) \mapsto \Psi^{\delta}(t,X) + at + \llangle {\mathcal B}Z,X \rrangle$$
attains a minimum over $K_{\delta}:= \{ (t,X) \in \mathbb{R} \times E : \|(t-t_0,P_n (X-X_0)\|_{\mathbb{R}\times E_{-1}} \leq \delta, \|P_n^{\perp} X \|_{-1} \leq \delta^2 \}$  at some point $(t_{\delta},X_{\delta}) \in K_\delta$; since  $\Psi^{\delta} \xrightarrow{\|X\|\to \infty} \infty$, uniformly in $\delta$,  we have  $\|X_\delta\| \leq C$  for some $C\geq 0$ independent of $\delta$ {and $\eta$}.

Since we will use the viscosity solution property of $V$, we check that $(t_{\delta},X_{\delta})\not \in \partial K_\delta$. Indeed, for $(t,X)\in \partial K_\delta$ (i.e., either such that $\|(t-t_0,P_n (X-X_0))\|_{\mathbb{R} \times E_{-1}} = \delta$, $\|P_n^{\perp}X\|_{-1} \leq \delta^2$ or such that $\|(t-t_0,P_n (X-X_0))\|_{\mathbb{R}\times E_{-1}} \leq \delta$, $\|P_n^{\perp} X \|_{-1}=\delta^2$), we have
$$\Psi^{\delta}(t,X)
	\geq \varepsilon \delta^2 + V(t,X) - V(t,P_n X),$$
where we also used that $V_n + \psi_n \geq 0$ (which follows since $V_n + \psi_n$ has a global minimum equal to zero at $(t_0,\mathbf x_0)=(t_0,\sqrt n \mathbf y)$).
Using the (uniform) continuity of   $D_{-1} V(t,\cdot):E_{-1}\to E_{-1}$, we have that   $DV(t,\cdot)\equiv \mathcal B D_{-1}V(t,\cdot):E \to E$ (this equality is standard) is (uniformly) continuous; hence, we use the mean-value theorem, i.e., denoting $X^\theta:=\theta X+(1-\theta)P_n X$ for some $\theta \in [0,1]$,  we have
\begin{equation}\label{eq:V-VPnX}
    \begin{aligned}
    &| V(t,X) - V(t,P_n X) | =| \llangle DV(t,X^\theta), P_n^{\perp}X\rrangle|=| \llangle DV(t,X^\theta) -  DV(t,X_0) , P_n^{\perp}X \rrangle|\\
     &= | \llangle    D_{-1}V(t,X^\theta) -   D_{-1}V(t,X_0) ,    P_n^{\perp}X \rrangle_{-1}|\leq \rho(\|X-X_0\|_{{-1}}) \|P_n^{\perp} X \|_{-1}  \leq  \rho(2\delta) \delta^2,
\end{aligned}
\end{equation}
where we have used the fact that $\llangle  DV(t,X_0) ,P_n^{\perp}X \rrangle=0$ (which follows from $DV(t,X_0)\in E_n$ as $X_0 \in E_n$, recall \eqref{eq:DV=DmuV}), $\rho$ is the modulus of continuity of $D_{-1} V(t,\cdot)$ (uniform in $t$), and we have noticed that, by definition of $K_\delta$, we have $\|X-X_0\|_{{-1}}\leq \|P_n(X-X_0)\|_{{-1}}+\|P_n^\perp X\|_{{-1}}\leq 2\delta$.
Thus, we have $$\Psi^{\delta}(t,X) \geq \varepsilon \delta^2 - \rho(2\delta) \delta^2=\frac{\varepsilon\delta_\varepsilon^2}{2}=:\gamma_\varepsilon,$$
where in the  equality above, for all $0<\varepsilon<1$, we have chosen a sufficiently small
$0<\delta:=\delta_\varepsilon < \varepsilon$. Next, we notice that if $\eta=\eta_\varepsilon$ is sufficiently small, denoting $a=a_\varepsilon,Z=Z_\varepsilon$, we have
	$|a_\varepsilon t + \llangle {\mathcal B}Z_\varepsilon,X \rrangle|  < \frac{\gamma_\varepsilon}{2}$, for all $(t,X)\in K_{\delta}$. Then
\begin{equation}
	\Psi^{\delta}(t,X) + a_\varepsilon t + \llangle {\mathcal B}Z_\varepsilon,X \rrangle >  \frac{\gamma_\varepsilon}{2}, \quad \forall (t,X) \in \partial K_{\delta}.
\end{equation}
 Taking into account that 
 \begin{equation}\label{eq:Psi(X_0)}
     \Psi^\delta(t_0,X_0)=(V + \psi)(t_0,X_0)+ \tilde \chi(X_0) =V_n(t_0,\mathbf x_0)+\psi_n(t_0,\mathbf x_0) = 0,
 \end{equation}
 it follows that $(t_{\delta},X_{\delta})$ must be in the interior of $K_{\delta}$. 

 \textbf{Step 2.} Recall that in the proof of Step 1, for all $\varepsilon>0$, we have chosen a suitable $\delta=\delta_\varepsilon$ (which we keep denoting $\delta$ to simplify the notation). Here, we prove that, as $\varepsilon\to 0$, we have $t_\delta \to t$ and  (extracting a sequence and  then a subsequence, if necessary, still denoted by $X_\delta$) $ X_\delta \to X_0$.
 
Indeed, first observe  that by definition of $K_\delta$, 
 \begin{equation}\label{eq:convergence_t,Xdelta_-1}
     |t_\delta-t_0|\to 0, \quad \|X_\delta-X_0\|_{-1} \to 0.
 \end{equation}
By Step 1, $X_\delta$ is bounded so that (extracting a sequence and then a subsequence, still denoted by $X_\delta$),  we have $X_\delta \rightharpoonup Y \in E$, so that $ \mathcal B^{1/2} X_\delta \rightharpoonup \mathcal B^{1/2} Y $; this, together with $ \mathcal B^{1/2}X_\delta  \to \mathcal B^{1/2}X_0$, implies $\mathcal B^{1/2}(Y-X_0)=0$; in turn, due to the strict positivity of $\mathcal B^{1/2}$, this implies $Y=X_0$, i.e., $$X_\delta \rightharpoonup X_0.$$
 
Next, we conclude the proof of Step 2 by proving   $$\|X_\delta\|\to \|X_0\|.$$ Indeed, since $X_\delta \rightharpoonup X_0$, we  have $\left\|X_0\right\| \leq \liminf _{\delta \rightarrow 0}\left\|X_\delta\right\|$. We are left to prove $\limsup_{\delta \to 0}\|X_\delta\| \leq \|X_0\|$. Indeed, by definitions of $\tilde \chi$, $\Psi^\delta$, taking into account that $V_n+\psi_n$ has a  global minimum at $(t_0,\mathbf{x}_0)$ equal to zero, and using \eqref{eq:Psi(X_0)}, we have
$\tilde \chi(X_\delta)+
V(t_\delta,X_\delta)-V(t_\delta,P_n X_\delta)\leq \Psi^\delta(t_\delta,X_\delta)
\leq\Psi^\delta(t_0,X_0)
+a_\varepsilon(t_0-t_\delta)
+\llangle\mathcal BZ_\varepsilon,X_0-X_\delta\rrangle_E
\leq\gamma_\varepsilon\to 0.$
Since $|V(t_\delta,X_\delta)-V(t_\delta,P_nX_\delta)|  \xrightarrow{\varepsilon \to 0} 0$ (recall \eqref{eq:V-VPnX}), we have  $\limsup_{\varepsilon\to 0} \tilde \chi(X_\delta)\leq 0$; by definition of $\tilde \chi$, we have the claim.

 
\textbf{Step 3.} We prove the claim using the viscosity supersolution property of $V$. 

Since $V$ is a $\mathcal B$-continuous viscosity supersolution of \eqref{lifted_HJB_equation} (recall Proposition \ref{prop:U_is_viscosity_with_enlarged_test_functions}), by Step 1 we have (here $\{e_m' \}_{m \in\mathbb N}$ is an orthonormal basis of $\Xi$) 
\begin{equation}\label{supersolution_inequality}
\begin{split}
	&-\partial_t \psi(t_{\delta},X_{\delta}) - 2\varepsilon (t_{\delta}-t_0) - a_\varepsilon- \llangle X_\delta, \mathcal A^*[D\varphi(t_{\delta},X_{\delta}) + 2 \varepsilon  \mathcal B P_n (X_{\delta}-X_0) + 2 \frac{\varepsilon}{\delta^2}   \mathcal B P_n^{\perp} X_{\delta}- \mathcal B Z_\varepsilon]\rrangle \\
    &\qquad - \frac12 \sum_{m=1}^{\infty} \llangle ( D^2 \psi(t_{\delta},X_{\delta}) + 2 \varepsilon P_n  \mathcal B P_n + 2\frac{\varepsilon}{\delta^2} P_n^\perp \mathcal B P_n^{\perp}+D^2 \tilde \chi(X_\delta )) \Sigma(X_{\delta}) e^{\prime}_m, \Sigma(X_{\delta}) e^{\prime}_m \rrangle\\
	&\qquad + \tilde{\mathcal{H}}(X_{\delta}, -D\psi(t_{\delta},X_{\delta}) - 2 \varepsilon  \mathcal B P_n {(X_{\delta}-X_0)} - 2 \frac{\varepsilon}{\delta^2}  \mathcal B P_n^{\perp} X_{\delta} -\mathcal B Z-D \tilde \chi(X_\delta )) \leq 0.
\end{split}
\end{equation}
Applying \eqref{eq:PnperpL} with $\mathcal L=\mathcal A^*\mathcal B \in L(E)$ (recall Section \ref{subsection_operator_A}), the weak $B$-condition (Assumption \ref{Assumption_weak_B_condition}), and using that $X_\delta \in K_\delta$, we have 
$-   \frac{\varepsilon}{\delta^2} \llangle  X_\delta,\mathcal A^* \mathcal B P_n^{\perp}X_{\delta}\rrangle= - \frac{\varepsilon}{\delta^2} \llangle  P_n^{\perp} X_\delta,\mathcal A^* \mathcal BP_n^{\perp} X_{\delta}\rrangle \geq -c_0\frac{\varepsilon}{\delta^2} \|  P_n^{\perp} X_{\delta}\|_{-1}^2 \geq -c_0 \eps \delta^2.$
Moreover,
for $X \in E$,  
$\Sigma(P_n X)(\omega) = \sigma \left (\sum_{i=1}^n y_i f_i(\omega),(P_n X)_{\texttt{\#}} \mathcal{L}^1 \right ) = \sum_{i=1}^n \sigma \left (\sqrt{n} y_i,(P_n X)_{\texttt{\#}} \mathcal{L}^1 \right ) \mathbf{1}_{A^n_i}(\omega),$
so that $\Sigma(P_n X) e^{\prime}_m \in E_n$, and, in turn, $P_n^\perp \Sigma(P_n X) e^{\prime}_m=0$. Then 
{\begin{equation}
\begin{split}
    |\sum_{m=1}^\infty \llangle \frac{\varepsilon}{\delta^2}  &P_n^\perp \mathcal B P_n^{\perp} \Sigma(X_{\delta})e^{\prime}_m, \Sigma(X_{\delta})e^{\prime}_m \rrangle | =|\sum_{m=1}^\infty \llangle \frac{\varepsilon}{\delta^2} P_n^{\perp} \Sigma(X_{\delta})e^{\prime}_m, P_n^{\perp} \Sigma(X_{\delta})e^{\prime}_m \rrangle_{-1} |\\
    &= \frac{\varepsilon}{\delta^2} \sum_{m=1}^\infty \|P_n^{\perp} (\Sigma(X_{\delta})- \Sigma(P_n X_{\delta})) e^{\prime}_m\|^2_{-1} 
    =\frac{\varepsilon}{\delta^2} \|P_n^{\perp} (\Sigma(X_{\delta})- \Sigma(P_n X_{\delta}))\|^2_{L_2(\Xi,E_{-1})}\\
    &\leq C \frac{\varepsilon}{\delta^2} \| P^{\perp}_n X_{\delta} \|^2_{-1} \leq C\varepsilon\delta^2.
\end{split}
\end{equation}}
Therefore, due to  Steps 1,2, we let $\varepsilon\to 0$,  so that using the continuity of all terms (recall Subsection \ref{subsec:properties_lifted_coeff}), we have
\begin{align}
	-\partial_t \psi(t_0,X_0)& - \frac12 \sum_{m=1}^{\infty} \llangle D^2 \psi(t_0,X_0) \Sigma(X_0) e^{\prime}_m, \Sigma(X_0) e^{\prime}_m \rrangle - \llangle X_0, \mathcal A^*D\varphi(t_{0},X_{0})\rrangle + \tilde{\mathcal{H}}(X_0,-D\psi(t_0,X_0)) \leq 0.
\end{align}
Recalling the definitions of the test functions and Notation \ref{not:notations_Pn}, for every $(t,X)\in (0,T]\times E$, we have
$
D\psi(t,X) = \sum_{i=1}^n \sum_{k=1}^\infty \sqrt{n} D_{x_i^k} \psi_n( t, \sqrt{n} \mathbf{y} ) f_i \otimes e_k$
and
$D^2 \psi(t,X) = \sum_{i,j=1}^n \sum_{k,l=1}^\infty n D_{x_i^k x_j^l} \psi_n(t,\sqrt{n} \mathbf{y}) \llangle f_i\otimes e_k, \cdot \rrangle \llangle f_j\otimes e_l ,\cdot \rrangle$. By definition of $\tilde{\mathcal{H}}$, we have
$\tilde{\mathcal{H}}\left(X_0,- D \psi\left(t_0, X_0\right)\right)=\frac{1}{n} \sum_{i=1}^n \mathcal{H}\left(x_{0,i}, \mu_{\mathbf x_0},- n D_{x_i} \psi_n\left(t_0, \mathbf{x}_0\right)\right)$, where $x_{0,i}$ denotes the $i$-th component of $\mathbf x_0$ (and similarly for the term in $\mathcal A^*$).
For the second order term, using similar techniques as in the proof of Theorem \ref{theorem:convergence} and \cite[Proof of Proposition 5.2]{swiech_wessels_2024}, we have 
$\sum_{m=1}^{\infty} \llangle D^2 \psi(t_0,X_0) \Sigma(X_0) e^{\prime}_m, \Sigma(X_0) e^{\prime}_m \rrangle= \text{Tr} ( D^2 \psi_n(t_0,\mathbf{x}_0)\sigma (\mathbf{x}_0,\mu_{\mathbf{x}_0})\sigma^T(\mathbf{x}_0,\mu_{\mathbf{x}_0})).$ Thus, we obtain 
\begin{align}
	-\partial_t \psi_n(t_0,\mathbf{x}_0) &- \frac12 \text{Tr} ( D^2 \psi_n(t_0,\mathbf{x}_0)\sigma (\mathbf{x}_0,\mu_{\mathbf{x}_0})\sigma^T(\mathbf{x}_0,\mu_{\mathbf{x}_0}))\\
    &+ \frac{1}{n} \sum_{i=1}^n \left(- \langle x_{0,i},nA^*D_{x_i} {\varphi_n}(t_0,x_{0,i}) \rangle +\mathcal{H}(x_{0,i},\mu_{\mathbf{x}_0},-n D_{x_i} \psi_n\left(t_0, \mathbf{x}_0\right))\right) \leq 0,
\end{align}
which proves the claim.
\end{proof}

\section{Lifting and Projection of Optimal Controls}\label{sec:lifting_proj}
In this section, we are going to show that if {$D_{-1}V(t,\cdot)$ is uniformly continuous} (resp. $V(t,\cdot)\in C^{1,1}(E_{-1})$ for the case of optimal feedback controls), then optimal (resp. optimal feedback) controls of the particle system correspond to optimal (resp. optimal feedback) controls of the lifted infinite dimensional control problem started at the corresponding initial condition. Conversely, we will show that piecewise constant optimal (resp. optimal feedback) controls of the infinite dimensional control problem project onto optimal (resp. optimal feedback) controls of the particle system.

Throughout this section, we assume that the assumptions of Theorem \ref{theorem:projection_C11} are satisfied.

\subsection{Lifting of Optimal Controls}\label{sub:lifting_optimal}

Let $n\geq 1$ and let $\mathbf{a}^{\ast}(\cdot) = (a_1^{\ast}(\cdot),\dots,a_n^{\ast}(\cdot))$ be an optimal control of the finite dimensional problem, i.e.,
\begin{equation}
    u_n(t,\mathbf{x}) = J_n(t,\mathbf{x};\mathbf{a}^{\ast}(\cdot)) := \mathbb{E} \left [ \int_t^T \left ( \frac1n \sum_{i=1}^n  l(x^{\ast}_i(s),\mu_{\mathbf{x}^{\ast}(s)},a^{\ast}_i(s)) \right ) \mathrm{d}s +  \frac{1}{n} \sum_{i=1}^n \mathcal{U}_T(x^*_i(T),\mu_{\mathbf{x}^{\ast}(T)})  \right ],
\end{equation}
where $\mathbf{x}^{\ast}(s) = (x_1^{\ast}(s),\dots,x_n^{\ast}(s))$ is the mild solution of the system of SDEs
\begin{equation}\label{eq:closed_loop}
		\mathrm{d}x^*_i(s) = [A x^*_i(s) +f(x^*_i(s),\mu_{{\bf x}^*(s)},a^*_i(s))]\mathrm{d}s + \sigma(x^*_i(s),\mu_{{\bf x}^*(s)})\mathrm{d}W(s),  \quad x_i^*(t) = x_i \in H,
    \end{equation}
$i=1,\dots,n$. By Lemma \ref{lemma_subsolution} and Theorem \ref{theorem:projection_C11}, we have
\begin{equation}
    J(t,X^{\mathbf{x}}_n ; a^{\ast,n}(\cdot))=J_n(t,\mathbf{x};\mathbf{a}^{\ast}(\cdot))  = u_n(t,\mathbf{x}) = V_n(t,\mathbf{x}) = V\left (t, \sum_{i=1}^n x_i \mathbf{1}_{A^n_i} \right ),
\end{equation}
where $a^{\ast,n}(\cdot) = \sum_{i=1}^n a_i(\cdot) \mathbf{1}_{A^n_i}$, i.e., the optimal control $\mathbf{a}^{\ast}(\cdot)$ lifts to an optimal control $a^{\ast,n}(\cdot)$ of the infinite dimensional problem started at $(t,\sum_{i=1}^n x_i \mathbf{1}_{A^n_i})$.

\subsection{Lifting of Optimal Feedback Controls}\label{sub:lifting_optimal_feedbacks}
Next, we show that an optimal feedback control of the finite dimensional problem lifts to an optimal feedback control of the corresponding lifted control problem.

Assume that $\tilde{\Lambda}=\Lambda=H$ and that $f,l$ are of the form $f(x,\mu,q):=f_1 (x,\mu)+q$, $l(x,\mu,q):=l_1(x,\mu)+l_2(q)$ with $l_2\in C^{1,1}(H)$. Moreover, in addition to the assumptions of Theorem \ref{theorem:projection_C11}, let Assumption \ref{Assumption_running_cost_nu} be satisfied with $\nu >0$, $V(t,\cdot)\in C^{1,1}(E_{-1})$ for every $t\in [0,T]$, and the semiconcavity and semiconvexity constants of $V(t,\cdot)$ be independent of $t\in [0,T]$. Notice that under this assumption, for all $t\geq 0$ we have $V(t,\cdot)\in C^{1,1}(E)$, $V_n(t,\cdot) \in C^{1,1}(H^n)$, {$DV,D_{x_i} V_n$ are continuous on $[0,T]\times E$ and 
$[0,T]\times H^n$ respectively,} and there is a constant $C\geq 0$ such that $\|D V(t, X)\| \leq C,\left\|D V(t, X)-D V\left(t, Y \right)\right\| \leq C\left\|X-Y\right\|_{-1}$, $|D_{x_i} V_n(t, \mathbf x)| \leq C,$ $\left|D_{x_i} V_n(t, \mathbf x)-D_{x_i} V_n\left(t, \mathbf y\right)\right| \leq C\left|\mathbf x-\mathbf y\right|_{-1,2}$, for all $i=1,\dots,n$.
Then we can proceed in a similar way to \cite[Example 4.10]{defeo_swiech_wessels_2023}  to have that $Dl_2$ is invertible with Lipschitz inverse $(Dl_2)^{-1}$ and then apply \cite[Theorem 5.18]{defeo_swiech_wessels_2023} to the finite particle system problem with value function $u_n=V_n$ to get that
  $\mathbf{a}^{\ast}(\cdot) =  (a^{\ast}_1(\cdot),\dots,a^{\ast}_n(\cdot)),$ $a_i^{\ast}(s):=(Dl_2)^{-1}(-nD_{x_i}V_n(s,{\mathbf{x}^{\ast}(s)}))$
    is an optimal feedback control for the finite particle system control problem\footnote{Recall that the factor $n$ enters together with $D_{x_i}u_n$ into the HJB equation \eqref{finite_dimensional_hjb}.}. Here $\mathbf{x}^{\ast}(s)=(x_1^{\ast}(s),\dots,x_n^{\ast}(s))$ is the solution of the closed loop system, for $i=1,\dots,n$,
    \begin{equation}
		\mathrm{d}x^{\ast}_i(s) = [Ax_i^*(s)+(Dl_2)^{-1}(-nD_{x_i}V_n(s,{\mathbf{x}^{\ast}(s)})) + f_1(x_i^{\ast}(s),\mu_{\mathbf{x}^{\ast}(s)})] \mathrm{d}s + \sigma(x^{\ast}_i(s),\mu_{\mathbf{x}^{\ast}(s)}) \mathrm{d}W(s), \quad x^{\ast}_i(t) = x_i.
    \end{equation}
 Now, consider the lifted control
        $a^{\ast,n}(s) = \sum_{i=1}^n a^{\ast}_i(s) \mathbf{1}_{A^n_i}.$
    By Lemma \ref{lemma_subsolution},  the solution $X^{\ast,n}(\cdot)$ of the lifted state equation with initial condition $X^{\mathbf{x}}_n$ and control $a^{\ast,n}(\cdot)$ coincides with the lift $X_n^{\mathbf x^*(\cdot)}$ of  $\mathbf{x}^{\ast}(\cdot)$. In particular, $X^{\ast,n}(\cdot)(\omega)$ is  constant for $\omega \in A_i^n$ . Therefore, by the same formula for \eqref{derivative_varphi_n} and using \eqref{eq:DV=DmuV},
    \begin{equation}
    \begin{split}
        a^{\ast,n}(s) &=\sum_{i=1}^n(Dl_2)^{-1}(-nD_{x_i}V_n(s,{\mathbf{x}^{\ast}(s)}))\mathbf{1}_{A^n_i}= \sum_{i=1}^n(Dl_2)^{-1}\left(-n\int_{A^n_i}DV(s, X^{\ast,n}(s))(\omega) \mathrm{d}\omega\right)\mathbf{1}_{A^n_i}\\
        &=\sum_{i=1}^n (Dl_2)^{-1} \left(-n\int_{A^n_i}\partial_\mu \mathcal{V}(s, X^{\ast,n}_{\texttt{\#}} \mathcal{L}^1)( X^{\ast,n}(s)(\omega))
        \mathrm{d}\omega \right)\mathbf{1}_{A^n_i}\\
        & =\sum_{i=1}^n(Dl_2)^{-1}\left(-\partial _\mu\mathcal{V}(s, X^{\ast,n}_{\texttt{\#}} \mathcal{L}^1)( X^{\ast,n}(s)) \right)\mathbf{1}_{A^n_i} =(Dl_2)^{-1}\left(-\partial _\mu\mathcal{V}(s, X^{\ast,n}_{\texttt{\#}} \mathcal{L}^1)( X^{\ast,n}(s)) \right)\\
        &=(DL_2)^{-1}(-DV(s, X^{\ast,n}(s))),
\end{split}
\end{equation}
where in  the last equality we have used the fact that $(DL_2)^{-1}(Y)(\omega)=(Dl_2)^{-1}(Y(\omega))$, for $Y\in E$, with 	$L_2(Q) := \int_{\Omega} l_2(Q(\omega)) \mathrm{d}\omega$. Then, applying the results of \cite[Theorem 5.18]{defeo_swiech_wessels_2023}  to the lifted control problem, we have that $a^{\ast,n}(s)$ is an optimal feedback control of the lifted problem with initial data $(t,\sum_{i=1}^n x_i \mathbf{1}_{A^n_i})$. 

\subsection{Projection of Optimal Controls}\label{sub:projection_optimal}
Let $a^{\ast,n}(\cdot)$ be a piecewise constant optimal control of the infinite dimensional problem started at $(t,X^{\mathbf{x}}_n)$, $\mathbf{x}\in H^n$, i.e., $a^{\ast,n}(s) = \sum_{i=1}^n a_i^{\ast}(s) \mathbf{1}_{A^n_i}$ for some $a_i^{\ast}(\cdot)$ taking values in $H$, $i=1,\dots,n$, and
\begin{equation}
    V(t,X^{\mathbf{x}}_n) = J(t,X^{\mathbf{x}}_n;a^{\ast,n}(\cdot)) = \mathbb{E} \left [ \int_t^T  L(X^{\ast}(s),a^{\ast,n}(s))  \mathrm{d}s + U_T(X^{\ast}(T)) \right ],
\end{equation}
where $X^*(s)$ is the mild solution of
\begin{equation}
		\mathrm{d}X^*(s) = [\mathcal{A}X^*(s) + F(X^*(s),a^*(s))] \mathrm{d}s + \Sigma(X^*(s)) \mathrm{d}W(s),\quad X^*(t) = X^{\mathbf{x}}_n \in E_n.
\end{equation}
Note that  $X^{\ast}(s) \in E_n$ for all $s\in [t,T]$ (see \eqref{E_n} for the definition of $E_n$). Let $\mathbf{a}^{\ast}(\cdot) = (a_1^{\ast}(\cdot),\dots,a_n^{\ast}(\cdot))$ be the projection of $a^{\ast,n}(\cdot)$. Then, by Lemma \ref{lemma_subsolution} and Theorem \ref{theorem:projection_C11}, we have
\begin{equation}
J_n(t,\mathbf{x};\mathbf{a}^{\ast}(\cdot))=J(t,X^{\mathbf{x}}_n;a^{\ast,n}(\cdot))=V(t,X^{\mathbf{x}}_n) = V_n(t,\mathbf{x}) = u_n(t,\mathbf{x}),
\end{equation}
i.e., the piecewise constant optimal control $a^{\ast,n}(\cdot)$ of the infinite dimensional control problem started at $(t,X^{\mathbf{x}}_n)$ projects to an optimal control of the finite dimensional control problem started at $(t,\mathbf{x})$.

\subsection{Projection of Optimal Feedback Controls}\label{sub:projection_optimal_feedbacks}
We work under the assumptions of Subsection \ref{sub:lifting_optimal_feedbacks}. We know that an optimal feedback control for the infinite dimensional control problem started at $(t,X^{\mathbf{x}}_n)$, $\mathbf{x}\in H^n$, is given by
\begin{equation}
    a^{\ast}(s)=(DL_2)^{-1}(-DV(s,X^{\ast}(s))),
\end{equation}
where $X^{\ast}(\cdot)$ is the solution of the infinite dimensional state equation with control $a^{\ast}(\cdot)$. Using Lemma \ref{lemma_subsolution} we have that $X^{\ast}(s)$ is   constant for $\omega \in A_i^n$, so that, using \eqref{eq:DV=DmuV}, we have that $DV(s,X^{\ast}(s))$ is constant for $\omega \in A_i^n$, and then $a^{\ast}(\cdot)$ is constant for $\omega \in A_i^n$. Therefore, in this case the same calculation (in reverse) as in Subsection \ref{sub:lifting_optimal_feedbacks} shows that the optimal feedback control projects onto an optimal feedback control of the finite dimensional control problem.

\section{Applications to problems in economics}\label{sec:applications}
In this section, we apply the theory developed to problems arising in economics in the context of optimization for large companies: a (path-dependent) stochastic optimal control problem with delays arising in the context of optimal advertising  and a stochastic optimal control problem with vintage capital, where the state equation is a stochastic partial differential equation.
{\begin{notation}
    Throughout this section, for $X\in E$, we denote $\mathbf E[X]:=\int_\Omega X(\omega)d\omega$. 
\end{notation}}
\subsection{Stochastic optimal advertising with delays for large companies }\label{subsec1}
Inspired by  mean-field optimal advertising problems with delays in \cite{gozzi_masiero_rosestolato_2024}, in this subsection we introduce a particle system type control problem for optimal advertising strategies of a large company. We also refer to  \cite{ricciardi_rosestolato_2024} for a mean-field game framework with delays and to
\cite{deFeo_2024,deFeo_2024,defeo_federico_swiech_2024,defeo_swiech,gozzi_marinelli_2004}, and the references therein, for classical stochastic optimal advertising models with delays.

Consider a large retail company with $n$ stores, where local advertising policies interact. {The company plans an advertising campaign to promote the sale of a particular product. 
As typically done in this kind of problems (see the seminal paper \cite{NerloveArrow62} or the survey \cite{FHSMS1994}), the state variable for each store $i$ is the so-called ``goodwill'' $y_i$ of the product which is an index capturing  the brand awareness, the product image, and the reputation that the average consumers have about the product itself.} {Concerning the dynamics of the goodwill stock, as in the cited papers, we assume that the goodwill increases because of the advertising campaign and decreases because people forget about the product and this process is distributed over time.
Hence the} dynamics of the stock of the advertising goodwill $y_i(s)$ of the product sold by store $i =1, \ldots, n$  is given by  the following stochastic delay differential equation (SDDE)
\begin{equation}
\begin{cases}
\mathrm{d}y_i(s) = \left[ b^0 y_i(s)+c^0 \overline{\mathbf{y}}_n(s)+ \int_{-d}^0 [\eta^1(\xi)y_i(s+\xi)+\chi^1(\xi)\overline{\mathbf{y}}_n(s+\xi)] \,\mathrm d\xi +e^0 a_i(s) \right]
         \mathrm{d}s   + \sigma^0 \, \mathrm{d}W(s),\\
y_i(t)=x_i^0, \quad y_i(t+\xi)=x_i^1(\xi), \quad  \xi\in[-d,0),
\end{cases}
\end{equation}
where $s \in [t,T]$, $d>0$, $\overline{\textbf{y}}_n(s)=\frac 1 n \sum_{i=1}^n y_i(s)=\int_{\mathbb R} z \mu_{{\bf y}(s)}(\mathrm d z)$  is the sample mean, with $\mu_{{\bf y}(s)}$ being the empirical measure for $\mathbf {y}(s)=(y_1(s),\dots,y_n(s))$; the control process $a_i(s)$ with values in $\tilde{\Lambda}= [0,\infty) \subset \Lambda =\mathbb R$ models the rate of advertising spending, $W$ is a real-valued Brownian motion acting as a common noise;   $b^0, c^0 \leq 0$ are constant factors of {instantaneous goodwill} deterioration in absence of advertising;  $e^0 \geq 0$ is a constant advertising effectiveness factor; $\eta^1,\chi^1 \leq 0$ are given deterministic functions in the Sobolev space $W^{1,2}([-d,0])$, such that $\eta^1(-d)=\chi^1(-d)=0$, {representing the distribution of the past forgetting times\footnote{{Notice that such regularity assumption on $\eta^i,\chi^i$ needed to apply our results, is not really restrictive as any function in $L^2$ can be approximated by maps of this type.}}}; $\sigma^0>0$ represents  the intensity of the uncertainty in the model {assumed to be common to all stores (while the intensity of the noise of each store is assumed to be negligible with respect to the {common noise})\footnote{{This type of assumption is done when the focus is on global uncertainties like the ones related to the actual environmental/political crises. In the context of Mean Field Games this has been done, e.g., in  \cite{gueant_lasry_lions_2011,carmona_2021,souganidis_zariphopoulou_2024,bertucci_bertucci_lasry_lions_2024}}}}, $x^0_i \in \mathbb R$ is the level of goodwill at the beginning of the advertising campaign;  $x^1_i \in L^2([-d,0])$ is the history of the goodwill level. 

{Again as usual in this kind of problems (see again \cite{NerloveArrow62} or \cite{FHSMS1994}) the company aims at maximizing the present value of its total profit. Hence 
the goal is to minimize}
\begin{equation}
	\mathbb{E}\left [ \int_t^T\frac{1}{n}\sum_{i=1}^n l^0(y_i(s),\overline{\mathbf{y}}_n(s),a_i(s)) \mathrm{d}s+ \frac{1}{n} \sum_{i=1}^n {\mathcal U}_T^0(y_i(T),\overline{\mathbf{y}}_n(T)) \right ],
\end{equation}
where $l^0(y,z,q)=g^0(q)-h^0(y,z)$ for a convex, non decreasing, cost function $g^0 :[0,\infty)\to[0,\infty)$ such that $-C_1+C_2|q|^2 \leq g^0(q) \leq C_1+C_3|q|^2$ for some $C_1,C_2,C_3>0,$  and a concave Lipschitz revenue function $h^0  :\mathbb R^2 \to \mathbb R$; $\mathcal U_T^0 :\mathbb R^2 \to \mathbb R$ is a convex Lipschitz function  (to be interpreted as the negative of a concave terminal revenue function). {To get some of the results, namely the ones concerning the regularity of the value function (Subsection \ref{subsection:C11_regularity}, in particular Theorem \ref{th:C11}) and the projections (Sections \ref{sec:projection}-\ref{sec:lifting_proj}, in particular Theorem \ref{theorem:projection_C11}) we will also  assume the following regularity:
$g^0 \in C^{1,1}([0,\infty);[0,\infty))$, $h^0  \in C^{1,1}(\mathbb R^2)$, $\mathcal U_T^0  \in C^{1,1}(\mathbb R^2)$. Finally, to get the results on the lifting of optimal feedback controls (Subsections \ref{sub:lifting_optimal_feedbacks}
and \ref{sub:projection_optimal_feedbacks})
we will need to add that, for some $\nu>0$ the map $g^0(q)-\nu |q|^2$ is convex. We notice that all the above assumptions, including the last two are typically verified in the above cited literature on optimal advertising.}

The system of SDDEs is rewritten (see e.g. \cite{defeo_federico_swiech_2024}) as a system without delays  over the Hilbert space $H^n$ with $H:=\mathbb R \times L^2([-d,0])$ (endowed with the inner product $\langle x,y \rangle:=x^0\cdot y^0 + \langle x^1,y^1 \rangle_{L^2}$ with $\langle x^1,y^1 \rangle_{L^2}$ denoting the standard inner product in $L^2([-d,0])$) for the extended state $$x_i(s)=(x_i^0(s),x_i^1(s)):=(y_i(s),y_i(s+\cdot)),\quad i =1, \ldots, n,$$ where we use the notation $x:=(x^0,x^1)  \in H$ and $\mathbf x=(\mathbf x^0, \mathbf x^1) \in H^n$, i.e.,
\begin{equation}
\mathrm{d}x_i(s) = \left[A x_i(s)+ b x_i(s)+c \overline{\mathbf{x}}_n(s)+  \eta x_i(s)+\chi \overline{\mathbf{x}}_n(s) +e a_i(s) \right] \mathrm{d}s   + \sigma \, \mathrm{d}W(s),\quad x_i(t)=x_i,
\end{equation}
where $\overline{\mathbf{x}}_n(s):=(\frac 1 n \sum_{i=1}^n x_i^0(s),\frac 1 n \sum_{i=1}^n x_i^1(s))=( \int_{H} z^0 \mu_{{\bf x}(s)}(\mathrm d z),\int_{H} z^1 \mu_{{\bf x}(s)}(\mathrm d z))$,  $ A x:= (-x^0,\frac{\mathrm{d}}{\mathrm{d} \xi} x^1),$ $ D(A):= \left\{ x=(x^0,x^1) \in H: x^1 \in W^{1,2}([-d,0]), \ x^1(0)=x^0\right\}$ is a maximally dissipative operator (so Assumption \ref{Assumption_A_maximal_dissipative} holds), $b \in L(H)$, $bx:=((b^0+1)x^0 ,0)$, $c \in L(H)$, $cx:=(c^0x^0 ,0)$, $\eta \in L(H)$, $\eta x := (\langle \eta^1,x^1\rangle_{L^2},0)$, $\chi \in L(H)$, $\chi x := (\langle \chi^1,x^1\rangle_{L^2},0)$, $e \in L(\Lambda;H)$, $e q :=(e^0 q,0 )$, $\sigma \in L_2(\mathbb R;H)$, $\sigma w :=(\sigma^0 w,0 )$. Then,  defining $ f \colon H \times \mathcal P(H) \times \tilde{\Lambda} \to H$, $\sigma\in L_2(\mathbb{R},H)$, $l \colon H \times \mathcal P(H) \times \tilde{\Lambda} \to \mathbb R$, $\mathcal U_T \colon H \times \mathcal P(H)  \to \mathbb R$, respectively, by  
\begin{align}
&f(x,\mu,q) :=
\left ( (b^0+1) x^0 + c^0 \int_H z^0 \mu(\mathrm d z)  +\langle \eta^1,x^1\rangle_{L^2} + \int_{H} \langle \chi^1,z^1\rangle_{L^2} \mu(\mathrm d z) +e^0 q,
0\right),\\
&\sigma w:=(\sigma^0 w,0),\quad l\left (x,\mu,q\right ):=l^0\left (x^0,\int_{H} z^0 \mu(\mathrm d z),q\right ), \quad \mathcal U_T(x,\mu):=\mathcal U_T^0\left (x^0,\int_{H} z^0 \mu(\mathrm d z)\right ),
\end{align}
for all $x=(x^0,x^1) \in H, \mu \in \mathcal P_2(H) \ q \in \tilde{\Lambda},  w \in \mathbb{R}$, we have the setup of Subsection \ref{subsec:finite_particle-system}. Here, we have  $\mathcal{H}(x,\mu,p) :=  \inf_{q\in \tilde{\Lambda}} ( \langle f(x,\mu,q) , p\rangle + l(x,\mu,q))$ {and if \(
    (A_n)_{ij}=\sigma\sigma^{\top},  i,j=1,\dots,n
\),}
the HJB equation for the  particle system is
\begin{equation}
\begin{cases}
	\partial_t u_n + \frac12 \text{Tr}({A_n }D^2u_n)+ \frac{1}{n} \sum_{i=1}^n \left ( \langle Ax_i,nD_{x_i} u_n \rangle + \mathcal{H}(x_i,\mu_{\mathbf{x}},nD_{x_i}u_n) \right ) = 0, \quad(t,\mathbf{x})\in (0,T)\times H^n\\
	u_n(T,\mathbf{x}) = \frac{1}{n} \sum_{i=1}^n \mathcal{U}_T(x_i,\mu_{\mathbf{x}}), \quad \mathbf{x}\in H^n.
\end{cases}
\end{equation}

We check now that the assumptions of the previous sections are satisfied (we use the setup and notations introduced there) so that  the results of our theory can be applied. Indeed, by \cite[Section 3]{defeo_federico_swiech_2024}, the linear operator $B:=(A^{-1})^*A^{-1}$ satisfies the weak $B$-condition Assumption \ref{Assumption_weak_B_condition} with $c_0=0$ and it is compact, so it satisfies Assumption \ref{Assumption_B_compact}. Moreover, it was shown in  \cite[Sections 3 and 4]{defeo_federico_swiech_2024} there that there is a constant $C\geq 0$ such that
\begin{align}\label{eq:inequalities-wrt--1}
    &|x^0|\leq  |x|_{-1}, \quad   \left|\langle \Gamma^1,x^1\rangle_{L^2} \right| \leq C |x|_{-1}  &  \Gamma^1=\eta^1,\chi^1, \quad  \forall x=(x^0,x^1) \in H.
\end{align}
Next, for all $ \mu,\beta \in \mathcal P_2(H)$,  $r \in [1,2]$, we claim
\begin{align}\label{eq:inequalities-wrt-W-1}
    \left |\int_H z^0 (\mu-\beta)(\mathrm dz) \right| \leq d_{-1,r}(\mu,\beta), \quad \left |  \int_{H} \langle \chi^1, z^1 \rangle_{L^2} (\mu-\beta)(\mathrm d z)  \right | \leq C d_{-1,r}(\mu,\beta) .
\end{align}
Indeed, let $X=(X^0,X^1), Y=(Y^0,Y^1) \in E$ be such that $X_{\texttt{\#}} \mathcal{L}^1 = \mu, \; Y_{\texttt{\#}} \mathcal{L}^1 = \beta $. Then, using  \eqref{eq:inequalities-wrt--1}, we have
\begin{align}
   & \left |\int_H z^0 (\mu-\beta)(\mathrm dz) \right| =|\mathbf{E}[X^0-Y^0]|\leq \mathbf{E}[|X^0-Y^0|]\leq \mathbf{E}[|X-Y|_{-1}]\leq \mathbf{E}[|X-Y|_{-1}^r]^{1/r},
   \\
   &\left |\int_{H} \langle \chi^1, z^1 \rangle_{L^2} (\mu-\beta)(\mathrm dz)\right |=\left | \langle \chi^1,\mathbf{E}[X^1-Y^1] \rangle_{L^2} \right | \leq C | \mathbf{E}[X-Y]|_{-1}\leq C \mathbf{E}[|X-Y|_{-1}^r]^{1/r}
\end{align}
{and thus \eqref{eq:inequalities-wrt-W-1} follows from \eqref{definition_wasserstein_distance}.}

Next, notice that $f$ satisfies Assumption \ref{Assumption_f_sigma_lipschitz} (while $\sigma$ trivially satisfies it); indeed, the first inequality in (i) follows from \eqref{eq:inequalities-wrt--1}, \eqref{eq:inequalities-wrt-W-1}, while the second is trivial; the first inequality in (ii) follows by choosing $\beta=\delta_0$ in \eqref{eq:inequalities-wrt-W-1} and then using the equality $d^r_{-1,r}(\mu,\delta_0) = \mathcal{M}_{-1,r}(\mu)$, while the second there is trivial;
for (iii)  notice that 
$D_x \tilde f(x,\mu,q)y=((b^0+1)y^0+  \langle \eta^1, y^1 \rangle_{L^2},
   0)$, $D_X \tilde f (x,X,q)Y=(c^0\mathbf E [Y^0 ]+ \mathbf E [\langle \chi^1, Y^1 \rangle_{L^2}],
   0)$, $D_q \tilde f(x,\mu,q)a= (
   e^0a,
   0)$;
then, we have (iii) and so Assumption \ref{Assumption_f_sigma_lipschitz} is satisfied. With {completely} similar techniques, we prove that $l, \mathcal U_T$ satisfy Assumption \ref{Assumption_running_terminal_cost} (i), (ii), (iv). {If we add, as said above, that} $g^0 \in C^{1,1}([0,\infty);[0,\infty))$, $h^0  \in C^{1,1}(\mathbb R^2)$, $\mathcal U_T^0  \in C^{1,1}(\mathbb R^2)$, then Assumption \ref{Assumption_running_terminal_cost} (iii), (v) is satisfied; indeed, note that  $D_x \tilde l(x,X,q)z=-\partial_{y} h^0(x^0,\mathbf E[X^0])z^0, D_X \tilde l(x,X,q)Y=-\partial_{z} h^0(x^0,\mathbf E[X^0])\mathbf E [Y^0], D_q \tilde l(x,X,q)a=\partial_{q} g^0(q)a$ and the claim follows thanks to \eqref{eq:inequalities-wrt--1}, \eqref{eq:inequalities-wrt-W-1}.  
{
Next, Assumption
\ref{Assumption_running_cost_nu} is satisfied with $C_1=C_2=\nu =0$ because of the convexity  of $l^0$;
indeed, 
by \eqref{eq:inequalities-wrt--1}, \eqref{eq:inequalities-wrt-W-1} (with $r=2$ there), for all $x,y \in H, X,Y \in E, a,q \in \tilde{\Lambda}$, we have} 
{
\begin{equation}\label{eq:verifyHP4.6}
\begin{aligned}
& \lambda \tilde l(x,X ,a)+(1-\lambda) \tilde  l\left(y, Y,q\right)-\tilde  l\left(\lambda x+(1-\lambda) y, \lambda X+(1-\lambda) Y,\lambda a+(1-\lambda) q\right) \\
& =\lambda l^0\left(x^0,\mathbf{E} [X^0],  a\right)+(1-\lambda) l^0\left(y^0,\mathbf{E} [Y^0] ,q\right)\\
&\quad -l^0\left(\lambda x^0+(1-\lambda) y^0,\lambda \mathbf{E} [X^0]+(1-\lambda) \mathbf{E} [Y^0], \lambda a+(1-\lambda) q\right)\geq 0.
\end{aligned}
\end{equation}
On the other hand, if we add that for some $\nu>0$ the map $g^0(q)-\nu |q|^2$ is convex, then, arguing as in \eqref{eq:verifyHP4.6}, Assumption
\ref{Assumption_running_cost_nu} is satisfied with such $\nu$ and $C_1=C_2=0$.
}
Finally, by \eqref{eq:inequalities-wrt--1}, \eqref{eq:inequalities-wrt-W-1}, we  see that $f$ satisfies Assumption \ref{Assumption_Linear_State_Equation} (i); Assumption \ref{Assumption_Linear_State_Equation} (ii) is trivially satisfied, while (iii) follows from the convexity of $l^0, \mathcal U_T^0$. 

{
Given the above
we can apply (without the additional regularity and convexity assumptions) Theorem \ref{theorem:convergence} to obtain convergence of the value functions $u_n$ for the particle control problems to the function ${\mathcal V}$ whose lift $V$ is the value function of the lifted problem.
Adding the $C^{1,1}$ regularity of $g^0,h^0,\mathcal{U}_T^0$
we can apply the results of Subsection \ref{subsection:C11_regularity} to get $C^{1,1}$-regularity of $V$,  Theorem \ref{theorem:projection_C11} to prove that ${\mathcal V}=u_n$ on averages of point masses, as well as lifting and projection of optimal controls, i.e., Subsections \ref{sub:lifting_optimal}, \ref{sub:projection_optimal}.
}

What remains is to adapt the results of lifting and projection of optimal feedback controls of Subsections \ref{sub:lifting_optimal_feedbacks}, \ref{sub:projection_optimal_feedbacks}, as we are not exactly in the setting there since we do not have $\tilde{\Lambda} =\Lambda=H$.
In virtue of the results of Subsection \ref{subsection:C11_regularity},  we can assume that $V(t,\cdot)\in C^{1,1}(E_{-1})$ for every $t\in [0,T]$, and the semiconcavity and semiconvexity constants of $V(t,\cdot)$ are independent of $t\in [0,T]$.  
{Assume here for simplicity
that $g^0(q)=\frac 1 2 q^2$.} 
Define
    $$\mathbf{a}^{\ast}(\cdot) =  (a^{\ast}_1(\cdot),\dots,a^{\ast}_n(\cdot)),\quad a_i^{\ast}(s):=\gamma^\ast(D_{x_i}V_n(s,{\mathbf{x}^{\ast}(s)}))), $$ 
where $\gamma^{\ast}(p):=\argmin_{a \in [0,\infty)}[  n e^0  p^0 a+ \frac 1 2 a^2]=
    ne^0 (p^0)^- , $ for $p=(p^0,p^1) \in H$, with $z^-=\max(-z,0)$  denoting the negative part of $z \in \mathbb R$, and  $\mathbf{x}^{\ast}(s)=(x_1^{\ast}(s),\dots,x_n^{\ast}(s))$ is the unique mild solution of the closed loop system 
    \begin{equation}
\mathrm{d}x_i^*(s) = \left[A x_i^*(s)+ b x_i^*(s)+c \overline{\mathbf{x}}_n^*(s)+  \eta x_i^*(s)+\chi \overline{\mathbf{x}}_n^*(s) +e \gamma^\ast(D_{x_i}V_n(s,{\mathbf{x}^{\ast}(s)}))^0) \right] \mathrm{d}s + \sigma \, \mathrm{d}W(s), 
\end{equation}
with $x_i^*(t)=x_i$, $i \leq n$.
    Then $\gamma^*$  defines an optimal advertising policy as an optimal feedback control $\mathbf{a}^{\ast}(\cdot)$
for the finite particle system  problem. We remark that only the $x^0$-component of $D_{x_i}V_n(s,{\mathbf{x}^{\ast}(s)}) \in H$, for $i=1, \ldots n,$ is used in the optimal feedback map. Then, we can proceed in a similar way to Subsections \ref{sub:lifting_optimal_feedbacks}, \ref{sub:projection_optimal_feedbacks}, to prove lifting and projection of the optimal feedback control. Here, the optimal feedback control for the lifted control problem is given by $$a^{\ast}(s)(\omega)=e^0\left [(DV(s,X^{\ast}(s))(\omega))^0\right]^-, \quad \omega \in (0,1).$$ Again we remark that only the $x^0$-component of $DV(s,X^{\ast}(s))(\omega)$ is used in the optimal feedback map.
\subsection{Optimal investment with vintage capital for large companies (SPDEs)}Inspired by  mean-field optimal investment problems with vintage capital as introduced in  \cite{cosso_gozzi_kharroubi_pham_rosestolato}, in this subsection we introduce a particle system type control problem for optimal investment strategies with vintage capital under uncertainty for a large company. We  refer to the recent review \cite{fabbri_faggian_federico_gozzi_2025} {and the references therein} for  the classical  deterministic setting.

Consider a large, innovation-driven firm composed of $n$ interconnected subsidiaries. Each subsidiary $i=1, \ldots, n$ produces technologically advanced goods by investing in capital goods that embody both new (disruptive) and older (traditional) technologies. This is modeled through a vintage capital structure, where capital goods are indexed by their age $\theta \in[0, \Theta]$, where $\Theta>0$ represents the maximum age. We assume that  the amount of capital goods $y_i(s ,\theta)$ of age $\theta$
accumulated at time $s \in [t,T]$ evolves according to the controlled SPDE
\begin{equation}
\begin{cases}
&\frac{\partial y_i(s ,\theta)}{\partial s }+\frac{\partial y_i(s ,\theta)}{
\partial \theta}=-\delta y_i(s ,\theta)-\bar \delta \,\overline{\mathbf{y}}_n(s ,\theta))+u_i(s ,\theta)+\sigma  \frac{\partial }{\partial s}W(s)(\theta)
,\quad (s ,\theta)\in (t,T)
\times (0,\Theta),\\
& y_i(s ,0)=0,\quad s \in (t,T ), \quad
y_i(t,\theta)=x_i(\theta),\quad \theta \in \lbrack 0,\Theta],
\end{cases}
\end{equation}
where  $x_i\in L^{2}([0,\Theta])$ is the initial value; $\overline{\mathbf{y}}_n(s,\theta)=\frac 1 n \sum_{i=1}^n y_i(s,\theta)=\int_{\mathbb R} z \mu_{{\bf y}(s,\theta)}(\mathrm d z)$, with $\mathbf {y}(s,\theta)=(y_1(s,\theta),\dots,y_n(s,\theta))$, is the sample mean and $ \mu_{{\bf y}(s,\theta)}$ being the empirical measure;
$\delta>0$ is a capital depreciation factor, $\bar \delta 
>0$ is a capital interaction depreciation factor modeling how capital becomes obsolete faster when average technological standards improve across subsidiaries; the control process $u_i(s,\theta)$ with values in {$\tilde{\Lambda}=L^2([0,\Theta];[0,\infty))\subset\Lambda=L^2([0,\Theta]$} is  the rate of investment at time $s $ in
capital goods of age $\theta$ undertaken by {the subsidiary $i$}; $W$ is a cylindrical Wiener process with values in $L^{2}([0,\Theta])$ acting as a common noise with intensity $\sigma \in L_2(L^{2}([0,\Theta]))$ {(while the intensity of the noise of each  subsidiary is assumed to be negligible with respect to the one of the common noise, {as it is done, e.g., in the context of Mean Field Games, in \cite{gueant_lasry_lions_2011,carmona_2021,souganidis_zariphopoulou_2024,bertucci_bertucci_lasry_lions_2024}})}.

We now introduce the objective functional.  The rate of investment costs at time $s$ undertaken by {the subsidiary $i$ is} $\int_{0}^{\Theta} c(\theta,u_i({s ,\theta}))\mathrm{d} \theta$ for a function  $ c\colon [0,\Theta] \times  [0,\infty) \to [0,\infty)$, such that there exist $C_1,C_2,C_3>0$, such that  $-C_1+C_2|q|^2 \leq c(\theta,q) \leq C_1+C_3|q|^2$;  the output rate at time $s$ is $\int_0^\Theta \eta(\theta) y_i(s,\theta)\mathrm{d}\theta$ for a suitable $\eta \colon [0,\Theta] \to [0,\infty)$.
The goal of the company is then to minimize an aggregate cost functional of the form
\begin{equation}
	\mathbb{E}\left [ \int_t^T\frac{1}{n}\sum_{i=1}^n\left\{\int_{0}^{\Theta} c(\theta,u_i({s ,\theta}))\mathrm{d} \theta  -h^0\left(\int_0^\Theta \eta(\theta) y_i(s,\theta)\mathrm{d}\theta\right) \right\} \mathrm{d}s+ \frac{1}{n} \sum_{i=1}^n {\mathcal U}_T^0\left (\int_0^\Theta \eta(\theta) y_i(T,\theta)\mathrm{d}\theta \right ) \right ],
\end{equation}
where  $h^0  :\mathbb R \to \mathbb R$ is a concave Lipschitz revenue function; $\mathcal U_T^0 :\mathbb R \to \mathbb R$ is a convex Lipschitz function  (to be interpreted as the negative of a concave terminal revenue function). {As in the previous subsection  more regularity assumptions will be added to apply the results of Sections \ref{sec:regularity}-\ref{sec:projection}-\ref{sec:lifting_proj}, see below.}
 
 The system is rewritten in the setup of Subsection \ref{subsec:finite_particle-system} over $H^n$ with  $H:=L^2([0,\Theta])$, where $ A x:= -\frac{\mathrm{d}}{\mathrm{d} \theta} x-\delta x,$ $ D(A):= \left\{ x \in W^{1,2}([0,\Theta]), \ x(0)=0\right\}$ is a maximally dissipative operator (so Assumption \ref{Assumption_A_maximal_dissipative} holds),
 $ f \colon  \mathcal P(H) \times \tilde{\Lambda} \to H$,  $f(\mu,q) :=-\bar \delta \int_{H} z \mu(\mathrm d z)+q$, $\sigma\in L_2(H)$, 
 $l \colon  H \times \mathcal P(H) \times \tilde{\Lambda} \to \mathbb R$, $l\left (x,q\right ):=\int_{0}^{\Theta} c(\theta,q({\theta}))\mathrm{d} \theta-h^0\left( \int_0^\Theta \eta(\theta) x(\theta)\mathrm{d}\theta\right)$, 
 $\mathcal U_T \colon H \times \mathcal P(H)  \to \mathbb R$, $ \mathcal U_T(x):=\mathcal U_T^0\left (\int_0^\Theta \eta(\theta) x(\theta)\mathrm{d}\theta\right )$,
for all $x \in H, \mu \in \mathcal P_2(H) \ q \in \tilde{\Lambda}.$ Here, we have  $\mathcal{H}(x,\mu,p) :=  \inf_{q\in \tilde{\Lambda}} ( \langle f(x,\mu,q) , p\rangle + l(x,q))$ {and if \((A_n)_{ij}=\sigma\sigma^*,  i,j=1,\dots,n
\),}
the HJB equation for the  particle system is
\begin{equation}
\begin{cases}
	\partial_t u_n + \frac12 \text{Tr}({A_n D^2u_n)}+ \frac{1}{n} \sum_{i=1}^n \left ( \langle Ax_i,nD_{x_i} u_n \rangle + \mathcal{H}(x_i,\mu_{\mathbf{x}},nD_{x_i}u_n) \right ) = 0, \quad(t,\mathbf{x})\in (0,T)\times H^n\\
	u_n(T,\mathbf{x}) = \frac{1}{n} \sum_{i=1}^n \mathcal{U}_T(x_i,\mu_{\mathbf{x}}), \quad \mathbf{x}\in H^n.
\end{cases}
\end{equation}

We check now that the assumptions of the previous sections are satisfied (we use the setup and notations introduced there) so that  the results of our theory can be applied. Indeed, similarly to \cite{defeo_federico_swiech_2024,fabbri_2008}, the  operator $B:=(A^{-1})^*A^{-1}$ satisfies the weak $B$-condition Assumption \ref{Assumption_weak_B_condition} with $c_0=0$ and it is compact, so it satisfies Assumption \ref{Assumption_B_compact}.  Moreover, notice that the adjoint operator of $A$ is $ A^* x:= \frac{\mathrm{d}}{\mathrm{d} \theta} x-\delta x,$ $ D(A^*):= \left\{ x \in W^{1,2}([0,\Theta]), \ x(\Theta)=0\right\}$ and assume from now on that $\eta \in D(A^*)$; then
\begin{align}\label{eq:inequalities-wrt--1_vintage}
\left|\int_0^\Theta \eta(\theta) x(\theta)\mathrm{d}\theta\right|=\left|\langle \eta,x\rangle_{L^2} \right|=\left|\langle A^*\eta,A^{-1}x\rangle_{L^2} \right| \leq C |x|_{-1} \quad  \forall x\in H,
\end{align}
where $C=|A^*\eta|$. {The inequality
\begin{align}\label{eq:inequalities-wrt-W-1_vintage}
    \left |\int_H z (\mu-\beta)(\mathrm dz) \right|_{-1} \leq d_{-1,r}(\mu,\beta)
\end{align}
is obtained in the same way as \eqref{eq:inequalities-wrt-W-1} in Section \ref{subsec1}.}

 Next, notice that $f$ satisfies Assumption \ref{Assumption_f_sigma_lipschitz} (while $\sigma$ trivially satisfies it); indeed,  the first inequality in (i) follows from  \eqref{eq:inequalities-wrt-W-1_vintage}, while the second is trivial; the first inequality in (ii) follows by choosing $\beta=\delta_0$ in \eqref{eq:inequalities-wrt-W-1_vintage} and then using the equality $d^r_{-1,r}(\mu,\delta_0) = \mathcal{M}_{-1,r}(\mu)$, while the second there is trivial;
for (iii)  notice that $D_X \tilde f (X,q)Y=-\bar \delta {\mathbf E} [Y ]$, $D_q \tilde f(x,\mu,q)a=a$;
then, we have (iii) and so Assumption \ref{Assumption_f_sigma_lipschitz} is satisfied. Next, by straightforward computations we have that $l, \mathcal U_T$ satisfy Assumption \ref{Assumption_running_terminal_cost} (i), (ii),  (iv), if we use \eqref{eq:inequalities-wrt--1_vintage}. If  $c(\theta,\cdot) \in C^{1,1}(\tilde{\Lambda};[0,\infty))$ with Lipschitz constant of $\partial_q c(\theta,q)$ independent of $\theta$, $h^0  \in C^{1,1}(\mathbb R)$, $\mathcal U_T^0  \in C^{1,1}(\mathbb R)$, then Assumption \ref{Assumption_running_terminal_cost} (iii), (v) is satisfied; indeed, note that $D_x \tilde l(x,q)z=-\partial_{y} h^0(\int_0^\Theta \eta(\theta) x(\theta)\mathrm{d}\theta) \int_0^\Theta \eta(\theta) z(\theta)\mathrm{d}\theta,$ $ D_q \tilde l(x,q)a=\int_{0}^{\Theta}\partial_q  c(\theta,q({\theta}))a({\theta})\mathrm{d} \theta$ and the claim follows thanks to the Lipschitzianity of $\partial_{y} h^0, \partial_{q} g^0,  \partial_u  c(\theta,\cdot)$, and to \eqref{eq:inequalities-wrt--1_vintage}.   Next, Assumption
\ref{Assumption_running_cost_nu} is satisfied for $C_1=C_2=\nu=0$ by the above concavity assumptions, while it is satisfied with $C_1=C_2=0$ and $\nu>0$ if there exist {$\nu> 0$} such that the map $q \mapsto c(\theta,q)-\nu|q|^2$ is convex ($\nu$ independent of $\theta$). 

As in \eqref{eq:inequalities-wrt-W-1_vintage}, we  see that $f$ satisfies Assumption \ref{Assumption_Linear_State_Equation} (i); Assumption \ref{Assumption_Linear_State_Equation} (ii) is trivially satisfied, while (iii) follows from the convexity of $-h^0, \mathcal U_T^0$.    

Therefore, as for the problem in Subsection \ref{subsec1}, we can apply the results of our theory.



\appendix

\section{Viscosity solutions of PDEs in Hilbert spaces} 
\label{app:viscosity_hilbert}

We recall here the definition of a viscosity solution for a terminal value degenerate parabolic PDE on a Hilbert space with an unbounded operator from \cite[Section 3.3]{fabbri_gozzi_swiech_2017}.

Throughout this section $(Y,\langle\cdot,\cdot\rangle)$ is a real separable Hilbert space {and $B \in S(Y)$ is a strictly positive operator.}  
\subsection{$B$-continuity}
\begin{definition}[$B$-continuity]  Let  $u: [0,T] \times Y \to \mathbb{R}$. We say that $u$ is $B$-upper semicontinuous (respectively, $B$-lower semicontinuous) if, for any sequences $\left(t_n\right)\subset [0,T]$ and $\left(x_n\right)\subset Y$ such that $t_n \rightarrow t \in [0,T],$ $ {x_n \rightharpoonup x }\in Y$, $B x_n \rightarrow B x$ as $n \rightarrow \infty$, we have
$
\limsup _{n \rightarrow \infty} u\left(t_n, x_n\right) \leq u(t, x)$(respectively,  $\liminf _{n \rightarrow \infty} u\left(t_n, x_n\right) \geq u(t, x)$).

We say that $u$ is $B$-continuous if it is both $B$-upper semicontinuous and $B$-lower semicontinuous.
\end{definition}

\subsection{Viscosity solutions.}Consider the following terminal value PDE in the Hilbert space $Y$
\begin{equation}\label{eq:PDE_app}
    \left\{\begin{array}{l}\partial_t u+\langle A x, D u\rangle+\mathcal F\left(x, D u, D^2 u\right)=0, \quad (t,x)\in(0, T) \times Y\\
    u(T, x)=g(x), \quad  x \in Y,\end{array}\right.
\end{equation}
where $A:\mathcal D(A)\subset Y\to Y$ is a linear densely defined maximal dissipative operator,   $\mathcal F: Y  \times Y \times S(Y) \rightarrow \mathbb{R}, g: Y \rightarrow \mathbb{R}$ are continuous. {We also assume that
$\mathcal F\left(x, p, X\right)\leq \mathcal F\left(x, p, Z\right)$
for all $x,p\in Y$ and $X,Z\in S(Y), X\leq Z$.} {Moreover, suppose that $A^*B\in L(Y)$.}
\begin{definition}\label{def:test_functions_hilbert}
A function $\psi \colon (0,T)\times Y\to \mathbb R$ is a test function if $\psi=\varphi+h(t,|x|)$, where $\varphi \in C^{1,2}((0, T) \times Y)$ is locally bounded and is such that $\varphi$ is $B$-lower semicontinuous,  $\partial_t \varphi, A^* D \varphi, D \varphi, D^2 \varphi$ are uniformly continuous on $(0, T) \times Y$; $h \in C^{1,2}((0, T) \times \mathbb{R})$ is such that for every $t \in(0, T), h(t, \cdot)$ is even and $h(t, \cdot)$ is non-decreasing on $[0,+\infty)$.
\end{definition}
\begin{definition}\label{def:viscosity_solution_hilbert}
A locally bounded $B$-upper semicontinuous function $u:(0, T] \times Y$ is a $B$-continuous viscosity subsolution of \eqref{eq:PDE_app} if $u(T, y) \leq g(y)$ for all $y \in Y$ and whenever $u-\psi$ has a local maximum at a point $(t, x) \in(0, T) \times Y$ for a test function $\psi(s, y)=\varphi(s, y)+h(s,|y|)$ then
$$
\partial_t \psi(t, x)+\left\langle x, A^* D \varphi(t, x)\right\rangle+\mathcal F\left(x, D \psi(t, x), D^2 \psi(t, x)\right) \geq 0.
$$
A locally bounded $B$-lower semicontinuous function $u$ on $(0, T] \times Y$ is a $ B$-continuous viscosity supersolution of \eqref{eq:PDE_app} if $u(T, y) \geq g(y)$ for $y \in Y$ and whenever $u+\psi$ has a local minimum at a point $(t, x) \in(0, T) \times Y$ for a test function $\psi(s, y)=\varphi(s, y)+h(s,|y|)$ then
$$
-\partial_t \psi(t, x)-\left\langle x, A^* D \varphi(t, x)\right\rangle+\mathcal F\left(x,-D \psi(t, x),-D^2 \psi(t, x)\right) \leq 0.
$$
A $B$-continuous viscosity solution of \eqref{eq:PDE_app} is a function which is both a $B$-continuous viscosity subsolution and a $ B$-continuous viscosity supersolution of \eqref{eq:PDE_app}.
\end{definition}

{We remind the reader that here a function is called locally bounded if it is bounded on bounded sets. Moreover, without loss of generality the maxima and minima in Definition \ref{def:viscosity_solution_hilbert} can be assumed to be global and strict, see \cite[Lemma 3.37]{fabbri_gozzi_swiech_2017}. A maximum/minimum of a function $f$ at $x$ is called strict if whenever $f(x_n)\to f(x)$ then $x_n\to x$.}

\medskip

\small{\paragraph{\textbf{Acknowledgments.}} Filippo de Feo acknowledges funding by the Deutsche Forschungsgemeinschaft (DFG, German Research Foundation) – CRC/TRR 388 "Rough Analysis, Stochastic Dynamics and Related Fields" – Project ID 516748464, by INdAM (Instituto Nazionale di Alta Matematica F. Severi) - GNAMPA (Gruppo Nazionale per l'Analisi Matematica, la Probabilità e le loro Applicazioni). Filippo de Feo and Fasuto Gozzi acknowledge funding  by the Italian Ministry of University and Research (MUR) in the framework of PRIN project 20223PNJ8K (Impact of the Human Activities on the Environment and Economic Decision Making in a Heterogeneous Setting: Mathematical Models and Policy Implications).}

\bibliographystyle{amsplain}

\end{document}